\documentclass[reqno]{amsart}
\usepackage{epsfig,latexsym,amsfonts,amssymb,amsmath,amscd}
\usepackage[foot]{amsaddr}
\usepackage{amsthm}
\usepackage{tikz-cd}
\usepackage[mathscr]{eucal}
\usepackage{mathrsfs}
\usepackage{mathbbol}
\usepackage{array}
\usetikzlibrary{matrix,arrows,decorations.pathmorphing}
\usepackage{oldgerm,units,color}
\usepackage{hyperref}
\usepackage{pst-node}
\usepackage{graphicx}

\usepackage{eepic, epic}

\usepackage{ifthen}
\DeclareMathOperator{\coker}{coker}

\newcommand{\too}[1]{\overset{#1}\to}

\def\tTN{{\tT}_{\mathcal N}}
\def\tropa{  {\operatorname{trop}}\, }
\def\Null{{\operatorname{Null}}}
\def\precpr{\preceq_{\operatorname{Null}}}
\def\strop{  {\operatorname{strop}}\, }
\def\vsupp{\nu$-$\operatorname{supp}}

\def\congg{{\operatorname{cong}}}
\def\Tker{\tT\operatorname{-}\ker}
\def\TkerC{\tT_{\operatorname{cong}}\operatorname{-}\ker}

\def\val{\operatorname{val}}
\def\Gz{ {\Gamma}_{\zero}}

\def\module0{module$^\dagger$}
\def\ssemiring0{$s$-semiring$^\dagger$}

\usepackage{tikz}
\usepackage{epsfig,latexsym,amsfonts,amssymb,amsmath,amscd,graphics,epic}
\usepackage{amsfonts,amssymb,amsmath,amscd,amsthm}
\usepackage[mathscr]{eucal}
\usepackage{mathrsfs}

\usepackage{mathbbol}

\usepackage{oldgerm,units}
\usepackage{wrapfig,epsfig}

\usepackage{ifthen}
\usepackage{mathbbol}

\usepackage{amsthm}
\usepackage[mathscr]{eucal}
\usepackage{mathrsfs}
\usepackage{mathbbol}
\usepackage{oldgerm,units}
\usepackage{wrapfig}

\newtheorem{theorem}{Theorem}[section]

\newtheorem{example}[theorem]{Example}

\newcommand{\Real}{\mathbb R}

\newcommand{\Net}{\mathbb N}

\newcommand{\one}{\mathbb{1}}
\newcommand{\zero}{\mathbb{0}}

\newcommand{\Trop}{\mathbb T}

\newcommand{\trop}[1]{\mathcal{#1}}

\newcommand{\tG}{\trop{G}}

\newcommand{\tT}{\trop{T}}

\newcommand{\Spec}{Spec}

\newcommand{\Hom}{Hom}

    \ifx\proof\undefined
    \newenvironment{proof}{
    \smallskip
    \noindent\emph{Proof.}}{\hfill\(\Box\)
    \bigskip
    } \fi

\newcommand{\ifdef}[3]{\ifthenelse{\equal{#1}{true}}{#2}{#3}}

\definecolor{lgray}{gray}{0.90}

\def\Spec{\operatorname{Spec}}

\def\ctw{\cdot_{\operatorname{tw}}}

\def\vep{\varepsilon}

\def\GM{group module}

\def\({\left(}
\def\){\right)}
\def\N{{\mathbb N}}
\def\Z{{\mathbb Z}}
\def\Q{{\mathbb Q}}

\def\pipe{{\underset{{\ \, }}{\mid}}}

\def\vsemifield0{$\nu$-semifield$^\dagger$}
\def\vsemiring0{$\nu$-semiring$^\dagger$}

\def\pipe1{{\underset{{1}}{\mid}}}
\def\lmod1{\mathrel  \pipe1  \joinrel \joinrel =}

\def\CFunFF1{\operatorname{CFun} (F,F)}

\def\semiring0{semiring$^{\dagger}$}
\def\Semiring0{Semiring$^{\dagger}$}
\def\Semirings0{Semirings$^{\dagger}$}
\def\semidomain0{semidomain$^{\dagger}$}
\def\semifield0{semifield$^{\dagger}$}
\def\semifields0{semifields$^{\dagger}$}
\def\vsemifields0{$\nu$-semifields$^{\dagger}$}
\def\domain0{domain$^{\dagger}$}
\def\predomain0{pre-domain$^{\dagger}$}
\def\predomains0{pre-domains$^{\dagger}$}
\def\domains0{domains$^{\dagger}$}
\def\vdomains0{$\nu$-domains$^{\dagger}$}

\def\Fun{\operatorname{Fun}}

\def\domains0{domains$^\dagger$}

\def\ker{\operatorname{ker}}

\newcommand{\etype}[1]{\renewcommand{\labelenumi}{(#1{enumi})}}
\def\eroman{\etype{\roman}}

\def\pipe{{\underset{{\tG}}{\mid}}}

\def\lmod{\mathrel  \pipe \joinrel \joinrel =}
\def\pipe{{\underset{{\tG}}{\mid}}}

\def\Ann{{\operatorname{Ann}}}

\def\a{\alpha}
\newtheorem{thm}[theorem]{Theorem}

\newtheorem*{thm*}{Theorem}

\newtheorem{cor}[theorem]{Corollary}

\def\Hom{\operatorname{Hom}}
\def\Mor{\operatorname{Mor}}
\newtheorem{lem}[theorem]{Lemma}
\newtheorem{rem}[theorem]{Remark}
\newtheorem{prop*}{Proposition}

\newtheorem{prop}[theorem]{Proposition}
\newtheorem{defn}[theorem]{Definition}
\newtheorem*{examp*}{Example}
\newtheorem*{examples*}{Examples}
\newtheorem*{remark*}{Remark}
\newtheorem*{defn*}{Definition}
\newtheorem*{nota}{Notation}

\def\la{\lambda}

\def\tT{\mathcal T}
\def\Fun{\operatorname{Fun}}
\def\tTz{\tT_\zero}
\def\tTMz{\tT_{\mathcal M,\zero}}

\def\tTNz{\tT_{\mathcal N,\zero}}

\numberwithin{equation}{section}

\def\M0{M_{\zero}}

\def\supp{\operatorname{supp}}

\def\tGz{\mathcal G_\zero}

\def\AA{\tilde A}
\def\BA{\tilde B}
\def\CA{\tilde C}
\def\DA{\tilde D}

\def\PS{P}

\def\Cong{\Phi}
\def\CongT{\Phi|_\tT}
\def\Diag{{\operatorname{Diag}}}

\def\semirings0{semirings$^\dagger$}

\newtheorem*{nothma}{\textbf{Proposition A}}
\newtheorem*{nothmb}{\textbf{Theorem B}}
\newtheorem*{nothmc}{\textbf{Corollary C}}
\newtheorem*{nothmd}{\textbf{Theorem D}}
\newtheorem*{nothme}{\textbf{Theorem E}}
\newtheorem*{nothmf}{\textbf{Theorem F}}
\newtheorem*{nothmg}{\textbf{Proposition G}}

\newcommand{\nPS}[1]{\PS_{(!#1)}}
\newcommand{\nPSo}[1]{\nPS{\one}}

\newcommand{\absl}[1]{|{#1}|}
\newcommand{\adj}[1]{\operatorname{adj}({#1})}

\begin{document}


\title[Categories with negation]
{Categories with negation}


\author[J.~Jun]{Jaiung~Jun~$^1$}
\dedicatory{In honor of our friend and colleague, S.K. Jain.}
\thanks{The authors would like to thank Oliver Lorscheid for helpful comments concerning the first draft of the paper.}
\address{$^1$Department of Mathematics, University of Iowa, Iowa City, IA 52242, USA} \email{$^1$jujun0915@gmail.com}

\author[L.~Rowen]{Louis Rowen~$^2$}
\address{$^2$Department of Mathematics, Bar-Ilan University, Ramat-Gan 52900,
Israel} \email{$^2$rowen@math.biu.ac.il}

\subjclass[2010]{Primary     08A05, 08A30, 14T05, 16Y60,  20N20.
  Secondary   06F05, 08A72, 12K10, 13C60.
}



\keywords{Category,  pseudo-triple, matroid, pseudo-system,  triple,
system, semiring, semifield, congruence, module, negation,
surpassing relation, symmetrization, congruence, tropical algebra,
supertropical algebra, bipotent, meta-tangible, symmetrized,
   hypergroup, hyperfield, polynomial, prime, dimension, algebraic, tensor product.}



\begin{abstract}

We continue the theory of $\tT$-systems from the work of the second
author,  describing both  ground systems and module systems over a
ground system (paralleling the theory of modules over an algebra).
Prime ground systems are introduced as a way of developing geometry.
One basic result is that the polynomial system over a prime system
is prime.

For module systems, special attention also is paid to tensor
products and $\Hom$. Abelian categories are replaced by
``semi-abelian'' categories (where $\Hom(A,B)$ is not a group) with
a negation morphism. The theory, summarized categorically at the
end,
 encapsulates general algebraic structures lacking
 negation but possessing a map resembling negation,
such as tropical algebras, hyperfields and fuzzy rings. We see
explicitly how it encompasses tropical algebraic theory and
hyperfields.
\end{abstract}
\dedicatory{In honor of our friend and colleague, S.K. Jain.}
\maketitle


  {\small \tableofcontents}



\section{Introduction}
This paper, based on \cite{JuR1}, is the continuation of a project
started in \cite{Row0,Row16} as summarized~in \cite{Row17}, in which
a generalization of classical algebraic theory was presented to
provide
 applications in related algebraic theories.  It is an attempt to understand why
basic algebraic theorems are mirrored in supertropical algebra,
spurred by the realization that some of the same results were
obtained in parallel research on hypergroups and hyperfields
\cite{Bak,GJL,Ju,Ju1, Ju2,Vi} and fuzzy rings \cite{Dr,DW,GJL},
which lack negatives although each has an operation resembling
negation \footnote{In a hypergroup, for each element $a$ there is a
unique element called $-a$, such that $0 \in a \boxplus (-a)$.}
\footnote{ A fuzzy ring $A$ has an element~$\varepsilon$ such
$1+\varepsilon$ is in a distinguished ideal  $A_0$, and we define
$(-)a = \varepsilon a$.}.

The underlying idea is to take a set $\tT$ that we want to study. In
the situation considered here, $\tT$ also has a partial additive
algebraic structure which is not defined on all of $\tT$; this is
resolved by having $\tT$ act on a  set $\mathcal A$ with a fuller
algebraic structure.

Lorscheid \cite{Lor1,Lor2} developed this idea when $\tT$ is a
monoid which is a subset of a semiring $\mathcal A$. Since semirings
may lack negatives,  we
  introduce a formal \textbf{negation map}  $(-)$ in Definition~\ref{negmap}, resembling negation, often requiring
   that $\tT $ generates $\mathcal A$ additively.  In tropical algebra,
 $(-)$ can be taken to be the identity map.  Or it can be supplied via ``symmetrization''
 (\S\ref{symm}),
motivated by \cite{AGG1,Ga,GaP,Pl}. Together with the negation map,
$\mathcal A$ and $\tT $ comprise a pseudo-\textbf{triple} $(\mathcal
A, \tT, (-))$. This is rounded out to a \textbf{pseudo-system}
$(\mathcal A, \tT, (-),\preceq)$ with a \textbf{surpassing relation}
$\preceq$, often a partial order (PO),
  replacing equality in the algebraic theory. Ironically,
 equality in classical mathematics is the only situation in
which $\preceq$ is an equivalence. We set forth a systemic
foundation for affine geometry (based on prime systems) and
representation theory, as well as to lay out the groundwork for
further research, cf.~\cite{AGR} for linear algebra, \cite{GatR} for
exterior algebra, \cite{JMR} for projective modules, and \cite{JMR1}
for homology, and other work in progress.

Familiar concepts from classical algebra were applied in
\cite{Row16} to produce new triples and systems, e.g., direct powers
{\cite[Definition 2.6]{Row16}}, matrices \cite[\S 6.5]{Row16},
involutions \cite[\S6.6]{Row16}, polynomials \cite[\S 6.7]{Row16},
localization \cite[\S 6.8]{Row16}, and tensor products {\cite[\S
8.6]{Row16}}. Recently  tracts, a special case of pseudo-systems,
were introduced in  \cite{BB1}  in order to investigate matroids. At
the end of this paper we also view systems categorically,   to make
them formally applicable to varied situations.

 Since the main motivation comes from tropical and
supertropical algebra, in Appendix A (\S\ref{comp}), we coordinate
the systemic theory with the main approaches to tropical
mathematics, demonstrating the parallels of some tropical notions
such as bend congruences (introduced by J.~Giansiracusa and
N.~Giansiracusa in \cite{GG}, and tropical ideals (introduced by
Maclagan and Rincon in \cite{MacR}). For example,
  \cite[Proposition~7.5]{Row17} says that the bend relation implies
  $\circ$-equivalence, cf.~Definition~\ref{sys1}.

 Category   theory
involving explicit algebraic structure  can be described in
universal algebra in terms of operations and identities,  reviewed
briefly in \S\ref{unalg}. But
 the ``surpassing relation''~$\preceq$  also is required.

  A crucial issue here is the ``correct'' definition of morphism.
    One's initial instinct would be to take ``homomorphisms,'' preserving
    all the structure from universal algebra.
However, this does not  tie into hyperfields, for which  a more
    encompassing definition pertains in \cite{Ju,Ju2}.
  Accordingly, we   define
    the more general $\preceq$-\textbf{morphism} in
   Definition~\ref{mor}, satisfying $f(a+a') \preceq f(a) + f(a')$.
 Another key point
 is that in the theory of semirings and their modules, homomorphisms
are described in terms of \textbf{congruences}. The trivial
congruences contain the diagonal,   not just zero.
%

$\Mor(A,B)_\preceq$ denotes the semigroup of $\preceq$-morphisms
from $A$ to $B,$  and has the sub-semigroup $\Hom(A,B)$ of
homomorphisms. Since $(\Mor(A,B)_\preceq,+)$ and $(\Hom(A,B),+)$
  no longer
 are   groups,  one needs to weaken the notion of additive
and abelian categories, respectively in \S\ref{semiadd} to
\textbf{semi-additive categories} and \textbf{semi-abelian
categories}, \cite[\S1.2.7]{grandis2013homological}.

The tensor product and its abstraction to monoidal categories,
 an important aspect of algebra,  is exposed in \cite{EGNO} for
monoidal abelian categories. But, to our dismay,  the functoriality
of the tensor product  runs into stumbling blocks  because of the
asymmetry involved in
  $\preceq$. So we have a give and play between $\preceq$-morphisms
  and  homomorphisms, which is treated in \S\ref{tenpro}.

\subsection{Objectives}\label{out1}$ $

Our objectives in this paper are as follows:


\begin{enumerate}
\item  Lay out the notions in \S\ref{modsys0} of $\tT$-module, negation map, ``ground'' triples and $\tT$-systems, which should parallel the classical structure theory of algebras.
In the process, we consider convolutions and the ensuing
construction of polynomials, Laurent series, etc.

\item Study affine geometry in in \S\ref{Pri} in terms of polynomials, with  special attention to the theory of prime ground  triples, to lay out the
   groundwork for the spectrum of prime congruences.

 \item Elevate congruences to their proper role in the theory, in
     \S\ref{homim} and
   \S\ref{modsys},
   since the  process of modding out ideals suffers from the lack of negation.
Investigate which classical module-theoretic
    concepts (such as  $\Hom$ and direct sums) have analogs for module triples over ground triples, viewed in terms of their
    congruences.  Special attention is given to the  tensor product of
    triples (Definition~\ref{bil3}).
In general $\preceq$-morphisms do not permit us to build tensor
    categories, as seen in Example~\ref{nonmonoidal}, but we do
    have the usual theory using homomorphisms, in view of \cite{KN}.

    \item  Express these notions in categorical terms.   This should
   parallel the theory of modules over semirings, which has been developed in the last few years by Katsov~\cite{Ka1,Ka2,Ka3},
   Katsov and Nam~\cite{KN},
   Patchkoria~\cite{Pat}, Macpherson~\cite{AM},
   Takahashi~\cite{Tak}.

%

%

    \item Provide the functorial context for the main
categories of this paper, as indicated in the diagram given in
\S\ref{func1}.

\item  Relate this theory in Appendix A to other approaches in
 tropical algebra.

 \item Define negation morphisms and negation functors, together with a surpassing relation, in the context of $N$-categories,
in Appendix B. In the process, we generalize abelian categories to
semi-abelian categories with  negation, and lay out the role of
functor categories.

\end{enumerate}

\subsection{Main results}\label{out0}$ $

Our results in geometry require prime congruences
(Definition~\ref{Pricon}).

 \begin{nothma}[{\bf Proposition \ref{prime2}}]
 For every $\tT$-congruence $\Cong$ on a commutative $\tT$-semiring system,  $\sqrt{\Cong}$
  is an intersection  of prime $\tT$-congruences.
 \end{nothma}

   \medskip

 \begin{nothmb}[{\bf Theorem \ref{polyroot}}] Over a commutative prime triple $\mathcal A$,
any nonzero polynomial $f \in \tT[\la]$ of degree $n$ cannot have
$n+1$ distinct $\circ$-roots in $\tT$. \end{nothmb}

 \medskip

  \begin{nothmc}[{\bf Corollary~\ref{poly11}}]
 If $(\mathcal A,\tT,(-))$ is a  prime commutative
 triple with $\tT$ infinite, then so is the polynomial triple $(\mathcal A[\la],\tT_{\mathcal A[\la]},(-)).$
\end{nothmc}

  \medskip
This corollary   plays a key role in the geometry of systems, in
terms of the prime spectrum.

 \begin{nothmd} [{\bf Theorem \ref{thma}, Artin-Tate lemma, $\preceq$-version}]
    Suppose
$\mathcal A'=\mathcal A[a_1,a_2,\dots ,a_n]$ is a  $\preceq$-affine
system over~$\mathcal A$, and $\mathcal K$ a subsystem of $\mathcal
A',$ with $\mathcal A'$ having a $\preceq$-base $v_1 =1, \ \dots,
v_d$ of $\mathcal A'$ over $K$. Then $\mathcal K$ is
$\preceq$-affine over $\mathcal A$.
\end{nothmd}

 \begin{nothme}[{\bf Theorem \ref{thma1}}] If $(\mathcal A, \tT, (-))$ is a semiring-group system
  and $(\mathcal A ' = \mathcal A[a_1, \dots, a_m],
\tT',(-))$ is a $\preceq$-affine semiring-group system, in which
$(f,g)(a_i, \zero)$ is invertible for every symmetrically
functionally tangible pair $(f,g)$ of polynomials, then $a_1, \dots,
a_m$ are symmetrically algebraic over $ \tT$,  and
$\widehat{\mathcal A '}$ has a symmetric base over
$\tT$.\end{nothme}

 The next result is rather easy, but its
statement is nice.

 \begin{nothmf} [{\bf Theorem \ref{polyzyz}}] Both the polynomial $\mathcal A[\la_1, \dots, \la_n]$ and Laurent
polynomial systems $\mathcal A[[\la_1, \dots, \la_n]]$ in $n$
commuting indeterminates over a $\tT$-semiring-group system have
dimension $n$.
\end{nothmf}

 Tensor products are described in detail in \S\ref{tenpro1}. From
 the categorical point of view:

 \begin{nothmg}  [{\bf Proposition \ref{kercok1}}]
The category of module systems $\mathcal A $-module triples  is a
monoidal semi-abelian category,  with respect to negated
$\tT$-tensor products.
\end{nothmg}

\section{Background}\label{modsys0}$ $

Here is a review of what is needed to understand the main results.

\subsection{Basic structures}\label{bn}$ $

A \textbf{\semiring0}  $(\mathcal A, +, \cdot, 1)$ is an additive
abelian semigroup $(\mathcal A, +)$ and a multiplicative monoid
$(\mathcal A, \cdot, \one)$ satisfying the usual distributive laws.
A \textbf{semiring} is a \semiring0 which contains a $\zero$
element.  A \textbf{semidomain} is a
 semiring $\mathcal A$ such that $\mathcal A \setminus
\{\zero\}$ is closed under multiplication, i.e., $ab=\zero$ only
when $a=\zero$ or $b=\zero$ for all $a, b \in \mathcal A$. We do not
assume commutativity of multiplication, since we want to consider
matrices and other noncommutative structures more formally. But   we
often specialize to the commutative situation (for example in
considering prime spectra) when the proofs are simpler. We stay
mainly with the associative case, in contrast to \cite{Row16}.
Distributivity is a stickier issue, since we do not want to
forego the applications to hypergroups.

As customary, $\Net$ denotes the nonnegative integers, $\Q$  the
rational numbers, and $\Real$   the real numbers, all ordered
monoids under addition. Let us give a brief review of the theory of
triples, from~\cite{Row16}.


\subsubsection{$\tT$-modules}\label{out}$ $

 First we put $\tT$ into the limelight, starting
with the following fundamental basic structure.

\begin{defn}\label{modu2}   A (left) $\tT$-\textbf{\module0} over a set $ \tT$
 is an additive semigroup $( \mathcal A,+)$ with a scalar
multiplication $\tT\times \mathcal A \to \mathcal A$ satisfying $a
(\sum _{j=1}^u b_j) = \sum _{j=1}^u (a b_j), \ \forall a \in \tT$
(distributivity with respect to $\tT$).

A  $\tT$-\textbf{module}   over $ \tT$
 is a $\tT$-\module0 $( \mathcal A,+,\zero_\mathcal A)$
satisfying $a\zero_{\mathcal A }= \zero_{\mathcal A }$ for any $a
\in \mathcal T$.
\end{defn}

In other words, $\tT$ acts on $\mathcal A$.  Often  $\tT$ is a
multiplicative group. For example, $\tT$ might be a hyperfield
generating $\mathcal A$ inside its power set. Or $\tT$ might be an
ordered group, and $\mathcal A$ its supertropical semiring (or
symmetrized semiring). Or $\mathcal A$ might be a fuzzy ring, and
$\tT$ its subset of invertible elements, as described in
\cite[Appendix B]{Row16}.
\begin{defn}
 A   $\tT$-\textbf{monoid module} over a monoid $ (\tT, \cdot, \one)$
 is a  $\tT$-module $(\mathcal A,+)$ that also respects the monoid structure, i.e., $(\mathcal A,+)$ satisfies the
  following additional axioms, $\forall   a_i \in \tT,$ $b
\in \mathcal A$:

 \begin{enumerate}\label{distr31}\eroman
   \item   $ \one _\tT b =  b.$
   \item   $(a_1 a_2)b = a_1 (a_2 b).$
 \end{enumerate}

We delete the prefix $\tT$- when it is unambiguous. When  $(\tT,
\cdot)$ is a group, we call $\mathcal A$ a  \textbf{\GM}\ to
emphasize this fact. In \S\ref{loc} we shall see how localization
permits us to reduce from monoid modules to \GM s.
\end{defn}

%
%

\subsubsection{Negation maps}$ $

We need some formalism to get around the lack of negation.

\begin{defn}\label{negmap}
 A \textbf{negation map} on  a $\tT$-module  $\mathcal A$
is a
 semigroup automorphism
$(-) $ of $(\mathcal A,+)$ and a map $\mathcal T \to \mathcal T$ of
order at most two, also denoted as $(-),$
 respecting the $\tT$-action in the sense that
$$a((-)b) = ((-)a)b$$ for $a \in \tT,\ b \in \mathcal A.$
When $\mathcal T$ is a subset of $\mathcal A$, we   require that
$(-)$ on $\mathcal A$ restricts to $(-)$ on $\tT$.
\end{defn}

For monoid modules, it is enough to know $(-)\one_{\tT}:$

 \begin{lem}\label{circm}
Let $\tT$ be a monoid, and take $\varepsilon$ in $\tT$ with
$\varepsilon^2 = \one.$
 \begin{enumerate}
 \eroman
    \item  There is a unique negation
map on $\tT$ and $\mathcal A$ for which $(-)\one_{\tT} =
\varepsilon,$ given by $(-)a = \varepsilon a$ and $(-)b =
\varepsilon b$ for $a \in \tT, \ b \in \mathcal A$. Furthermore,
\begin{equation}\label{neg0}
(-)(ab) = ((-)a)b = a((-)b).
\end{equation}
\item When $\mathcal A$ also is a \semiring0 and $\tT \subseteq  \mathcal A$
generates $(\mathcal A,+),$ then \eqref{neg0} holds
 for any $a,b \in \mathcal A$.
      \end{enumerate}
\end{lem}
\begin{proof} (i)
$a((-)b) = a(\varepsilon   b) = a\,\one_{\tT}\, \varepsilon b =
\varepsilon a\,\one_{\tT}\,  b = ((-)a)b,$ and $(-)(ab) =
\varepsilon (a b) = (\varepsilon a)   b = ((-)a)b.$

(ii) Write $a = \sum_i a_i$ for $a_i \in \tT$. Then
\[
(-)(ab) = \varepsilon\left(  \sum_i a_i\right)b =
\sum_i\left((-)a_i\right)b = \left(\sum_i a_i\right)((-)b)= a((-)b)
.
\]
\end{proof}


 We write $a(-)b$ for $a
+ ((-)b),$ and $a^\circ$ for $a(-)a,$ called a \textbf{quasi-zero}.
Thus, in classical algebra, the only quasi-zero is $\zero$ itself if
the negation map is the usual negation.
 Of special interest are the sets $\mathcal
A^\circ := \{ a ^\circ: a \in \mathcal A\}$ and $\tT^\circ := \{ a
^\circ: a \in \tT \}$,  inferred from the ``diodes'' of \cite{Ga},
Izhakian's thesis \cite{zur05TropicalAlgebra}, and the ``ghost
ideal'' of \cite{IR} and studied explicitly in the symmetrized case
in \cite[\S 3.5.1]{Row16}, and \cite[\S 4]{CC2} (over the Boolean
semifield $\mathbb B$).

\begin{example}\label{neg217} Major instances of negation maps:
 \begin{enumerate}
 \eroman
    \item Equality (taking $\varepsilon = \one$).
      \item The switch map in symmetrization,
      in Definition~\ref{wideh7} below (taking $\varepsilon = (\zero, \one)$ under the twisted multiplication).
\item The negation map in a hypergroup (taking $\varepsilon = -\one$).
\item The negation map $(-)a \mapsto \varepsilon a$ in a fuzzy ring.
        \end{enumerate}
\end{example}

%

%
%
%
%
%
%

\subsubsection{Pseudo-triples and triples}$ $


\begin{defn}\label{modu21} $ $
\begin{enumerate}\eroman
   \item
     A  \textbf{pseudo-triple} $(\mathcal A, \tT, (-))$ is a $\tT$-module $\mathcal A$ with a negation map
        $(-)$.
   \item
     A $\tT_{\mathcal
        A}$-\textbf{pseudo-triple} $(\mathcal A, \tT_{\mathcal
        A}, (-))$ is a pseudo-triple  $(\mathcal A, \tT_{\mathcal
        A}, (-))$, with $\tT_{\mathcal A}$ designated as a distinguished subset of $\mathcal A$.
        The  elements of $\tT_\mathcal A$ are called \textbf{tangible}.
    \item
    A $\tT_{\mathcal
        A}$-\textbf{triple} is a $\tT_{\mathcal
        A}$-pseudo-triple $(\mathcal A, \tT_{\mathcal A}, (-))$, in which
    $\tT_{\mathcal A}  \cap \mathcal A^ \circ = \emptyset$ and
    $\tT_{\mathcal A} $ generates $( \mathcal A,+).$ If $\mathcal A$
    has a zero element~$\zero$ we only require that $\tT_{\mathcal A}$
    generates $(\mathcal A \setminus \{\zero\},+).$
    \item
A $\tT$-\textbf{\semiring0 pseudo-triple} $(\mathcal A,
\tT_{\mathcal
    A}, (-))$ is $\tT_{\mathcal
        A}$-pseudo-triple for which $\mathcal
A$ is also  a \semiring0, 
and the \semiring0 multiplication on $\mathcal A$ restricts to the
$\tT_{\mathcal
        A}$-\module0
multiplication $\tT_{\mathcal A} \times \mathcal A \to \mathcal A.$
\end{enumerate}

\end{defn}

\begin{rem}\label{ps1}  The condition  that $\tT_{\mathcal A} \cap
\mathcal A^\circ = \emptyset$  in (iii), fails in the max-plus
algebra, but holds in the other examples of interest to~us.

When $\one_A \in \mathcal A,$ we can put $\tT_{\mathcal
        A} = \tT \one_A,$ and thus get a $\tT_{\mathcal
        A}$-pseudo-triple.
        The most straightforward way of ensuring that $\tT_{\mathcal A}$ generates $(\mathcal A,+)$ is   to restrict $(\mathcal A,+)$ to the
$\tT$-submodule generated by~$\tT_{\mathcal A}$.  In this case we
define the \textbf{height} of an element $c \in \mathcal A$ as the
minimal $t$ such that $c = \sum _{i=1}^t a_i$ with each $a_i \in
\tT_{\mathcal A}.$ ($\zero$ has height~ 0.) The \textbf{height}
of~$\mathcal A$ is the maximal height of its elements (which is said
to be~$\infty$ if these heights are not bounded).
\end{rem}

The following properties play a basic role in the theory.

\begin{defn} \label{metadef} $ $
\begin{enumerate} \eroman
    \item
        A $\tT$-triple $ (\mathcal A, \tT_\mathcal A , (-))$ has
    \textbf{unique negation} if $a_0 + a_1 \in \mathcal A^\circ$ for
    $a_i \in \tT_{\mathcal A}$ implies $a_1 = (-)a_0.$
    \item
$(\mathcal A, \tT_{\mathcal A}, (-))$  is \textbf{meta-tangible} if
$a_0+a_1 \in \tT_{\mathcal A}  $ for any $a_1 \ne
(-)a_0$ in $\tT_{\mathcal A}  $.
\item
    (A very important special case) A $\tT$-pseudo-triple $(\mathcal A,
    \tT_\mathcal A, (-))$ is  $(-)$-\textbf{bipotent} if $a + b \in \{a ,b\}$
    whenever $a, b \in \tT_\mathcal A$ with $b \neq (-) a.$ In other words, $a + b
    \in \{a ,b, a^\circ \}$ for all $a,b \in \tT_\mathcal A$.
\end{enumerate}

\end{defn}

The unique negation property is needed to get started in the theory,
and meta-tangibility is a rather pervasive property studied in
\cite{Row16}.

\begin{defn} \label{homodef} $ $
  A
    \textbf{homomorphism}
    $\varphi: (\mathcal A, \tT_\mathcal A , (-)) \to  (\mathcal A', \tT_{\mathcal A '}, (-))$
of pseudo-triples is a $\tT$-module homomorphism $\varphi: \mathcal
A\to \mathcal A'$ satisfying $\varphi(a) \in \tT_\mathcal A ' \cup
\{ \zero\},\footnote{One might prefer to require $\mathcal A$ only
to be a $\tT$-\module0 with $\varphi(a) \in \tT_\mathcal A '$ for
all $a\in \tT_\mathcal A$, but we want to permit projections of the
free module to be module homomorphisms.} \forall a \in \tT_\mathcal
A$, and $\varphi((-)a) = (-)\varphi(a),\ \forall a \in \tT_\mathcal
A$.
\end{defn}

\subsubsection{Surpassing relations and systems}$ $

Having generalized negation to the negation map, our next goal is  a workable substitute for equality.

\begin{defn}[ {\cite[Definition 4.5]{Row16}}]\label{precedeq07}
A \textbf{surpassing relation} on a pseudo-triple $(\mathcal A, \tT,
(-))$, denoted
  $\preceq $, is a partial pre-order on $ \tT$ and on $\mathcal A$ satisfying the following, for elements  $a_i, a \in \tT$ and $b_i, b \in  \mathcal A$:

  \begin{enumerate}
 \eroman
  \item   $b_1 \preceq b$
  whenever $b_1 + c^\circ = b$ for some $c\in \mathcal A^\circ$.
     \item If $b_1 \preceq b_2$ then $(-)b_1 \preceq (-)b_2$.
      \item   If   $b_1 \preceq b_2$ then $ab_1 \preceq ab_2.$
 \item If $b_1 \preceq b_2$ and $b_1' \preceq b_2'$ for $i= 1,2$ then  $b_1 + b_1' \preceq b_2
   + b_2'.$
   \item  If $a_1 \preceq a_2$ then $a_1 =  a_2.$
   \end{enumerate}
By a \textbf{surpassing PO}, we mean a surpassing relation that is a PO. A \textbf{$\tT$-surpassing relation} on a $\tT$-pseudo-triple $\mathcal A$  is a surpassing relation which also satisfies the following extra condition:
\[
b^\circ \not \preceq a \textrm{ for any }a \in \tT, \ b \in \mathcal A.
\]
\end{defn}

\begin{lem} When $\zero \in \mathcal A$, condition (i) is a formal consequence of the following weaker condition:
\[
  \zero  \preceq c^\circ,~\forall c \in
\mathcal A.
\]
\end{lem}
\begin{proof}  Since $\zero
\preceq c^\circ$, it follows from $(iv)$ that $b_1 = b_1 + \zero
\preceq  b_1 + c^\circ = b .$
\end{proof}

\begin{defn}\label{precmain0}
Let $(\mathcal A, \tT, (-))$ be a $\tT$-pseudo-triple.

     We define the  relation $\preceq_\circ$ on $\mathcal A$ by:
    \[
    a \preceq_\circ b \textrm{ if and only if } b=a + c^\circ \textrm{ for some } c \in \mathcal A.
    \]
\end{defn}

\begin{defn}
    We define the following subset of $\mathcal A$:
        \[
{\mathcal A}_{\Null} = \{ b \in \mathcal A: b+b'\succeq b', \forall
b' \in  \mathcal A\}.\]
\end{defn}

 \begin{rem}\label{PO1} \begin{enumerate}\eroman\item
One can easily check that $ {\mathcal A}_{\Null}$ is a sub-\module0
of ${\mathcal A}$. Furthermore, when $\zero \in \mathcal A,$
$\mathcal A_{\Null}$ has the following simpler description
\[
{\mathcal A}_{\Null} = \{ b \in \mathcal A: b\succeq \zero \}.
\]
\item  ${\mathcal A}_{\Null}
= \mathcal A^\circ$,  when $\zero \in \mathcal A$ and $\preceq$ is
$\preceq_\circ$ as given in Definition~\ref{precmain0}. In this
case, for a $\tT$-pseudo-triple, ${\tT}_{\Null} := \emptyset $.

\item In the hypergroup setting of
\cite[\S 3.6, Definition~4.23]{Row16}, ${\mathcal A}_{\Null}$
consists of those sets containing $\zero$, which is the version
usually considered in the hypergroup literature, for instance, in
\cite{GJL}.
\item ${\mathcal A}_{\Null}= \mathcal A\setminus \{\zero\}$ for the Green relation \cite[Remark~3.1(i)]{Row16}, so in this case
the theory degenerates.
\end{enumerate}
\end{rem}

A motivating example from classical algebra (for which $\preceq$ is
just equality) is for $\mathcal A$ to be an associative algebra
graded by a monoid; $\tT$~could be the subset of homogeneous
elements, in particular, a submonoid of~$(\mathcal A, \cdot)$. A
transparent example is when $\tT$ is the multiplicative subgroup of
a field~$\mathcal A$. We are more interested in the non-classical
situation, involving semirings which are not rings. Our structure of
choice is as follows:

\begin{defn}
  A \textbf{system} (resp.~\textbf{pseudo-system}) is a quadruple
  $(\mathcal A, \tT  ,
(-), \preceq),$ where $\preceq$ is a surpassing relation  on the
triple
  (resp.~pseudo-triple) $(\mathcal A, \tT  , (-))$, which is \textbf{uniquely
negated} in the sense that if $a + a' \in \tT_\Null$ for $a, a' \in
\tT,$ then $a' = (-)a$. (Compare with Definition~\ref{metadef}.)

A \textbf{semiring system} is a system where $\mathcal A$ is a
semiring.

A \textbf{semiring-group system} is a semiring system where $\tT$ is
a group.

The \textbf{default $\circ$-pseudo-system} of a pseudo-triple
$(\mathcal A, \tT ,
 (-))$ is
$(\mathcal A, \tT , (-),\preceq_\circ )$.

$\tT$-systems, etc., are defined analogously, where we assume that
$\tT \subseteq \mathcal A.$\end{defn}

It is convenient to modify $\preceq$, generalizing
Definition~\ref{precmain0}.
\begin{defn} Given a pseudo-system    $(\mathcal A, \tT  ,
(-), \preceq),$ define $\precpr$ by $b \precpr b'$ if $b' = b+ c$
for some $c \in \mathcal A_\Null.$

 The \textbf{default $\preceq$-pseudo-system} of a pseudo-system $(\mathcal A, \tT ,
 (-))$ is
$(\mathcal A, \tT , (-),\precpr )$.
\end{defn}

$\precpr$ is a PO if it satisfies   the condition, called
\textbf{upper bound}~(ub), that $a + b +c = a$ for $b,c \in \mathcal
A_{\Null}$ implies $a+b = a$.

When $\mathcal A_\Null\mathcal = \mathcal A^\circ$, then $\precpr$
is just $\preceq_\circ$, which happens in virtually all of our
applications, in which case we can skip the technicality of
utilizing the system in defining another surpassing relation.

%
Systems
 are the main subject in
\cite{Row16}, employed there to establish basic connections with
tropical structures, hypergroups, and  fuzzy rings, by means of the
following  examples:

\begin{itemize}
\item (The standard supertropical $\tT$-triple) $(\mathcal A, \tT, (-))$ where  $\mathcal A = \tT
\cup \mathcal A ^\circ$ and $(-)$ is the identity map. ($ \mathcal A
^\circ$ is   called the set of ``ghost elements'' $\tG$.) We get the
default $\circ$-system,  and $\preceq$ is a PO.

\item (The hypersystem  \cite[\S 3.6.1]{Row16}) Let $\tT$ be a hyperfield. Then one can associate a triple $(\mathcal S (\tT),\tT,(-))$, where $\mathcal P (\tT)$
is its power set (with componentwise operations), $\mathcal S (\tT)$
is the additive sub-semigroup of $\mathcal P (\tT)$ spanned by the
singletons, and $(-)$ on the
 power set is induced from the hypernegation. Here $\preceq$ is $\subseteq$, which   is a PO.

\item Symmetrized systems. (See  \S \ref{symm} below.)

\item   Fuzzy
 rings are described as default $\circ$-systems,  in \cite[Appendix B]{Row16}.
\end{itemize}

 This wealth of examples   motivates a
further development of the algebraic theory of $\tT$-systems, which
is the rationale for this paper.

\begin{rem} Recently \cite{BB1}, \cite[\S1.2.1]{BL} defined a \textbf{tract} to be
a   pair $(G;N_G)$ consisting of an abelian group $G$ (written
multiplicatively), together with a subset $N_G$ (called the nullset
of the tract) of the group semiring $N[G]$   satisfying:
\begin{itemize}
\item The zero element of $N[G]$ belongs to $N_G$, and the identity
element 1 of $G$ is not in $N_G$.

\item  $N_G$ is closed under the natural
action of $G$ on $N[G]$.

\item There is a unique element $\varepsilon$ of G with $1+\varepsilon \in
N_G$.
\end{itemize}

Tracts  are special cases of semiring-group $\tT$-systems, where
$\tT$ is the given Abelian group $ G$, $\mathcal A = \mathbb N[G],$
$\vep = (-) \one,$ and $N_G $ can be taken to be ${\mathcal
A}_{\operatorname{Null}},$ often taken to be $\mathcal A^\circ$.

Likewise, pasteurized  blueprints \cite{BL} are semiring
$\tT$-systems with unique negation, when the map $\preceq$ is  a PO.
\end{rem}




\subsection{Symmetrization and the twist action}$ $\label{symm}

In this subsection, we utilize an idea of Gaubert \cite{Ga} to
provide a negation map and surpassing relation for an arbitrary
$\tT$-module $\mathcal{A}$, when it is lacking.

\begin{defn} \label{Liesup}
Let $\tT  = \tT_0 \cup \tT_1$.    A $\tT$-module $\mathcal{A}$ is
said to be a \textbf{$\tT$-super module} if $\mathcal{A}$ is a
 $\Z _2$-graded semigroup $ {\mathcal A}_0 \oplus
{\mathcal A}_1$, satisfying   $\tT_i {\mathcal A}_j \subseteq
{\mathcal A}_{i+j}$, where $\tT_i = \tT \cap \mathcal A_i $,
subscripts modulo 2.
\end{defn}

 From now on, $\widehat{\mathcal A}$
 denotes $\mathcal A \times \mathcal A$, i.e., ${\mathcal A}_0 =
{\mathcal A}_1 ={\mathcal A}. $

 \begin{defn}\label{wideh7}  The \textbf{switch map} on $\widehat{\mathcal A}$ is given by
 $(b_0,b_1)\mapsto (b_1,b_0).$ \end{defn}

 In order to identify the second component as the negation of
 the first, we employ an idea dating back to Bourbaki~\cite{Bo} and the Grothendieck group completion, as well as
 \cite{BeE,CC2,Ga,JoM, Ju1, Row16} (done here in generality for
 $\tT$-modules),  which   comes from the familiar construction
 of $\Z$ from~$\N$. The idea arises from the elementary computation:
 $$(a_0-a_1) (b_0-b_1) =
 (a_0b_0 + a_1 b_1)-( a_0 b_1 + a_1 b_0).$$

\begin{defn}\label{twistact} $\widehat{\mathcal A}$ is called the \textbf{symmetrization}
of $\mathcal A.$ For any $\tT$-module $\mathcal A$, the
\textbf{twist action}, denoted $\ctw$, on~$\widehat{\mathcal A}$
over $\widehat \tT $
 is given by the super-action, namely
 \begin{equation}\label{twi} (a_0,a_1)\ctw (b_0,b_1) =
 (a_0b_0 + a_1 b_1, a_0 b_1 + a_1 b_0).  \end{equation}

 When $\tT$ is a monoid, we view $\widehat
\tT $ as a monoid, also via the \textbf{twist action} as in
\eqref{twi}, where $a_i,b_i \in \tT$, the    unit element of
$\widehat \tT $ being $(\one_\tT, \zero).$
\end{defn}

\begin{lem} $\widehat{\mathcal A}$ is a
$\widehat \tT $-module. When $\mathcal A$ is a $ \tT $-monoid
module, $\widehat{\mathcal A}$ is a $\widehat \tT $-monoid module.
\end{lem}
\begin{proof} To see that the twist action is associative over $\widehat \tT $, we note for $a_i, b_i \in \tT$ and $c_i \in \mathcal M$
 that \begin{equation}\begin{aligned}((a_0,a_1)\ctw (b_0,b_1)) \ctw (c_0,c_1) & =
 (a_0b_0 + a_1 b_1, a_0 b_1 + a_1 b_0) \ctw (c_0,c_1)\\ & = (a_0b_0c_0 + a_1 b_1c_0 + a_0 b_1c_1
  + a_1 b_0c_1, a_0b_0c_1 + a_1 b_1c_1+ a_0 b_1c_0
 + a_1 b_0c_0),\end{aligned}\end{equation}
 in which we see that the subscript $\phantom{w}_0$ appears an odd
 number of times on the left and an even
 number of times on the right, independently of the original placement of
 parentheses.
\end{proof}

 \begin{rem}\label{princmod} For any $(a_0, a_1) \in \widehat{\tT},$  $\{(a_0, a_1)\}\ctw
\widehat{\mathcal A} \subseteq \widehat{\mathcal A}$ yields a $ \tT
$-submodule of $\mathcal A$.
\end{rem}

Symmetrization plays an especially important role in geometry,
cf.~\S\ref{weakn}, and in homology theory, to be seen in
\cite{JMR1}.

But $\widehat \tT $ is too big for our purposes (since it is not a
group), so we take instead  $\tT_{\widehat{\mathcal A} }
    =(\tT \times \{ \zero \}) \cup ( \{ \zero \} \times \tT)$.

\begin{lem} If $\tT$ is a group, then $\tT_{\widehat{\mathcal A}}  $ also is a
group under the twist action.\end{lem} \begin{proof} In
$\tT_{\widehat{\mathcal A}} $, $(a,\zero)^{-1} = (a^{-1},\zero)$ and
$(\zero,a)^{-1} = (\zero,a^{-1})$.\end{proof}
\begin{defn}\label{symr1}
Let $ \mathcal A$ be a $\tT$-module.
\begin{enumerate}\eroman
    \item
    The \textbf{symmetrized pseudo-triple} $(\widehat {\mathcal A}, {\tT}_{\widehat{\mathcal A} },(-))$ is a $\tT$-pseudo-triple,
where $\widehat {\mathcal A} = \mathcal A
    \times \mathcal A$ with componentwise addition, and  with
    multiplication $\widehat {\tT}_{\widehat{\mathcal A} } \times \widehat {\mathcal A}\to
    \widehat {\mathcal A}$ given by the twist action. Here we take $(-)$ to be the switch
    map.
    \item
The \textbf{symmetrized pseudo-system} is the default $\circ$-system
of the symmetrized pseudo-triple. Here
\[
\widehat {\mathcal A}_\Null = \widehat {\mathcal A}^\circ = \{(b,b)
: b \in \mathcal A \}.
\]
\end{enumerate}
\end{defn}

We identify $\mathcal{A}$ inside $\widehat{\mathcal{A}}$ via the
injection $a \mapsto (a,\zero).$


 \begin{lem} $(\widehat{\mathcal A},
\tT_{\widehat{\mathcal A}} , (-))$  is a triple with unique negation
and $\widehat{\tT} ^\circ = \{ (a,a): a \in \tT \setminus \{\zero
\}\}.$
\end{lem}
\begin{proof}
The first assertion is clear, and $\widehat{\tT} ^\circ = \{ (a,a):
a \in \tT \setminus \{\zero\}\}$ since for any $(a,\zero)$ (resp.~
$(\zero,a)$), $(-)(a,\zero)=(\zero,a)$
(resp.~$(-)(\zero,a)=(a,\zero)$) and hence $(a,a) = (a,\zero)+
(\zero, a) = (a,\zero)(-)(a,\zero).$
\end{proof}

 The
triple $\widehat{\mathcal A}$   is not meta-tangible. This can be
rectified by redefining addition on ${\tT}_{\widehat{\mathcal A}}$
to make it meta-tangible as done in \cite{Ga} or
\cite[Example~3.9]{Row16}, but here we find it convenient to use the
natural (componentwise) addition on $\widehat{\mathcal A}$, to make
it applicable for congruences.


Here is an important application of the twist, whose role in
tropical geometry is featured in \cite{JoM}.

\begin{defn}\label{symr}
The \textbf{symmetrized \semiring0} $\widehat{\mathcal{A}}: =
\mathcal A \times \mathcal A$ of a semiring $\mathcal{A}$ is
$\widehat{\mathcal{A}}$ viewed as a symmetrized
${\tT}_{\widehat{\mathcal A} }$-module, and made into a semiring via
the ``twisted'' multiplication of \eqref{twi} for $a_i, b_i \in
\mathcal{A}$.
\end{defn}



%


%


\subsection{Ground systems and module
systems}\label{tripsys1}$ $

%

Representation theory often is described in terms of an abelian
category, such as the class of modules over a given ring.
Analogously, there are two main ways of utilizing pseudo-systems.
%

\subsubsection{Ground systems}\label{tripsys11}$ $

We call a triple (resp.~system) $(\mathcal A, \tT, (-))$ a
\textbf{ground triple (resp.~ground  system)} when we study it as a
small category with a single object in its own right, often a
commutative semidomain. In short, our overall strategy is to fix a
ground
 $\tT$-triple $(\mathcal A, \tT, (-), \preceq),$  often $(-)$-bipotent,
  and then
consider its ``prime'' homomorphic images, as well as the module
systems over this ground
 $\tT$-system, to be defined presently.

\subsubsection{Module systems}\label{tripsys117}$ $

One defines a  \textbf{module} (called \textbf{semimodule} in
\cite{golan92}) over a semiring $\mathcal A$, in analogy to
 modules over rings.

\begin{defn} $ $
\begin{enumerate}\eroman
\item
Let $\mathcal A=(\mathcal A,\tT_\mathcal A,(-))$ be a
$\tT$-\semiring0 triple. A $\tT_{\mathcal M}$-triple $\mathcal M :=
(\mathcal M, \tT_{\mathcal M} , (-))$ is said to be a \textbf{left}
$\mathcal A$-\textbf{module triple} over a monoid triple $\mathcal A
= (\mathcal A, \tT_\mathcal A, (-))$ if $\mathcal M$ is an $\mathcal
A$-module such that $ \tT_\mathcal A$ acts on $\tT_\mathcal M$ and
satisfies the following condition:
\[
((-)a)b = a((-) b) = (-)(ab), \quad \forall a \in \tT_\mathcal A, \
b \in \tT_{\mathcal M}.
\]
Analogously, we define a \textbf{right} $\mathcal A$-\textbf{module}
and \textbf{right} $\mathcal A$-\textbf{module
 triple} from the other side, and an $(\mathcal A, \mathcal A')$-\textbf{bimodule
 triple} when  $\mathcal M$ is an $\mathcal A, \mathcal A'$-bimodule.
\item
 An  $\mathcal A$-\textbf{module system} over a ground
$\tT$-system $\mathcal A = (\mathcal A, \tT, (-),\preceq)$, is an $
\mathcal A $-module triple $(\mathcal M, \tT, (-))$ with a
surpassing relation  satisfying $ a_1 b_1 \preceq a_2b_2$ whenever
$a_1 \preceq a_2$ in $\tT $ and $b_1 \preceq
 b_2$ in $\mathcal M$. Analogously, one can define right module systems and bimodule
 systems.

\item
A \textbf{\GM\  system} is an $\mathcal A$-module system for which
$\tT_\mathcal A$ is a group.
\end{enumerate}
\end{defn}

So we study module systems over a fixed ground triple.
 We use $\mathcal M$ instead of $\mathcal A$ to denote a module
system.
\begin{rem}\label{morph79}The twist action (Definition~\ref{twistact}) on the module
$\widehat{ \mathcal M}$ over $\widehat{ \mathcal A}$ extends the
$\tT_{\mathcal A}$-module
      action on $\widehat{ \mathcal M}.$  Indeed, suppose $(x_0,x_1) \in \widehat{\mathcal M}$ and $(a_0,\zero) \in \widehat{\tT_{\mathcal A}}$.
Then $(a_0,\zero) \ctw (x_0,x_1)= (a_0 x_0,a_0 x_1)\in
\widehat{\mathcal M} $. \end{rem}

 \begin{prop}\label{congmult} If   $\mathcal N$ is a sub-triple of an   $\mathcal A $-module triple
 $\mathcal M$  with negation map
 $(-)$, then $\mathcal N$ becomes  a $\widehat{\tT_{\mathcal A}}$-submodule of $\mathcal
 M$ under the action $(a_0,a_1)x = a_0x (-) a_1 x.$\end{prop}
\begin{proof} $(a(a_0,a_1))x = aa_0x (-) aa_1 x = a((a_0,a_1)x).$
Likewise for addition.
\end{proof}

 We note  a conflict between the switch map on $\widehat {\mathcal A} $
and a given negation map on $\mathcal A $, which do not match;
$(a_1, a_0) = \widehat{(-)}(a_0,a_1) \ne ((-)a_0,(-)a_1)$ unless
$a_1 = (-)a_0$. Fortunately this does not affect
Proposition~\ref{congmult} since
$$(a_1,a_0)x =  a_1x (-) a_0 x = (-)( a_0x (-) a_1 x) =
(-)((a_0,a_1)x)= (((-)a_0,(-)a_1)x).$$

\subsubsection{The characteristic sub-triple}$ $

\begin{defn}
A \textbf{sub-$\tT$-triple} of a  $\tT$-triple $(\mathcal A,
\tT_\mathcal A, (-))$ is a   triple $(\mathcal {A'}, \tT_{ \mathcal
A'}, (-))$ where $\tT_{ \mathcal A'}$ is a subset of $\tT_\mathcal
A$ (with the relevant structure) and $\mathcal {A'}$ is the
sub-$\tT$-module of $\mathcal A$ generated by $\tT_{ \mathcal A'}$.
%

\end{defn}

\begin{example}\label{sub7} Suppose $(\mathcal A, \tT_{\mathcal A},
(-))$ is a triple. The \textbf{characteristic sub-triple} $\mathcal
A_\one$ is the sub-triple generated by $\one,$ which is $\mathcal
A_\one: = \{ \one, (-)\one,\, e:=\one(-)\one,\, \dots\}$ and
$\tT_{\mathcal A_\one} := \{ \one, (-)\one\}.$ If $\mathcal A$ is a
\semiring0 then clearly $(\mathcal A, \tT_{\mathcal A}, (-))$ is a
module triple over $\mathcal A_\one$. This ties in with other
approaches to tropical algebra, and to some fundamental hyperfields,
as follows.

If $e=\one$ then we have the Boolean semifield, so assume that $e
\ne \one.$

If $e + \one =\one$ then   $e$ behaves like $0$, and we wind up with
$\mathcal A_\one$ isomorphic to $\Z$ or $\Z/n\mathbb Z$ for some
$n>1$.

So assume that $e + \one \ne\one$.
 In
height 2, $e + \one \in \{(-) \one, e\}.$ Recall that a negation map
is of the first kind if $(-)a=a$ for any $a \in \tT$, and of the
second kind if $(-)a \neq a$ for any $a \in \tT$ (\cite[Definition
2.22]{Row16})).

  \begin{enumerate}
 \eroman
    \item
    If $(-)$ is of the first kind     (as in the supertropical case), then $e+ \one = e.$
       $\mathcal A_\one$ is $\{ \one,   e\}$. This corresponds to the   ``Krasner hyperfield''  $K = \{ 0, 1 \}$
with the usual operations of Boolean algebra, except that now $ 1
\boxplus  1 = \{ 0, 1 \} ,$  and we can identify $\{ 0, 1 \} $ with
$ 1^\nu.$

 \item  If $(-)$ is of the second kind, the we have two cases.
      \begin{enumerate}
\item In the $(-)$-bipotent
case $\mathcal A_\one = \{ \one, (-)\one, e\}$, with $\one+ \one =
\one,$ which is the symmetrized triple of the trivial idempotent
triple $\{ \one \}.$ This corresponds to the ``hyperfield of signs''
$ S := \{ 0, 1 , -1\}$  with the usual multiplication law and
hyperaddition defined by $1 \boxplus  1 = \{ 1\} ,$ $-1 \boxplus  -1
= \{ -1\} ,$ $ x \boxplus  0 = 0 \boxplus  x = \{ x\} , \forall x,$
and $1 \boxplus  -1 = -1 \boxplus  1 = \{ 0, 1,-1\} = S.$
 \item $(-)$ is of the second kind but
 non-$(-)$-bipotent.
Then $e+ \one = (-)\one,$ which leads to a strange structure of
characteristic 4 since $\mathbf 2 = (e (-)\one)+\one = e + e =
(-)\one(-)\one = (-)\mathbf 2 $, $e + \mathbf 2 = \one (-)\one = e$,
and $e+ e +e = ( e +\one) + e +(-)\one = e$.
    \end{enumerate}
      \end{enumerate}
    In either case one could adjoin $\{ \zero \},$ of course.
    Height $>2$ is more intricate, involving layered structures.
\end{example}

\subsubsection{The  $\Hom$ triple}\label{sysrep}$ $

 $\Hom (\mathcal M,\mathcal N)$ denotes  the set of morphisms from
$\mathcal  M$ to $\mathcal N$ and $\Hom (\mathcal M,\mathcal
N)_{\tT}$ be the subset of $\tT$-morphisms. We use the given module
negation map to define a negation map on $\Hom$, and get a
pseudo-triple.
%
%

\begin{rem}  Suppose $\mathcal M: = (\mathcal M, \tT_{\mathcal M}, (-), \preceq)$ and $\mathcal N : = (N, \tT_{\mathcal N},
  (-)
  \preceq)$ are module systems over a ground $\tT$-system $  (\mathcal A, \tT,
(-), \preceq).$ If $\mathcal  M,\mathcal N$ are $\tT$-monoid modules
then so is $\Hom (\mathcal M,\mathcal N)$, and, for $\mathcal A$
commutative, $\tT$ acts on $\Hom _\tT (\mathcal M,\mathcal  N)$ via
the left multiplication map $\ell_a $ given by $\ell_a(f) = af.$ The
action is elementwise: $(af)(x) : = a f(x).$ We view $\tT_{\Hom
(\mathcal M,\mathcal N)}: = \{f \in \Hom (M,N): f(\tT_{\mathcal M})
\subseteq \tT_N\} $ in $\Hom (M,N)$ via these left multiplication
maps.\end{rem}

  \begin{prop}\label{endo}$ $

\begin{enumerate}  \item  \eroman
  \ $\Hom
(\mathcal M,\mathcal N): = (\Hom (\mathcal  M, \mathcal N),
\tT_{\Hom (\mathcal M,\mathcal N)}, (-))$ is a pseudo-triple, where
 $(-)$   is defined elementwise, i.e., $((-)f)(x) = (-)(f(x))$.
 \item  The pseudo-triple $\Hom (\mathcal M,\mathcal N) $ has unique negation if $(\mathcal N, \tT_{\mathcal N},
  (-)
  \preceq)$ has unique negation.
 \item
  \ $\Hom
(\mathcal M,\mathcal N): = (\Hom (\mathcal  M, \mathcal N),
\tT_{\Hom (\mathcal M,\mathcal N)}, (-), \preceq)$ is a
pseudo-system, where   $f \preceq g$ if and only if $f(x) \preceq
g(x)$ for all $x \in M.$

    \item  If $\mathcal A$ is a $\tT$-semiring system, then
     $ \Hom (\mathcal M,\mathcal M),\tT_{\Hom (\mathcal M,\mathcal M)}, (-), \preceq)$ is a  semiring
    system.

 \end{enumerate}
  \end{prop}
 \begin{proof} (i)  Check all properties elementwise

 (ii) Suppose $f+g = h^\circ.$ Then $f(x)+g(x) = h(x)^\circ$ for each $x \in
 \tT_{\mathcal M}$,
 implying $g(x) = (-)f(x),$ and thus $g = (-)f.$

 (iii) Check  the
 conditions of Definition~\ref{precedeq07}       .

(iv) $fg(x) = f(g(x)).$

 \end{proof}

\begin{defn}  For a system $\mathcal S = (\mathcal S, \tT_{\mathcal S}, (-),
\preceq)$  write $\mathcal S ^*$ for $\Hom (\mathcal S,\mathcal A)$,
and  $\tT _{\mathcal S} ^*$  for   $\{ f|_{\tT_{\mathcal S}} : f \in
\mathcal S ^*$ with $f(\tT_{\mathcal S}) \subseteq \tT_{\mathcal
A}\}.$\end{defn}

\begin{prop}\label{endosys2}  Let $\mathcal S =
{\mathcal A}^{(I)},$ where  $\mathcal A = (\mathcal A, \tT, (-),
\preceq)$ is a system. For $I$ finite, the \textbf{dual system} $(
\mathcal S^* ,\tT_\mathcal S^*, (-),\preceq)$ also is a system,
where $(-)$ and $\preceq$ are defined elementwise.
\end{prop}
\begin{proof}  Let $\one_i$ denote the vector whose only component $\ne \zero$
is $\one$ in the $i$-th position. Given a morphism $f: \mathcal S
\to \mathcal A ,$ we define $f_{i }\in \tT_{\Hom (\tT_{\mathcal M},
\tT_{\mathcal N})}$ to be the map sending $\one_i$ to  $f(\one_i),$
and  $\tT_{i'}$ to $\zero$ for all other components $i' \ne i.$ It
is easy to see that $f$ is generated by the $f_{i }.$
\end{proof}

  (We need the hypothesis
that $I$ is finite in order for  $\tT _{\mathcal S} ^*$   to
generate $\mathcal S^*$.)

 As in usual linear algebra, when  $\mathcal A$ is commutative as well as associative, we can embed  $\mathcal S  $ into  $\mathcal S
 ^*$. Write $\bold a $ for $ (a_i)\in \tT^{(I)},$ and define  $\bold a^* \in  \mathcal S ^*$
 by $\bold a^*  ((b_i)) = \bold a \cdot (b_i) = \sum a_ib_i.$
Let $e_i$ denote the vector with $\one$ in the $i$ position and
$\zero$ elsewhere.

\begin{prop}\label{endosys3} Suppose  $\bold a = (a_i)\in \mathcal S =
{\mathcal A}^{(I)}$.
  Then $\bold a^* = \sum _i a_ie_i^*\in  \mathcal S^*$ is spanned over $\tT$  by the
$e_i^*$. There is an injection $(\mathcal S, \tT_{\mathcal S},(-))
\to (\mathcal S^*, \tT_{\mathcal S}^*,(-))$ given by $\bold a \to
\bold a^*,$ which is onto when $I$ is finite.
\end{prop}

\begin{proof} Just as in the classical case, noting that negation is
not used in its proof.
\end{proof}

 \subsection{Morphisms of systems}\label{extsys}$ $

We have two kinds of morphisms.

    \begin{defn}\label{mor} A \textbf{$\preceq$-morphism}  of   pseudo-systems $$\varphi:
(\mathcal A, \tT, (-), \preceq)\to (\mathcal A', \tT', (-)',
\preceq')$$ is a map $\varphi: \mathcal A \to \mathcal A'$ together
with $\varphi: \tT \to \tT '$ satisfying the following
properties  for $a  \in \tT$ and   $b\preceq b'$, $b_i$  in
$\mathcal A$:
\begin{enumerate}\eroman  \item $ \varphi((-)b_1)\preceq    (-)
\varphi(b_1);$
\item $ \varphi(b_1 + b_2) \preceq ' \varphi(b_1) + \varphi( b_2) ;$
\item  $ \varphi(a  b)\preceq  \varphi(a)   \varphi( b) $.
\item $ \varphi(b) \preceq ' \varphi(b').$
\item
$\varphi ( {\mathcal A}_{\Null} )\subseteq  \mathcal {A'}_{\Null}.$
\item When $\zero_\mathcal A \in \mathcal A,$ we also require that
$\varphi (\zero_\mathcal A ) = \zero_\mathcal A' .$
 \end{enumerate}

A \textbf{homomorphism}  of   pseudo-systems $\varphi: (\mathcal A,
\tT, (-), \preceq)\to (\mathcal A', \tT', (-)', \preceq')$ is
defined in the same way, but with equality holding in (i),(ii) and
(iii).
  \end{defn}

These will be cast in terms of universal algebra in
\S\ref{surpre202}.

\begin{rem}\label{morph71}
Let $\varphi: (\mathcal A, \tT, (-), \preceq)\to (\mathcal A', \tT',
(-)', \preceq')$ be a $\preceq$-morphism of pseudo-systems.
Conditions (ii) and (vi) imply (v) when $\zero \in \mathcal A$.
 \end{rem}

Even if $(\mathcal A, \tT, (-))$ is a system, $\varphi(\mathcal
A)\cap \tT_ {\mathcal {A'}}$ need not generate $ \mathcal {A'}$, so
we add this stipulation for morphisms of systems.

 As in classical module theory, when treating
$\preceq$-morphisms of module systems over a given ground system,
one always assumes that $\varphi$ is the identity on $\tT,$ so (iii)
becomes $ \varphi(a b) \preceq
 a    \varphi( b) $.

\begin{lem}\label{circm2}
 The map $a \mapsto a ^\circ$
 is a $ \preceq_\circ$-morphism of $\tT$-semiring systems.
 \end{lem}
\begin{proof}
 $a^\circ  b^\circ =   \mathbf 2 (ab)^\circ  = (ab)^\circ + c^\circ$
where $c = ab ,$ so $(ab)^\circ \preceq_\circ a^\circ  b^\circ .$
Also $(a+ b)^\circ  = a ^\circ + b ^\circ.$
\end{proof}

Although condition (ii) works well for hypersystems, (iii) does not
fit in so well intuitively, but fortunately
 the following
easy result will provide  equality for  (iii) in
Proposition~\ref{endo5}. We say that a $\tT$-module homomorphism
$\phi : \mathcal A \to \mathcal A $ is \textbf{invertible} if there
is some $\tT$-module homomorphism $\psi : \mathcal A  \to \mathcal A
$ such that $\psi\phi = 1_{\mathcal A } = \phi\psi$.

\begin{prop}\label{endo1}
Suppose that $\phi  :  \mathcal A     \to \mathcal A $ is an
invertible homomorphism, and $\preceq$ is a surpassing~PO. Then
 $f(\phi
 (b)) =
\phi  (f (b))$, for any $\preceq$-morphism $f: \mathcal A \to
\mathcal A' $ satisfying  $f(\phi
 (b)) \preceq
\phi  (f (b))  \ \forall b \in \mathcal A$.
 \end{prop}
 \begin{proof} $\phi  (f ( b) )=
\phi  (f (\psi \phi (b))) \preceq \phi \psi  (\phi  (f(b))) =  \phi
(f(b)),$ so we get equality.
  \end{proof}

We have the following consequences  at our disposal,   unifying
several ad hoc observations in~\cite{Row16}.


\begin{prop}\label{endo3} Any $\preceq$-morphism $f$ satisfies
 $f ((-)b) =
(-) f(b).$
  \end{prop}
 \begin{proof}  $(-)$ is an invertible homomorphism of additive semigroups, so Proposition~\ref{endo1} is applicable.
 \end{proof}


 \begin{prop}\label{endo5} Any $\preceq$-morphism $f$ of $\tT$-\GM s
 satisfies $f(ab) = af(b)$ for all $a\in \tT$ and $b \in \mathcal
 {A}.$
  \end{prop}
 \begin{proof} The left multiplication map $\ell_a$ by $a\in \tT$ is
 invertible on $\tT$,
 having the inverse $\ell_{a^{-1}},$ and thus is
 invertible on $\mathcal
 {A}.$
 \end{proof}

\begin{lem}
Any  $\preceq$-morphism $f$ with respect to a surpassing PO
  satisfies the following convexity condition:
 \medskip
 If $f(b_0)=f(b_1) $ and $b_0 \preceq b \preceq b_1,$ then
 $f(b_0)=f(b)$.
\end{lem}
\begin{proof}
$f(b_1) = f(b_0)  \preceq  f(b) \preceq f(b_1),$ so equality holds
at each stage.
\end{proof}

It follows that every $\preceq$-morphism ``collapses'' intervals, so
triples, systems, etc., do not provide varieties (since they are not
closed under arbitrary homomorphic images).

At times we need to decide  whether to take homomorphisms or
$\preceq$-morphisms.  Here is a   compromise.

\begin{defn}\label{idealdef2}
Let $\mathcal M$ and $\mathcal N$ be $\mathcal A$-module systems. A
$\preceq$-morphism $f:\mathcal M \to \mathcal N$ is
\textbf{$\tT$-admissible} when it satisfies the condition that if
$\sum_{i=1}^t  a_i  = \sum_{j=1}^u a'_j $ for $a_i,a_j' \in
\tT_{\mathcal M},$ then $\sum_{i=1}^t f(a_i) = \sum_{j=1}^u f(a'_j
)$.
\end{defn}

%

\begin{lem}\label{bil21}
Every  homomorphism is $\tT$-admissible.
\end{lem}
\begin{proof}
If $\sum_{i=1}^t  a_i  = \sum_{j=1}^u  a'_j $, then  $\sum f(a_i) =
f\left( \sum_i a_i \right) = f\left( \sum_j  a_j'\right) = \sum
f(a'_j ).$ \end{proof}



%

%



\subsection{Function triples}\label{convcat}$ $

 Here is a wide-ranging example needed for geometry and linear algebra,
 unifying polynomials and Laurent polynomials, cf.~\cite[Example~2.19]{Bak}, \cite{golan92}, \cite[\S 3.5]{IR}. It is convenient to assume that $\mathcal A$ is with $\zero$.

\begin{defn}\label{Func} For a   $\tT$-module  $( \mathcal A,+)$ over   $ \tT$ and a given set $S$, we  define $\mathcal A^S$ to
be
  the set of functions from $S$ to $\mathcal A$ \footnote{For an $\tT$-\module0 $( \mathcal
  A,+)$ without $\zero$, one would take
  $\mathcal A^S$ to
be
  the set of partial functions from $S$ to $\mathcal A$, and define  $\supp(f)$ to be those $s$
  on which $f$ is defined.}, also written as $\Fun (S,\mathcal A)$.
\end{defn}



%
%

For $c \in \mathcal A,$ the \textbf{constant function} $\tilde c $
is given by $\tilde c  (s) = c$ for all $s \in S.$ In particular,
  the \textbf{zero function}~$\tilde \zero $ is given
by $\tilde \zero(s) = \zero$ for all $s\in S.$

\begin{defn}  Given $f\in  \mathcal A ^S$ we define its \textbf{support } $\supp (f): =\{ s \in S: f(s )
\ne \zero\},$ and $ \supp ( \mathcal A ^S)$ for~$\{ \supp (f): f \in
 \mathcal A ^S\}$.
 \end{defn}

\begin{lem}  \label{canonicalmap1} For any
$f,g \in  \mathcal A ^S,$ we have the following:
\begin{enumerate}\eroman
  \item  $\supp (f+g) \subseteq \supp (f) \cup \supp (g)  $.
   \item (Under componentwise multiplication) $\supp (fg) \subseteq \supp(f) \cap \supp (g)  $.\end{enumerate}
\end{lem}
\begin{proof}
For the first statement, one can see that $f(s) = \zero = g(s)$
implies $f(s) + g(s) = \zero.$ The second statement is clear; $f(s)
= \zero$ or $ g(s)  = \zero  $ implies $f(s)  g(s) = \zero.$
\end{proof}

We must cope with a delicate issue. We have required for
$\tT$-triples that $\tT_{\mathcal A}$ generates $(\mathcal A,+)$.
Thus, $\mathcal A ^S$ a priori is only a pseudo-triple.  $\mathcal A
^S$ is a triple when $\mathcal A$ is a triple of finite height.
Alternatively, those maps having finite support play a special role.

 \begin{defn}
We introduce the following notations:
 \begin{itemize}
 \item
$ \mathcal A^{(S)}:=\{f\in \mathcal A ^S: \supp (f) $ is finite$\}$.
\item A
\textbf{monomial} is an element $f\in \mathcal A ^{(S)}$ for which
$|\supp (f)| = 1. $
\item  $  \tT_{ \mathcal A^{S}} =\{f\in \tT^{(S)}:
|\supp (f)|=1 \}$.
 \end{itemize}
 \end{defn}

 Explicitly, we
get polynomials when $S = \Net,$
 and Laurent polynomials when $S = \Z.$

\begin{lem}  \label{canonicalmap}
If $(\mathcal A, \tT, (-))$  is a pseudo-triple, then $(\mathcal
A^{(S)}, \tT_{ \mathcal A^{S}},(-))$ also is a   pseudo-triple,
where $ \tT_{ \mathcal A^{S}}:=\{f\in \tT^{(S)}: |\supp (f)|=1 \},$
and $(-)f(s) = (-)f(s)$ for $f\in \mathcal A^{S}$ and $s \in S$.
\end{lem}
\begin{proof}
This is clear, noting that any element of  $\mathcal A^{(S)}$ is a
finite sum of monomials.
\end{proof}

 If $\mathcal A$  is a
triple, then $\mathcal A^{(S)}$ also is a   triple since the
monomials span $\mathcal A^{(S)}$.

For systems, we define $ \preceq$  componentwise   on $\mathcal A^S$
by putting $f \preceq g$ when $f(s) \preceq g(s)$ for each $s$
in~$S$.

\subsubsection{Direct sums and powers}$ $

 \begin{defn}\label{cop} The \textbf{direct sum} $\oplus _{i \in I}(\mathcal A _i , \tT_{ \mathcal A _i}, (-))$ of pseudo-triples
  is defined as  $(\oplus \mathcal A _i, \tT_{\oplus \mathcal A _i},
  (-))$ where $\tT_{\oplus \mathcal A _i} = \cup   \tT_{ \mathcal A _i}$, viewed in $\oplus   \tT_{ \mathcal A _i}$  via  $\nu_i : \mathcal A _i \to \oplus \mathcal A
  _i$ being the canonical homomorphism.
  \end{defn}

Another natural possibility for $\tT_{\oplus \mathcal A _i}$ would
be $ \sum \tT_{ \mathcal A _i}$, viewed in $\oplus \mathcal A _i$,
but this essentially is the same, since $ \sum   _i
 \tT_{ \mathcal A _i}$ is generated by $\cup _i  \tT_{ \mathcal A _i}$. Definition~\ref{cop}   works out better for systems.

  \begin{prop}\label{cop1} The direct sum $\oplus (\mathcal A _i, \tT_i,
  (-))$ of $\tT_i$-triples is a $\tT_{\oplus \mathcal A _i}$-triple.
  \end{prop}
 \begin{proof}   $\tT_{\oplus \mathcal A _i}$ generates $\oplus \mathcal A
 _i$, and unique negation is obtained componentwise.
 \end{proof}

\begin{example}\label{modlift}
 When all $\mathcal A_i = \mathcal A,$ this takes us back to $\mathcal A
 ^{(S)}$, now viewed as a $\tT_{\mathcal A ^{(S)}}$-triple via the componentwise negation map.
 We often write $I = S$ and call $ ( \mathcal
A^{(I)},\tT_{\mathcal A}^{(I)},(-))$ the \textbf{free $\mathcal
A$-module triple}, to stress the role of $I$ as an index set and the
analogy with the free module.

This  provides the  \textbf{free module system} $(\mathcal A^{(I)},
\tT^{(I)}, (-), \preceq)$, cf.~\cite[Definition~2.6]{Row16}.
\end{example}

\begin{rem}\label{Funtri20}  When $(\mathcal M , \tT_\mathcal M , (-))$ is an  $\mathcal
A$-module triple, then  $( \mathcal M ^{(S)},  \tT_{\mathcal M
^{(S)}},(-))$  is an
 $\mathcal
A  $-module triple, under the action $(a f)(s) = a f(s).$
\end{rem}

\begin{rem}\label{Funtri21}  $\mathcal A^{(I)}$ is not meta-tangible when $|I| > 1,$  and the
theory of module triples is quite different from that of
meta-tangible triples, much as module theory differs from the
structure theory of rings.
\end{rem}

\subsubsection{The convolution product: Polynomials and Laurent
polynomials}\label{convol}$ $

If   $(S,+)$ is a monoid, we can define   the \textbf{convolution
product} $f*g$ for $f: S \to \tT $ and $g: S \to \mathcal A$ by
$$(f*g)(s) = \sum _{u+v = s} f(u)g(v),$$ which makes sense in $\mathcal A^{(S)} $
since there are only finitely many $u \in \supp (f)$ and $v \in
\supp (g) $ with $u+v = s.$

\begin{lem}\label{surpp} Under convolution,  $\supp (f*g) \subseteq \supp(f)+ \supp (g)  $.\end{lem}
\begin{proof} $f(u)g(v)\ne \zero$ requires $u \in \supp f$ and  $v \in \supp
g$, which is necessary for $u+v \in \supp (f*g) $.
\end{proof}

\begin{lem}
If $\tT$ is a monoid then $ \tT_{ \mathcal A^{(S)}}$
 also is a monoid under the convolution product.   When $\tT$ is a group,  $
\tT_{\mathcal A ^{(I)}}$ is   a group.
    \end{lem}
\begin{proof}
The product of monomials is a monomial.
\end{proof}

\begin{prop}\label{conv1} If $\{\mathcal A^ , \tT , (-) \}$ is  a semiring-group
triple, then
    $\{\mathcal A^{(S)}, \tT_{\mathcal A^{(S)}}, (-) \}$ is also a semiring-group triple (with the convolution product).
 \end{prop}
\begin{proof}
This  follows directly from the previous two lemmas.
\end{proof}

 The convolution product unifies
polynomial-type constructions.

 \begin{defn}\label{poly1} The \textbf{polynomial system}  $(\mathcal A ^{(\mathbb
 N)},  \tT_{ \mathcal A^{(\mathbb N)}}, (-), \preceq)$ over a system
  $(\mathcal A,  \tT , (-), \preceq)$
 is taken with  $\mathcal A ^{(\mathbb N)}$ endowed with
 the convolution product,    $  \tT_{ \mathcal A^{(\mathbb N)}}$ the set of
 monomials with tangible coefficients, and $(-)$ and $\preceq$ defined componentwise (i.e., according
 to the corresponding monomials). \end{defn}

 One  can iterate this construction to define  $\mathcal A [\la_1, \dots
 \la_n]$ (and then take direct limits to handle an infinite number
 of indeterminates).

 There is a subtlety here which should be addressed. When defining
 the module of monomials $\mathcal A\la$, we could view $\la$ either as a formal indeterminate,
 or as a placemark for a function $f:\la \to \mathcal A$ given by
 choosing $a$ and defining $f: \la \to a.$ These are not the same,
 since different formal polynomials could agree as functions. Our point of
 view is the functional one.

%

Another intriguing example of functor categories (not pursued here)
is the \textbf{exterior product} where the functions in
$\tT^{(\mathbb N)}$ satisfy $f(u)g(v) = (-1)^{uv}g(v)f(u)$ for
$u,v,\in \mathbb N.$

\subsection{The role of universal algebra}\label{unalg}$ $

 As in
\cite[\S 5]{Row16}, universal algebra provides a guide for the
definitions, especially with regard to the roles of possible
multiplication on $\tT$.
 The notions of
universal algebra are particularly appropriate here since we have a
simultaneous double structure, of~$\mathcal A$ and its designated
subset $\tT$  of tangible elements, together with a negation map
$(-)$.

We recall briefly that a \textbf{carrier}, called a
\textbf{universe} in  \cite{MMT}, is a $t$-tuple of sets $\{\mathcal
A_1, \mathcal A_2, \dots, \mathcal A_t\}$ for some given $t$.  A set
of \textbf{operators} is a set $\Omega : = \cup _{m \in \Net} \,
\Omega(m)$, where  each $\Omega(m)$ in turn is a set of formal
$m$-ary symbols $\{ \omega _{m,j} = \omega_{m,j}(x_{1,j}, \dots,
x_{m,j}):   j \in J_{\Omega(m)}$\},
 which are maps
$\omega_{m,j}: \mathcal A_{i_{j,1}} \times \dots \times  \mathcal
A_{i_{j,m}} \to \mathcal A_{i_{j}}$.  
 The 0-ary operators are just
distinguished elements  called \textbf{constants}.

(Here $\mathcal A_1 = \mathcal A$, and $\mathcal A_2 = \tT, $ and
the $\omega_{m,j}$ include the various operations needed to define
$\tT$-modules.  In particular $\zero$ is assumed to be a constant of
$\mathcal A_j$ when $\zero \in \mathcal A_j$.)


 The algebraic structure has \textbf{universal
relations} (otherwise called \textbf{identities} in the literature),
such as associativity, distributivity, which are expressed in terms
of the operators, and this package of the carriers, operators and
universal relations is called the \textbf{signature} of the carrier.
For example, in classical algebra one might take the signature to be
a ring or  a module, endowed with various operations, together with
 identities written as universal relations.

\subsubsection{The category of a variety in universal
algebra}\label{surpre201}$ $

The class of carriers of a given structure is called an
\textbf{algebraic variety}. These are well known to be characterized
by being closed under sub-algebras, homomorphic images, and direct
products.
 To view universal algebra more categorically, we work with a given
 variety (of a given signature). The objects of our category
$\mathcal C$ are the carriers $\mathcal A$ of that signature (which
clearly are sets), and its morphisms are the homomorphisms, which
are maps $f : \mathcal A \to \mathcal A'$ satisfying, for all
operators $\omega_{m,j}: \mathcal A_{i_{j,1}} \times \dots \times
\mathcal A_{i_{j,m}} \to \mathcal A_{i_{m,j}}$.

$$f(\omega_{m,j} (a_{1 }, \dots, a_{m }))= \omega_{m,j} (f(a_{1 }),
\dots, f(a_{m })), \quad \forall a_k \in \mathcal A_{i_{j,k}}.$$

For a system, the signature includes $\mathcal A_2 = \tT,$ usually a
multiplicative monoid, and $\mathcal A_1 = \mathcal A$, a
$\tT$-module. In this philosophy, any homomorphism preserves
constants and tangible elements. For example, if we fix $\tT$ then
the class of $\tT$-modules is an algebraic variety, and there are
many other general examples, cf.~\cite[\S 5.6.1]{Row16}, although as
explained in \cite[\S 5.6.2]{Row16} some major tropical-oriented
axioms do not define varieties.
%
%
%

We can extend the convolution product to morphisms:

 \begin{defn} Given a semiring $T$ and maps $h_1: \mathcal A_1 \to  T
 $ and  $h_2: \mathcal B_1 \to  T
 $, define the \textbf{convolution product}
 $h_1 * h_2 : \mathcal A_1^{(S)}\times \mathcal B_1^{(S)}\to  T $
 by $$((h_1 * h_2)(f_1,f_2))(s) = \sum _{u+v = s} h_1 (f_1(u))h_2
 (f_2(v)).$$
 \end{defn}

\begin{example}\label{trivconv}(The trivial case) When $S$ is the trivial category
consisting of one object $\zero$, and writing $c_i = h_i(\zero),$
then $\mathcal A_i^{(S)} = \mathcal A_i$, and the convolution
product is just given by $h_1 * h_2 (c_1,c_2)= h_1(c_1)h_2(c_2).$
 This
viewpoint will be  useful when we consider tensor
products.\end{example}

\subsubsection{Triples $(-)$-layered by a semiring $L$}$ $

Here is an example paralleling graded algebras, which both relates
to the symmetrized structure and is needed in differential calculus
of $\tT$-systems.

\begin{defn}\label{Func1}
 A pseudo-triple $\mathcal A$ is \textbf{$(-)$-layered} by a monoid $(L,\cdot)$  if
 $$\mathcal A = \cup_{\ell \in L} \mathcal A _\ell; \qquad (-) \mathcal A _\ell =  \mathcal A _{(-)\ell}$$
  where the union of the $A _\ell$ is disjoint, and $\mathcal A _{\ell_1} \mathcal A
  _{\ell_2} \subseteq \mathcal A _{\ell_1 \ell_2}.$

 \end{defn}

\begin{example}\label{lay1}  $(-)$-layering also provides ``layered'' structures, as  described in
\cite[\S2.1]{AGR}. Namely,   a $\tT$-monoid triple layered by a
\semiring0
  $L$ is viewed as a special case of Definition~\ref{Func1}, where all the $\mathcal A _\ell$ are the
  same (taking $(-)$ on $\mathcal A _\ell$ to be the identity map).
This can be viewed as $ \mathcal A_1 ^L$, viewing a monoid $L$ as a
small category, with $\tT_{ \mathcal A} =  \mathcal A_1$. Addition
in the layered triple $\mathcal A' $ of $\mathcal A $ by $L$  is
given in \cite[Example 3.7]{Row16}. This is treated in
\cite[\S2.1]{AGR}, which also considers other subtleties
  concerning layering of triples in general. We also note that there is a $\tT$-homomorphism  $ \mathcal A_1 ^{(L)}\to \mathcal A' $
  given by $(a_\ell) \mapsto (\ell, \sum a_\ell).$
  \end{example}

\section{The structure theory of ground  triples via congruences}\label{homim}$ $

We are ready to embark on the structure theory of ground  triples,
in analogy to the structure theory of rings and integral domains.
Our objective in this section is to modify the ideal theory  (which
does not work for homomorphisms over semirings) and the
corresponding factor-module theory to an analog which is robust
enough to support the structure theory of ground $\tT$-systems. To
do this, we first ignore the issue of negatives, and then, as in
\cite{Row16}, use ``symmetry'' (which is formal negation) instead of
negatives.

\subsection{The role of congruences}\label{congrole}$ $

Unfortunately everything starts to unravel at once. For starters,
\cite[\S 1.6.2]{grandis2013homological} is too restrictive for our
purposes. Factor $\tT_{\mathcal A}$-modules (or factor \semirings0)
are a serious obstacle, since cosets need not be disjoint (this fact
relying on additive cancellation, which fails in the max-plus
algebra, since $1+3 = 2+3 = 3$).

We also have the  following problematic homomorphism, if we want to use
the preimage of $\zero$ in the theory.

\begin{example}
Define the homomorphism $f: \mathcal A \times  \mathcal A   \to
\mathcal A \times \mathcal A$ by $f(a_0,a_1) = (a_0+a_1,a_0+a_1).$
Then $f$ is not 1:1 over the max-plus algebra, but
$f^{-1}(\zero,\zero) = (\zero,\zero).$ This example also blocks the
naive definitions of kernels and cokernels. (In what follows, the
null elements will take the place of the element $\zero$ in
$\mathcal A$.)
\end{example}
This is solved in universal algebra via the use of congruences,
which are equivalence relations respecting the given algebraic
structure.  Congruences can also be viewed as subalgebras of
$\mathcal{A} \times \mathcal{A}$ which are also equivalence
relations.%

\begin{defn} 
Viewing a congruence of  $\mathcal A$ as a substructure of
$\widehat{\mathcal A} $, we define the \textbf{trivial} congruence as follows:
$$\Diag_\mathcal A  := \{ (a,a): a \in \mathcal A \}.$$
\end{defn}

Clearly every congruence contains $\Diag_\mathcal A$. Furthermore,
since we are viewing $(-)$ in the signature, we require that $((-)
a_1 , (-) a_2)\in \Cong$ whenever $(  a_1 ,  a_2)\in \Cong.$




 Since we are incorporating the
switch morphism into the signature, we require for an $\mathcal
A$-module $\mathcal M$ that if $(a,b)\in \mathcal M$ then $(b,a) =
(-)(a,b)\in \mathcal M.$ This is a very mild condition; for example, if
$\zero \in \mathcal A$ then $(b,a) = (\zero,\one)(a,b)
 \in \mathcal M.$

\begin{lem}\label{twistcl}$ $
  Any congruence is   closed under the
twist action.
 \end{lem}
\begin{proof}
Let $\Cong$ be a congruence. By applying (i) again, we obtain that
$$(\zero,a_1) \ctw (x_0,x_1) = (a_1 x_0,a_1 x_1)\in \Cong.$$
But, since $\Cong$ is a congruence (in particular symmetric), we
have that $(a_1 x_1,a_1 x_0)\in \Cong$ and hence the sum $(a_0
x_0+a_1 x_1,a_0 x_1+a_1 x_0) \in \Cong$, yielding the assertion.
\end{proof}

\subsubsection{The twist action on congruences}$ $

In universal algebra, in particular semirings, congruences are more
important to us than ideals. But $\tT_{\mathcal A}$-module
congruences are difficult to work with, since they are
$\tT_{\mathcal A}$-submodules of $\widehat{ \mathcal M} = \mathcal M
\times \mathcal M$ over $\mathcal A.$ The next concept eases this
difficulty by bringing in the twist action as a negation map on
\semirings0.

 We write $\Cong_1\ctw \Cong_2$ for the congruence generated by $
 \{(a_0,a_1)\ctw (b_0,b_1): a_i \in \Cong_1, \ b_i \in \Cong_2\},$
 cf. \eqref{twi}.
From  associativity, it makes sense to write $\Cong^n$ for the twist
product of $n$ copies of $\Cong.$

\subsubsection{$\tT$-congruences}$ $

We restrict our congruences somewhat, to give $\tT$ its proper role.

\begin{defn}  For any congruence $\Cong$, we
  define $\CongT = \{ (a ,a' )\in \Cong : a,a' \in \tTz\}$. We say
that a congruence $\Cong$ is a  $\tT$-\textbf{congruence} if and only if
$\Cong$ is
   additively generated by elements of the form $(a,b)$ where $a,b \in
   \tT$.
\end{defn}


%
%


Clearly  $\Diag_\mathcal A$ is a $\tT$-congruence.

As with rings and modules, there are two notions of factoring out
$\tT$-congruences. If $\Cong$ is a $\tT$-congruence on a triple
$\mathcal A,$ we can form the \textbf{factor triple} $\mathcal
A/\Cong,$ generated by
$$  \tT/\CongT: = \{([a_1],[a_2]): (a_1,a_2)\in \tT \times \tT\},$$ where the equivalence
classes are taken with respect to $\Cong.$

\begin{lem}\label{tcon1}
If $ \CongT$ is a $\tT$-congruence on a triple  $(\mathcal A, \tT,
 (-))$, then $(\mathcal A/\CongT, \tT/\CongT,(-))$ is a triple,
 where one defines
 $(-)[a] = [(-)a] $, and there is a  morphism $\mathcal A \to  \mathcal A/\Cong$ (as $\mathcal A$-module triples)
 given by $a \mapsto [a]$.
\end{lem}
\begin{proof} We need $(-)$ to be well-defined on $\mathcal
A/\CongT,$ which means that if $(\sum a_{1i},\sum a_{2i}) \in
\CongT,$ then $(\sum (-)a_{1i},\sum (-) a_{2i}) \in \CongT.$ But
this is patent.
\end{proof}

 On the other hand, if $\Cong_2\subseteq \Cong_1$ are $\tT$-congruences, we can
define
 the \textbf{factor $\tT$-congruence}
$ \Cong_1/\Cong_2$ generated by $$\{ ([a_1],[a_2]): (a_1,a_2)\in
\Cong_1\}$$ where the equivalence classes are taken with respect to
$\Cong_2.$
 $ \Cong/\Diag_\mathcal A$ is just $ \Cong$. If $\Cong_1$ is a
 $\tT$-congruence then so is $\Cong_1 / \Cong_2$.


\subsection{Prime $\tT$-systems and prime $\tT$-congruences}\label{Pri}$ $

Since classical algebra focuses on algebras over integral domains
(i.e., prime commutative rings), we look for their systemic
generalization.
%
The idea is taken from Jo\'{o} and Mincheva \cite{JoM} as well as
Berkovich~\cite{Ber}. In \cite{Ber}, Berkovich defined a notion of
the prime spectrum $\Spec (\mathcal{A})$ when $\mathcal{A}$ is a
commutative monoid. Berkovich's definition of a prime ideal of a
commutative monoid $\mathcal{A}$ is a congruence relation $\sim$
on~$\mathcal{A}$ such that the monoid $A/\sim$ is nontrivial and
cancellative.

\begin{rem}\label{qc0} For a semiring $\mathcal{A}$, Jo\'{o} and
Mincheva defined a prime congruence for $\mathcal{A}$ as a
congruence~$P$ of $\mathcal{A}$ such that if $a\ctw b \in P$ then $a
\in P$ or $b \in P$ for $a,b \in \widehat{\mathcal{A}}$. This
definition implies the definition of Berkovich in the following
sense: Let $\mathcal{A}$ be an idempotent commutative $\tT$-semiring
triple. Then, for any prime $\tT$-congruence~$P$ (defined as in
\cite{JoM}), the $\tT/P$-semiring triple $\mathcal{A}/P$ is
cancellative. The case when $\mathcal A$ is an idempotent semiring
is proved in \cite[Proposition 2.8]{JoM} although the converse is
not true in general (cf. \cite[Theorem 2.12]{JoM}).
\end{rem}

Let us modify this, to get both directions. We drop the assumption
of commutativity whenever the proofs are essentially the same.





\begin{defn}\label{Pricon}
Let $(\mathcal A, \tT, (-))$ be a $\tT$-\semiring0 triple.
\begin{enumerate} \eroman
\item
A  congruence $\Cong\ne \widehat {\mathcal A}$ of a \semiring0
triple $\mathcal A$ is  \textbf{prime}
 if $\Cong'\ctw \Cong'' \subseteq \Cong$ for congruences implies $\Cong' \subseteq \Cong$
 or $\Cong'' \subseteq \Cong.$
\item
A $\tT$-congruence $\Cong\ne \widehat {\mathcal A}$ of a \semiring0
triple $\mathcal A$ is $\tT$-\textbf{prime}
 if $\Cong'\ctw \Cong'' \subseteq \Cong$ for $\tT$-congruences implies $\Cong' \subseteq \Cong$
 or $\Cong'' \subseteq \Cong.$
A $\tT$-congruence $\Cong$  is  $\tT$-\textbf{semiprime} if
$(\Cong')^2 \subseteq \Cong$ implies $\Cong' \subseteq \Cong$.
Semiprime $\tT$-congruences are called \textbf{radical} when
$\mathcal A$ is commutative.
\item The triple $(\mathcal A, \tT, (-))$ is  \textbf{prime} (resp.~\textbf{semiprime}  if the
 trivial $\tT$-congruence $\Diag_\mathcal A$ is a prime  (resp.~ -semiprime)
 $\tT$-congruence.

\item The triple $(\mathcal A, \tT, (-))$ is $\tT$-\textbf{prime} (resp.~\textbf{semiprime}  if the
 trivial $\tT$-congruence $\Diag_\mathcal A$ is a $\tT$-prime  (resp.~$\tT$-semiprime) $\tT$-congruence.
A commutative semiprime triple is called \textbf{reduced}, in
analogy with the classical theory.

\item The triple $(\mathcal A, \tT, (-))$ is $\tT$-\textbf{irreducible}  if the
 intersection of nontrivial $\tT$-congruences is nontrivial. \end{enumerate}
\end{defn}

For   notational convenience, we  write congruences (instead of
$\tT$-congruences) if the context is clear. Prime congruences arise
naturally as follows (which could be formulated much more
generally):

\begin{prop}\label{primemax} \begin{enumerate} \eroman
\item Given any $\ctw$-multiplicative subset $S$ of $\widehat{\mathcal
A},$ there is a $\tT$-congruence $\Cong$ maximal with respect to
$\Cong \cap S = \emptyset,$ and the $\tT$-congruence $\Cong$ is
$\tT$-prime.

\item Given any $\ctw$-multiplicative subset $S$ of $\widehat{\mathcal
A},$ there is a  congruence $\Cong$ maximal with respect to $\Cong
\cap S = \emptyset,$ and the  congruence $\Cong$ is $\tT$-prime.
\end{enumerate}
\end{prop}
 \begin{proof} (i) The union of a chain of $\tT$-congruences is a
 $\tT$-congruence, so the existence of a maximal such $\tT$-congruence is
 by Zorn's lemma. For the second assertion, suppose that $\Cong'\ctw \Cong'' \subseteq \Cong$, but $\Cong' \not \subseteq \Cong$ and $\Cong'' \not \subseteq \Cong$. Then $S
 \cap \Cong'$ (resp.~$S
 \cap \Cong'$) contains an element $s'$ (resp.~$s''$). It follows that $S
 \cap \Cong' \ctw \Cong''$ contains $s \ctw s'' \in S$. This contradicts to our assumption that $\Cong$ is disjoint from $S$ and hence $\Cong$ is prime.

(ii) Analogous to (i).
 \end{proof}

 \begin{cor}\label{primemax1} If $S$ is a submonoid of $\tT$, then
  there is a congruence $\Cong$ maximal with respect to $$\Cong
\cap ((S\times \{\zero\}) \cup (\{\zero\} \times S) = \emptyset ,$$
and it is prime.
\end{cor}
\begin{proof}
 $(S\times \{\zero\}) \cup (\{\zero\} \times S)$ is a monoid
under $\ctw.$
\end{proof}

We get all prime $\tT$-congruences in this way, because of the
following observation:

\begin{rem} For any  prime  $\tT$-congruence  $\Cong$ of $\mathcal A$, $S: = \mathcal A \setminus \Cong$   is a $\ctw$-submonoid
of $\mathcal A \times \mathcal A,$ and obviously  $\Cong$ is maximal
with respect to $\Cong \cap S = \emptyset .$ \end{rem}

 We say that a
$\tT$-congruence  $\Cong$ is \textbf{maximal} if there is no
$\tT$-congruence  $\Cong'$ with $\CongT \subset \Cong'|_\tT \subset
\tT_{\widehat{\mathcal A }}.$

\begin{lem} The $\tT$-congruence $\Cong$ is maximal if and only if
it is maximal with respect to the condition that $(\one, \zero) \notin \Cong$.
\end{lem}
\begin{proof} If $(\one, \zero) \in \Cong$ then $(b_0,b_1) =
(b_0,b_1)\ctw (\one, \zero) \in \Cong$.
\end{proof}

 \begin{cor}\label{maxisprime} Every maximal $\tT$-congruence $\Cong$ is prime. \end{cor}

\begin{lem}\label{primecrit1}$ $
A triple is prime if and only if it is  semiprime and
$\tT$-irreducible.
\end{lem}
\begin{proof}
$(\Rightarrow)$ Semiprime is a fortiori. But if $\Cong
\cap \Cong' $ is trivial then $\Cong\ctw\Cong' \subseteq \Cong \cap
\Cong' $ is trivial, implying $\Cong$ or $ \Cong'$ is trivial.

$(\Leftarrow)$ If $\Cong\ctw \Cong'  $ is trivial, then $(\Cong \cap
\Cong' )^2$ is trivial, implying $\Cong \cap \Cong' $ is trivial, so
$\Cong$ or $ \Cong'$ is trivial.
\end{proof}
\medskip

 The following assertions are straightforward.
\begin{lem}$ $
\begin{enumerate} \eroman
\item For a $\tT$-congruence $\Cong'$,  $\CongT' \subseteq \Cong$ if and only if $\Cong' \subseteq \Cong$.

\item The intersection of  prime  $\tT$-congruences is  semiprime.

\end{enumerate}
\end{lem}

Jo\'{o} and Mincheva \cite{JoM} showed that any irreducible,
cancellative commutative $\mathbb B$-algebra $\mathcal A$ is prime,
and it follows that the polynomial system $ \mathcal A[\Lambda]$ is
prime.

We say that $\mathcal A$ satisfies the
 ACC on $\tT$-congruences if for every ascending chain  $\{\Cong_i : i \in I\}$  there is some $i$ such that ${\Cong}_i = {\Cong}_{i'}$ for all $i'>i.$

\begin{prop}\label{prime1}
 If $\mathcal A$ satisfies the
ACC on $\tT$-congruences, then every $\tT$-congruence $\Cong$
contains a finite product of prime $\tT$-congruences, and in
particular there is a finite set of prime $\tT$-congruences whose
product is trivial.
 \end{prop}
 \begin{proof} A standard argument on Noetherian induction: We take a maximal
 counterexample $\Cong$. If $\Cong$ is not
 already prime, then there are two
 $\tT$-congruences $\Cong',$ $\Cong''$ whose intersection with $\tT$ is not in  $\Cong$,
  but whose product is in  $\Cong$. By
 Noetherian induction applied to $\Cong/\Cong'$ in $\mathcal
 A/\Cong'$, and to $\Cong/\Cong''$ in $\mathcal
 A/\Cong''$, we get finite sets of  prime
$\tT$-congruences whose respective products are in $\Cong'$ and
$\Cong''$, so the product of all of them is in $\Cong$.
\end{proof}

%

 For $\mathcal A$ commutative with a $\tT$-congruence $\Cong$, define $\sqrt{\Cong}$ to be the $\tT$-congruence   generated by the set
 \[ \{(a_0, a_1)\in \widehat{\tT}:
  (a_0, a_1)^n \in \Cong \textrm{ for some }n\}.\]

\begin{lem}\label{primecrit}$ $ In a commutative $\tT$-semiring system,
\begin{enumerate} \eroman
\item  $\Cong$
is prime if and only if it satisfies the condition that $(a_0, a_1)\ctw (b_0,
b_1) \in \Cong$ implies $(a_0, a_1) \in \Cong$ or $(b_0, b_1)\in
\Cong$ for $(a_0, a_1), (b_0, b_1)\in \widehat{\tT}$.
\item  $\Cong$
is $\tT$-radical if and only if $(a_0, a_1)^2 \in \Cong$ implies $(a_0, a_1)
\in \Cong$ for $(a_0, a_1)\in \widehat{\tT}$.

\item If $\Cong$ is a  $\tT$-congruence, then $\sqrt{\Cong}$ is a radical
 $\tT$-congruence.
\end{enumerate}
\end{lem}

%
%


 \begin{prop}\label{prime2}
 For every $\tT$-congruence $\Cong$ on a commutative $\tT$-semiring system,  $\sqrt{\Cong}$
  is an intersection  of prime $\tT$-congruences.
\end{prop}
 \begin{proof}
 For any given $\bold a = (a_0, a_1)\in \widehat{\tT},$
 let $S_\bold a = \{ (a_0, a_1)^n: n \in \Net\}.$ If $S_\bold a \cap \Cong
=\emptyset,$   Zorn's lemma
 gives us a   $\tT$-congruence containing  $\sqrt{\Cong}$, maximal with respect to being disjoint
 from $S_\bold a $, and easily seen to be prime, so their intersection is
 disjoint from all such $\bold a $, which is precisely
 $\sqrt{\Cong}$.
\end{proof}


\subsection{Annihilators,  maximal $\tT$-congruences, and simple $\tT$-systems}$ $

\begin{defn} Suppose $\mathcal M$ is  a $\tT$-module system  over a $\tT$-\semiring0 system $\mathcal
A$. For any $S \subseteq \mathcal M$,   the \textbf{annihilator}
$\Ann _{\tT} (S)$ is the $\tT$-congruence generated by $\{ (a_0,
a_1) \in \widehat{\tTz} : a_0s = a_1 s, \ \forall s_i \in S \}.$

Likewise, suppose $\Cong$ is a $\tT$-congruence on the $\tT$-module
system $\mathcal M$. For  $S \subseteq \Cong$, define the
\textbf{annihilator} $\Ann _{\tT}
 (S)$ to be the
$\tT$-congruence generated by $\{ (a_0, a_1) \in \widehat{\tT} :
a_0s_0 +a_1 s_1 = a_0s_1 +a_1 s_0, \ \forall s \in S \}.$
\end{defn}

%
In other words, $\Ann _{\tT} (S) S$ is trivial, under the
twist multiplication. 

\begin{defn}\label{simplemod} An $\mathcal A$-module pseudo-system $\mathcal M$  is  \textbf{simple}
 if  proper sub-pseudo-systems are all in $\mathcal M_\Null$.
\end{defn}

As in classical algebra, these are the building blocks in the
structure theory of triples  that are  monoid modules.

\begin{rem}  For any $s \in \tT_{\mathcal M},$ the map $b \mapsto bs$ induces an isomorphism
$\mathcal A/\Ann _{\tT}
 (s) \cong \mathcal A s$. Consequently, $\Ann _{\tT} (s)$ is a maximal
 congruence iff the pseudo-system $(\mathcal A s, \tT s, (-), \preceq)$ is simple.
 \end{rem}

\begin{prop} If $\mathcal M$ is  a simple $\tT$-module pseudo-triple,
then $\Ann _{\tT} (\mathcal M)$ is a $\tT$-prime $\tT$-congruence
of~$\mathcal A$.
\end{prop}
\begin{proof} If $\Cong \Cong' \subseteq \Ann _{\tT} (\mathcal
M)$ for $\tT$-congruences $\Cong, \Cong' ,$ then $\Cong (\Cong'
\mathcal M)$ is trivial. Thus either $\Cong' \subseteq \Ann
_{\mathcal A} (\mathcal M)$ or $\Cong' \mathcal M = \mathcal M,$ in
which case $\Cong  \mathcal M$ is trivial and hence $\Cong \subseteq
\Ann _{\mathcal A} (\mathcal M)$.
\end{proof}

\begin{prop} For ${\mathcal A} $ commutative, $\mathcal M$ is  a simple $\tT$-module system if and only if
$\Ann _{\tT} (\mathcal M)$ is a maximal $\tT$-congruence of
$\mathcal A$.
\end{prop}
\begin{proof} As in the usual commutative theory,  $\Ann _{\tT} (\mathcal M) =  \Ann _{\tT} (\{ s
\})$ for any nonzero $s \in \mathcal M$, since $\mathcal A s
=\mathcal M.$
\end{proof}

 In the noncommutative case, one could go on to define primitive
$\tT$-congruences of a ground $\tT$-system to be the annihilators of
simple module systems, but that is outside the scope of this work.

\section{The geometry  of    prime systems}\label{geoprim}

As in classical algebra, the ``prime'' ground systems  lend
themselves to affine geometry.

\subsection{Primeness of  the polynomial system $ \mathcal A[\la]$}$ $

Let us consider when  $ \mathcal A^S$ is prime. We need a slightly
different definition of ``trivial.''

\begin{defn}\label{ft}
Let \[
 \mathcal{A}^{(S)}_\Cong:=\{(f,g)\in \mathcal A^{(S)} \times \mathcal
A^{(S)} \mid (f(s),g(s))\in \Cong,~\forall s \in S\}.
\]
 $\mathcal{A}^{ S}_\Cong$ is defined analogously, using
 $\mathcal{A}^{ S }$ instead of $\mathcal{A}^{(S)}.$
\item
An element $(f,g) \in \mathcal{A}^S_\Cong$ is \textbf{functionally
trivial} if $(f (s),g(s)) \in \Diag_{\mathcal A}$, $\forall s \in
S$.
\item
A $\tT$-congruence $\Cong$ is \textbf{functionally trivial} if any
pair $(f,g) \in \mathcal{A}^S_\Cong$ is functionally trivial.
\end{defn}

\begin{prop}
Let $\Cong$ be a $\tT$-congruence on $\mathcal{A}$. If $\Cong$ is
radical then so are the congruences $\mathcal A^{(S)}_\Cong$
on~$\mathcal A^{(S)}$, and $\mathcal{A}^{S}_\Cong$ on
$\mathcal{A}^{S}$, for any small category $S$, and in particular so
is $\mathcal{A}[\la_1, \dots, \la_n]_\Cong$.
\end{prop}
\begin{proof}
Let $(f,g) \in \mathcal A^{(S)}_\Cong$ and suppose that
$(f,g)^2=(f^2,g^2) \in \mathcal A^{(S)}_\Cong$ , i.e., one has
\[
(f^2 (s),g^2(s)) \in \Cong^2,~\forall s \in S.
\]
It follows that $(f(s),g(s)) \in \Cong$ since $\Cong$ is radical,
and hence $(f,g) \in \mathcal A^{(S)}_\Cong$. This proves that
$\mathcal A^{(S)}_\Cong$ is radical. The case of
$\mathcal{A}^{S}_\Cong$ is analogous.
\end{proof}

On the other hand, the other ingredient, irreducibility, is harder
to attain. Given a $\tT$-congruence $\Cong$ on $\mathcal A^{(S)}$
(in particular, $\Cong$ is a subset of $\mathcal A^{(S)} \times
\mathcal A^{(S)}$), we define
\[
S_\Cong := \{ s \in S : (f(s),g(s)) \not \in \textrm{Diag}_\mathcal
A,~\forall (f,g) \in \Cong\}.
\]
This
leads to a kind of consideration of density.

\begin{lem}\label{Zartyp}
Suppose that $\mathcal{A}$ has  $\tT$-congruences $\Cong, \Cong '$
with $\Cong \cap \Cong '$ functionally trivial. If $\mathcal{A}$ is
irreducible, then $S_{\mathcal A^{(S)}_\Cong} \cap S_{\mathcal
A^{(S)}_{\Cong'}} = \emptyset.$
\end{lem}
 \begin{proof}
For each $s \in S$, since $\Cong \cap \Cong'$ is functionally
trivial, either $s \not \in S_{\mathcal A^{(S)}_\Cong}$ or $s \not
\in S_{\mathcal A^{(S)}_{\Cong'}}$. This implies that $S_{\mathcal
A^{(S)}_\Cong} \cap S_{\mathcal A^{(S)}_{\Cong'}} = \emptyset.$
\end{proof}

It is well-known by means of a Vandermonde determinant argument that
over an integral domain, any nonzero polynomial of degree~$n$ cannot
have $n+1$ distinct zeros. The analog for semirings also holds for
triples, using ideas from \cite{AGG1}. Namely, we recall
\cite[Definition~6.20]{Row16}:

 \begin{defn}\label{signeddet} Suppose $\mathcal A$ has a negation map
 $(-)$.
 For a permutation $\pi$, write
 $$(-)^\pi a = \begin{cases} a: \pi \text{ even}; \\ (-)a:   \pi \text{ odd}.\end{cases}$$

\begin{enumerate} \eroman
\item The \textbf{$(-)$-determinant} $\absl A $ of a matrix $ A $ is
 \begin{equation*}\label{eq:tropicalDetsign}
  \sum_{\pi \in  S_n}  (-)^{\pi} \left( \prod_i a_{i,\pi(i)}\right).
 \end{equation*}

\item   The $n \times n$ \textbf{Vandermonde matrix} $V(a_1, \dots, a_n)$
  is defined to be $\(\begin{matrix}
               \one & a_1 & a_1^2 & \dots & a_1^{n-1}\\

               \one & a_2 & a_2^2 & \dots & a_2^{n-1}\\

               \vdots & \vdots & \vdots & \ddots & \vdots\\

               \one & a_n & a_n^2 & \dots & a_n^{n-1}
               \end{matrix}\).$

\item Write $a_{i,j}'$ for the $(-)$-determinant of the $j,i$ minor of
a matrix $A$. The \textbf{$(-)$-adjoint} matrix $\adj A$ is
$(a_{i,j}')$.
\end{enumerate}
\end{defn}

We have the following adjoint formula from \cite[Theorem~1.57]{AGR}.
\begin{lem}\label{Vand}$ $
 \begin{enumerate} \eroman \item  $|A| = \sum
_{j=1}^n (-)^{i+j} \, a'_{i,j}a_{i,j},$ for any given $i$.
\item
$\absl{V(a_1, \dots, a_n)}  = \prod _{i>j}(a_j (-)a_i).$
\end{enumerate}\end{lem}
 \begin{proof} This is well-known for rings, so is an application of
 the transfer principle in \cite{AGG1}. (Put another way, one could
 view the $a_i$ as indeterminates, so the assertion holds formally.)
 \end{proof}

We say $b$ is a \textbf{$\circ$-root} of $f \in \tT[\la]$ if
$f(b)\in \mathcal A^\circ,$ i.e., $f(b) \succeq_\circ \zero.$ This
definition  was the underlying approach to supertropical affine
varieties in \cite{IR}.

 \begin{thm}\label{polyroot} Over a commutative prime triple
$(\mathcal A, \tT, (-))$ with unique negation, any nonzero
polynomial $f \in \tT[\la]$ of degree~$n$ cannot have $n+1$ distinct
$\circ$-roots in $\tT$.
\end{thm}
 \begin{proof} Write $f = \sum _{i=0}^n b _i \la^i$ for $b_i \in
 \tT$. Suppose on the contrary that $a_1, \dots, a_{n+1}$ are distinct $\circ$-roots.
 Write $v $ for the column vector $(a_0, \dots, a_{n}).$ Then $Av$ is
 the column vector $(f(a_1), \dots, f(a_n))$ which is a quasi-zero, so
 $$\prod _{i>j}(a_j (-)a_i)v  = \absl A v = \adj A A v \in \adj A \mathcal A^\circ \preceq \mathcal
 A^\circ,$$
implying $\prod _{i>j}(a_j (-)a_i)\in \mathcal
 A^\circ,$ contrary to $a_1, \dots, a_{n+1}$ being distinct.
 \end{proof}

 \begin{cor}\label{poly11}
 If $(\mathcal A,\tT,(-))$ is a  prime commutative
 triple with $\tT$ infinite, then so is  $(\mathcal A[\la],\tT_{\mathcal A[\la]},(-)).$
 \end{cor}
\begin{proof}
Follows from Lemma~\ref{Zartyp}, since any finite set of functions cannot
have infinitely many common roots.
\end{proof}

\begin{example}\label{symadj} We will need the congruence version of Definition~\ref{signeddet}, for
which we turn to symmetrization, as treated in \cite{AGR}. Namely,
according to Definition~\ref{symr1}, we embed  $( \mathcal A, \tT)$
into the symmetrized system $(\widehat {\mathcal
        A},\widehat {\tT},(-))$    with
    multiplication $\widehat {\tT} \times \widehat {\mathcal A}\to
    \widehat {\mathcal A}$ given by the twist action, $(-)$ is the switch
    map, and $\preceq$ is~$\precpr.$

Given $(f(\la),g(\la))
        \in \widehat{\tT}[\la]$, we define  $$(f ,g)(b_0,b_1) =
        (f(b_0) +g(b_1), f(b_1)+g(b_0)),$$
        for any element $(b_0,b_1)$ of $\widehat {\mathcal
        A}$. The element $(b_0,b_1)$
        is a \textbf{symmetrized root} of $(f(\la),g(\la))
        \in \widehat{\tT}[\la]$ if
$(f ,g)(b_0,b_1)\in \widehat {\mathcal A}^\circ,$ evaluated under
twist multiplication. (In particular, for $b_1 = \zero$, this means
$f(b_0) = g(b_0).$)

Now one can define the \textbf{symmetrized determinant} to be the
$(-)$-determinant in this sense, and of the $j,i$ minor of a matrix
$A$. The \textbf{symmetrized adjoint matrix} is  the $(-)$-adjoint
matrix.

 In this context, Theorem~\ref{polyroot} says that over a commutative
 prime triple
$\mathcal A$, any pair of polynomials $(f,g)\in  \widehat{\tT[\la]}$
of degree~$n$ cannot have $n+1$ distinct symmetrized roots in
$\widehat{\tT}$.
\end{example}


\subsection{Localization}\label{loc}$ $

 We refer to \cite[\S6.8]{Row16}, where we  used Bourbaki's  standard technique of localization~\cite{Bo}, to pass
 from commutative
metatangible  $\tT$-\semiring0 systems (resp.~$\tT$-monoid module
triples) to  metatangible $\tT$-systems (resp.~$\tT$- triples) over
groups. We assume that $\tT$ is a monoid and  $S$ is a central
submonoid of $\mathcal A\setminus \mathcal A_{\Null} $ (i.e., $sa =
as$ for all $a\in \mathcal A, \ s \in S$). Often $S \subseteq \tT.$
Recall that one defines the equivalence $(s_1,b_1) \equiv (s_2,b_2)$
when $s(s_1b_2) = s(s_2 b_1)$ for some $s\in S$, and we write
$s^{-1}b$ or $\frac bs$ for the equivalence class of $(s,b)$. We
might as well assume that $\one \in S$ since $\frac b \one = \frac
{bs}s.$ We localize a $\tT$-semiring triple $( \mathcal A, \tT, (-)
)$ with respect to $S$ by imposing multiplication:
\[
(s_1^{-1}b_0) (s_2^{-1}b_0') = (s_1s_2)^{-1}b_0b_1,
\]
and  addition:
\[
(s_1^{-1} b_0) + (s_2^{-1}b_0') = (s_1s_2)^{-1}( s_2b_0+
s_1b_0').
\]

%

The standard ring-theoretic facts are mirrored in the systemic
situation.

\begin{rem}\label{5.13}
\item Any finite set of fractions $\frac {a_1}{s_1}, \frac
{a_2}{s_2}, \dots ,\frac {a_n}{s_n}$ has a common denominator $s =
s_1\cdots s_n,$ since
$$\frac
{a_i}{s_i} = \frac {s_1\cdots s_{i-1}a_is_{i+1}\cdots s_n}s.$$
 \end{rem}

  We say that $a\in \tT$ is $\preceq$-\textbf{regular} if
$ab_1 \preceq ab_2$ implies $b_1 \preceq b_2$. We say that $a\in S$
is \textbf{regular} if $ab_1 = ab_2$ implies $b_1 = b_2$, and $S  $
is \textbf{regular} if each of its elements is regular.

\begin{lem} When $\preceq$ is a PO, then $\preceq$-regular implies
regular.
 \end{lem}
\begin{proof} $ab_1 = ab_2$ implies $ab_1 \preceq ab_2$ and $ab_1 \succeq
ab_2$, so $ b_1 \preceq  b_2$ and $ b_1 \succeq  b_2$, implying $
b_1=  b_2$

\end{proof}

\begin{prop}\label{local} Let   $ (\mathcal A, \tT, (-))$ be a  pseudo-triple  with unique quasi-negatives, and
$S$ be a  multiplicative
submonoid
 of~$\mathcal A$. Then the following hold.
 \begin{enumerate}\eroman
    \item
    $(S^{-1}\mathcal A,\, S^{-1}\tT
    , (-))$ is a pseudo-triple with unique quasi-negatives.
\item
If   $ (\mathcal A, \tT, (-))$ is a  $\tT$-triple  with unique
quasi-negatives, then $(S^{-1}\mathcal A,\, S^{-1}\tT , (-))$ is
also a $\tT$-triple which has unique quasi-negatives.
 \item
 There is a canonical homomorphism
 \[
 S^{-1}:(\mathcal
 A, \tT, (-))\longrightarrow (S^{-1}\mathcal A,\, S^{-1}\tT , (-)), \quad b\mapsto \frac b\one,
 \]
 whose congruence kernel is the following $\tT$-congruence $$\Cong = \{
 (b_0,b_1): sb_0 = sb_1 \text{ for some }s\in S\}.$$
 \item
 The map $S^{-1}$ induces
an isomorphism $S^{-1}\mathcal A/ S^{-1}\Cong  \cong  \frac{S}{
\one}^{-1}(\mathcal A/\Cong).$
 \item If $S$ is regular then the map of (iii) is an injection.
 \item If $\tT$ is regular then  $ (\mathcal A, \tT, (-))$ injects into the triple $(\tT^{-1}\mathcal A,\, \tT^{-1}\tT , (-))$
 over the group $ \tT^{-1}\tT$.
 \end{enumerate}
 \end{prop}
\begin{proof}
(i): One can easily check that $(S^{-1}\mathcal A,\, S^{-1}\tT ,
(-))$ is a pseudo-triple. This is standard (where $(-)(s^{-1}a ): =
s^{-1}( (-) a), \quad s\in S)$). For the assertion about unique
quasi-negatives, suppose $s_1^{-1}a_1 $ is a quasi-negative of
$s^{-1}a.$ Then $$ (s s_1)^{-1}(sa_1 + s_1a) = s_1^{-1}a_1 +
 s^{-1}a \in \mathcal (S^{-1}A)^\circ,$$ implying $sa_1 + s_1a \in \mathcal
 A^\circ,$ and thus $sa_1 = (-) s_1a = s_1((-)a),$ and $s_1^{-1}a_1
 = s^{-1}((-)a).$ 

(ii): The proof is essentially same as $(i)$. We only have to check
that $S^{-1}\tT \cap (S^{-1}\mathcal A)^\circ =\emptyset$. Suppose
that $\frac{b}{s} \in S^{-1}\tT \cap (S^{-1}\mathcal A)^\circ$. In
particular, $\frac{b}{s}=\frac{a}{s_1}(-) \frac{a}{s_1}=\frac{s_1a
(-) s_1a}{s_1^2}$ for some $\frac{a}{s_1} \in S^{-1}\mathcal A$. It
follows that $s's_1a (-) s's_1a \in \tT$ for some $s' \in S$.
However, this implies that $\tT \cap \mathcal A^\circ \neq
\emptyset$, which contradicts to the assumption that $(\mathcal
A,\tT,(-))$ is a $\tT$-triple.

 (iii): Clearly $S^{-1}$
is a homomorphism and $\Cong $ is the congruence kernel of $S^{-1}$.

(iv): The kernel of the map $S^{-1}$ is $S^{-1}\Cong .$

(v):  $\Cong $ is trivial, by (iii) and the definition of regular.

(vi): $\tT^{-1}\tT $ is a group.
\end{proof}

\begin{lem} Any surpassing relation $\preceq$ on a monoid system $( \mathcal A,  \tT, (-),\preceq)$
extends  to  $(S^{-1}\mathcal A, S^{-1}\tT, (-))$, by putting
$s^{-1}b \preceq (s')^{-1}b'$ whenever $s'b\preceq sb'.$
 \end{lem}
\begin{proof}  We verify the conditions of Definition~\ref{precedeq07}.

$(i)$: Suppose that $\frac{b}{s},$ and $\frac c{s'} \in S^{-1}\tT$
 Then  $s'b \preceq s'b+ sc^\circ $, implying
$$\frac {b}{s} = \frac {s'b}{ss'}\preceq \frac {s'b+ sc^\circ }{ss'}= \frac{b}{s} + \frac c{s'}^\circ
.$$ 
%

$(ii)$: Suppose that $\frac{b}{s} \preceq \frac{b'}{s'}$. Then we have $s'b \preceq sb'$ and hence $(-)s'b \preceq (-)sb'$, showing that
\[
\frac{(-)b}{s}=(-)\frac{b}{s} \preceq (-)\frac{b'}{s'}=\frac{(-)b'}{s'}.
\]

$(iii)$: Suppose that $\frac{b_1}{s_1} \preceq \frac{b_1'}{s_1'}$
and $\frac{b_2}{s_2} \preceq \frac{b_2'}{s_2'}$. Then $ b_1s_1'
\preceq s_1b_1'$ and $ b_2s_2'\preceq s_2b_2'. $ It follows that
\[
s_1's_2' (b_1s_2+s_1b_2)=s_2'(s_1'b_1)s_2+s_1'(s_2'b_2)s_1 \preceq s_2's_1b_1's_2 +s_1's_2b_2's_1 = (s_1s_2)(s_2'b_1'+s_1'b_2'),
\]
which shows that
\[
\frac{b_1}{s_1}+\frac{b_2}{s_2}=\frac{s_2b_1+s_1b_2}{s_1s_2} \preceq \frac{s_2'b_1'+s_1'b_2'}{s_1's_2'}=\frac{b_1'}{s_1'}+\frac{b_2'}{s_2'}.
\]

$(iv)$: Suppose that $\frac{b }{s } \preceq \frac{b'}{s'}$. Then we
have that $bs' \preceq sb'$. In particular, since $S \subseteq \tT$,
\[
a s_2(s_1'b_1) \preceq a s_2(s_1b_1').
\]
for any $a \in \tT$ and $s_2 \in S$. It follows that for any
$\frac{a}{s_2} \in S^{-1}\tT$, we have
\[
\frac{a}{s_2}\frac{b_1}{s_1}=\frac{ab_1}{s_2s_1} \preceq
\frac{a}{s_2}\frac{b_1'}{s_1'}=\frac{ab_1'}{s_2s_1'}.
\]

$(v)$: Suppose that $\frac{a_1}{s_1} \preceq \frac{a_2}{s_2}$ for $\frac{a_i}{s_i} \in S^{-1}\tT$. This implies that $s_2a_1 \preceq s_1a_2$, however $s_2a_1, s_1a_2 \in \tT$ and hence $s_2a_1=s_1a_2$, showing that $\frac{a_1}{s_1}=\frac{a_2}{s_2}$.

\end{proof}

\subsubsection{Localization of $\tT$-congruences}$ $

Next, we introduce localization for $\tT$-congruences. Again we take
$S $ a submonoid of $\tT$, but the flavor is different. Let
$(b_0,b_1)$ and $(b_0',b_1')$ be elements of a $\tT$-congruence
$\Cong$. One defines the following equivalence
\[
(b_0,b_1) \equiv (b_0',b_1') \textrm{ if and only if } s(b_0,b_1) =
s'(b_0',b_1') \textrm{ for some }s,s'\in S.
\]


\begin{defn}   $S^{-1}\Cong = \{ (\frac {b_0}s , \frac {b_1}s ): (b_0 , b_1)\in \Cong\}.$
 \end{defn}
\begin{rem} If $\Cong$ is a
 $\tT$-congruence of $( \mathcal A, \tT, (-)),$ then $S^{-1}\Cong$ is a
 $\tT$-congruence of  $(S^{-1}\mathcal A,
S^{-1}\tT, (-))$.
 \end{rem}

\begin{lem} Any $\tT$-congruence of $(S^{-1}\mathcal A,
S^{-1}\tT, (-))$ has the form $S^{-1}\Cong$, where $\Cong$ is a
 $\tT$-congruence of $( \mathcal A, \tT, (-)).$
 \end{lem}
\begin{proof} Given the $\tT$-congruence $\Cong'$ of $S^{-1}\mathcal A,$
define $$\Cong = \{ (b_0, b_1) \in \mathcal A :  (\frac{b_0}\one,
\frac{b_1}\one)\in \Cong'\}.$$ If $(b_0', b_1') \in \Cong'$ then
writing $b_i' = \frac {b_i}{s_i'}$ for $b_i \in \mathcal A$ we have
$s_0's_1'(b_0', b_1') = (s_1' b_0, s_0' b_1)\in \Cong,$ so $\Cong' =
S^{-1}\Cong$.
\end{proof}

 From now on we
assume for convenience  that $S \subseteq \tT$ is a submonoid of
$\Cong$-\textbf{regular} elements, in the sense that $ (sb_0,sb_1)
\in \Cong$ implies $ ( b_0, b_1) \in \Cong$ for any $s \in S$.

\begin{prop}\label{local1} If $\Cong$ is a prime $\tT$-congruence of $( \mathcal A,  \tT, (-))$,
then $S^{-1}\Cong$ is a  prime $\tT$-congruence of $(S^{-1}\mathcal
A, S^{-1}\tT, (-)).$
\end{prop}
\begin{proof}  Suppose  $S^{-1} \Cong' \ctw S^{-1} \Cong '' \subseteq  S^{-1} \Cong.$ Then clearly
$\Cong'\ctw \Cong'' \subseteq \Cong$ by regularity, so $\Cong'
\subseteq \Cong$
 or $\Cong'' \subseteq \Cong,$ implying $S^{-1} \Cong' \subseteq S^{-1} \Cong$
 or $S^{-1} \Cong'' \subseteq S^{-1} \Cong.$
\end{proof}

\begin{example} Take $S = \{ s \in \tT: (s,\zero) \notin \Cong\}.$ If
$ \Cong$ is prime then $S$ is $ \Cong$-regular. Hence $S^{-1}\Cong$
is maximal with respect to being disjoint from $S^{-1}\tT,$ in the
sense of Corollary~\ref{primemax1}.
\end{example}

 \subsection{Extensions of systems}\label{extsys1}$ $

 \begin{defn}\label{ext1} When we have a homomorphism
    $\varphi: (\mathcal A, \tT_{\mathcal A} , (-)) \to  (\mathcal A', \tT_{\mathcal A '}, (-))$
of pseudo-triples   whose congruence kernel is trivial, we say that
$(\mathcal A', \tT_{\mathcal A '}, (-))$ is an \textbf{extension} of
$(\mathcal A, \tT_{\mathcal A }, (-)).$ In other words, $(\mathcal
A, \tT_\mathcal A , (-))$ can be viewed as a sub-triple of
$(\mathcal A', \tT_{\mathcal A '}, (-)).$
\end{defn}

As in the classical theory, for $a_i \in  \tT_{\mathcal A '},$ $i
\in I,$ we write $\mathcal A[a_i: i \in I]$ for the sub-triple of
$\mathcal A '$ generated by the $a_i$ (where $\tT_{\mathcal A[a_i: i
\in I]}$ is the set of monomials in the $a_i$ with coefficients in
$\tT$).

For $\mathcal B,\mathcal B ' \subseteq \mathcal A',$  we write
$\mathcal B \preceq \mathcal B'$, and say $\mathcal B'$
$\preceq$-generates $\mathcal B$, if for each $b\in \mathcal B$
there is $b' \in \mathcal B'$ such that $b \preceq b'.$

\begin{defn}
We say that $\mathcal A '$ is $\preceq$-\textbf{affine over}
    $\mathcal A $ if $\tT '\preceq \mathcal \tT[a_i: i \in I]$
    for $I$ finite. (In other words, taking $I = \{1, \dots, n\},$ we
    write $\mathcal A ' = \mathcal A [a_1 , \dots, a_n ]_\preceq$, where for any
    $a'\in \tT '$ there is $f(\la_1, \dots, \la_n)\in  \mathcal
    \tT[a_1, \dots, a_n]$ such that $ a' \preceq f(a_1, \dots, a_n).)$

    For the symmetrized case, we say that $\mathcal A '$ is $\preceq$-\textbf{affine over}
    $\mathcal A $ if $\mathcal A '\preceq \mathcal \tT[(a_i,\zero): i \in I]$
    for $I$ finite.
\end{defn}

%
\begin{rem}
For any extension $(\mathcal A', \tT_{\mathcal A '}, (-))$ of
$(\mathcal A, \tT_{\mathcal A }, (-)),$ and $a_i \in {\mathcal A},$
There is a natural homomorphism $\mathcal A[\la_i: i \in I] \to
\mathcal A'$ given by $\la _i \mapsto a_i$, whose kernel is prime if
and only if $\mathcal A[a_i: i \in I]$ is prime.
\end{rem}


\begin{rem} Given a prime system $(\mathcal A , \tT, (-), \preceq) $, we form
$\mathcal A [\Lambda]$, where $\Lambda = \{ \lambda_i :i \in I  \}$
and some subset $\Lambda' = \{ \lambda_i :i \in I' \}$, and we take
$S$ to be the submonoid of monomials in $\{\lambda_i :i \in I' \}$,
and $\tT' $ to be the submonoid of monomials in $\{\lambda_i :i \in
I'\}  $. Then $S ^{-1}\tT'$ is a group, and we have the prime system
$S ^{-1}\Lambda = (S ^{-1}\mathcal A  [\Lambda], S
^{-1}\tT,(-),\preceq )$. In this way we ``expand'' $\tT$ to $S
^{-1}\tT'$ and lower the number of indeterminates under
consideration.
\end{rem}

\begin{lem}\label{algaff} If   $\mathcal A ''$ is  $\preceq$-affine over  $\mathcal A' $ and $\mathcal A '$ is  $\preceq$-affine over  $\mathcal A
$, then   $\mathcal A ''$ is  $\preceq$-affine over  $\mathcal A $.
\end{lem}
\begin{proof} Write $\mathcal A '' =  \mathcal A '[a_1', \dots, a_m']_\preceq$ and
$\mathcal A ' = \mathcal A [a_1 , \dots, a_n ]_\preceq$. Then
writing $a'' \preceq f(a_1', \dots, a_n')$ and $   a_j'\preceq
g_j(a_1 , \dots, a_n )$, we have $ a'' \preceq f(g_1(a_1 , \dots,
a_n ), \dots, g_m(a_1 , \dots, a_n ) ).$
\end{proof}
\subsubsection{Weak nullstellensatz}\label{weakn}$ $

Having introduced affine systems, we would like to develop
techniques to analyze them, and present some results related to the Nullstellensatz. To this end, we  need an observation about spanning sets. \begin{defn}
 Elements $\{v_i: i \in I\}$ of an $\mathcal A $-module $\mathcal M $
  $\preceq$-\textbf{span}
 a submodule $\mathcal N$ if   $\mathcal N \preceq \sum \tT v_i.$

 Elements $\{v_i: i \in I\}$
  are $\preceq$-\textbf{independent over}
 a submodule $\mathcal N$ if $\sum b_i v_i \in \mathcal N_\Null$ implies each $b_i
 \in \mathcal A_\Null$.

 A $\preceq$-\textbf{base} is a  $\preceq$-independent
 $\preceq$-spanning set.

 A  \textbf{symmetric base} is a $\preceq$-base for the symmetrized
 module triple.
 \end{defn}

Even though dependence is not transitive for modules over semirings,
we do have the following result.

\begin{lem}\label{key111} If the
extension  $(\mathcal A'', \tT_\mathcal A '', (-))$ of $(\mathcal
A', \tT_{\mathcal A'} , (-)) $ has symmetric base $b''_1, \dots,
b''_{m''}$ and the extension $(\mathcal A', \tT_\mathcal A ', (-))$
of $(\mathcal A, \tT_\mathcal A , (-)) $  has symmetric base $b'_1,
\dots, b'_{m'}$, then  the extension $(\mathcal A'', \tT_\mathcal A
'', (-))$ over $(\mathcal A, \tT_\mathcal A , (-))  $ has symmetric
base   $\{ b_i' b_j'': 1 \le i \le m', 1\le j \le m''\}$.
\end{lem}
\begin{proof} Any $b'' \in \mathcal A''$ satisfies $b'' \preceq
\sum_j a_j' b_j'',$ and $a_j' \preceq \sum_i a_{i,j}  b_i' ,$
implying $b'' \preceq \sum_{i,j} a_{i,j}  b_i' b_j'',$ proving
$\preceq$-spanning. For $\preceq$-independence, suppose
$$\sum_{i,j} a_{i,j} b_i' b_j'' =  \sum_j(\sum _i a_{i,j} b_i')
b_j'' \in \mathcal A''_\Null.$$ Then each $(\sum _i a_{i,j} b_i')\in
\mathcal A'_\Null,$ implying each $ a_{i,j}\in  A _\Null.$
\end{proof}

\begin{lem}\label{5.14}  Suppose $\mathcal A'$ is a  module over a
system $\mathcal A $, and has a $\preceq$-base $B = \{ v_i : i \in I
\}$ over $\tT.$ If $(\mathcal H, \tT_{\mathcal H},(-))$ is a
sub-semiring system of $\mathcal A'$ containing $\mathcal A'_\Null,$
over which $B$
 still   $\preceq$-spans $\mathcal A'$ over $\tT$, then $\mathcal A'\preceq \mathcal H.$\end{lem}
\begin{proof}
For any element $w$ of $\mathcal A'$, we can write $wv_1
  \preceq \sum h_iv_i$ for suitable $h_i \in \mathcal H \subseteq \mathcal A'$.
Hence, we have that
$$ (h_1(-)w)v_1 + h_2v_2 + \dots + h_nv_n \in \mathcal A'_{\Null}.$$

It follows from linear dependence of the $v_i$ over $\mathcal A'$
that each coefficient must be  in $\mathcal A'_{\Null};$ in
particular, $w \preceq w +(h_1(-)w) = h_1 + w^\circ,$ proving
$\mathcal A'\preceq \mathcal H.$
\end{proof}
%

\begin{thm} [Artin-Tate lemma, $\preceq$-version]\label{thma}   Suppose
$\mathcal A'=\mathcal A[a_1,a_2,\dots ,a_n]$ is a  $\preceq$-affine
system over~$\mathcal A$, and $\mathcal K$ a subsystem of $\mathcal
A',$ with $\mathcal A'$ having a $\preceq$-base $v_1 =1, \ \dots,
v_d$ of $\mathcal A'$ over $\mathcal K$. Then $\mathcal K$ is
$\preceq$-affine over~$\mathcal A$.\end{thm}
\begin{proof}
There are suitable $\a_{ijk}, \a_{iv} \in \mathcal K$ such that
\begin{equation}\label{5.4} v_i v_j \preceq \sum _{k=1}^d \a_{ijk}v_k, \qquad a_u \preceq \sum _{k=1}^d\a_{uk}v_k, \quad 1\le i,j \le d,\
1\le u \le n. \end{equation} Let $H = \mathcal A[\a_{ijk}, \a_{uk} :
1 \le i,j,k \le d, \ 1 \le u \le n]\subseteq \mathcal K$, and
$\mathcal A_0:= \{v \in \mathcal A': v \preceq  \sum _{i=1}^d
Hv_i\}.$ The relations \eqref{5.4} imply that $\mathcal A_0$ is
closed under multiplication, and thus is a subalgebra of $\mathcal
A$ containing $a_1, \dots, a_n,$ implying $\mathcal A\preceq
\mathcal A_0,$ which is obviously $\preceq$-affine.\end{proof}

These results can all be viewed in terms of the symmetrized triple.
To continue, we need a workable definition of ``algebraic.''
Presumably an algebraic element $b$ should be a symmetric root of a
tangible pair $(f,g)$ of polynomials. This is tricky since we need
to identify the tangible polynomials, since examining its
coefficients leads to the difficulty that $(\la +\one)(\la  (-)\one)
= \la ^2 + \one ^\circ \la (-)\one.$ To exclude such examples leads
us to functional considerations, when $\tT$ is infinite.

Suppose we are given an extension $(\mathcal A', \tT_\mathcal A ',
(-))$ of $(\mathcal A, \tT_\mathcal A , (-)) $.

\begin{defn} A polynomial $f \in \mathcal A[\la_1, \dots, a_m]$ is
\textbf{functionally tangible} (with respect to $\tT'$) if $f(a'_1,
\dots, a'_m)\in \tT$ for almost all $a'_1, \dots, a'_m$  in $\tT'.$

In the symmetric situation, a pair of polynomials $(f,g)\in
\widehat{ \mathcal A[\la_1, \dots, \la_m]}$ is \textbf{symmetrically
functionally tangible} if for each $b \in \tT$ there are only
finitely many $ b ' \in \tT$ for which $f(b )= g(b')$.
\end{defn}

\begin{rem} Suppose $\tT$ is  infinite, and $(\mathcal A , \tT , (-))$ is a prime triple. Then
it is enough to check when   $b' \in \tT$ (so that $\tT'$ is
irrelevant), and the set of symmetrically functionally tangible
pairs is a submonoid of $\mathcal A[\la_1, \dots, \la_m]$.\end{rem}

\begin{defn}\begin{enumerate} \eroman
\item
An element $a \in \tT' := \tT_{\mathcal A '}$ is
\textbf{symmetrically algebraic} if there is a symmetrically
functionally tangible pair $(f,g) \in \widehat{ \mathcal A[\la_1]}$
for which $f(a) = g(a)$.

\item Given an   extension $(\mathcal A', \tT_\mathcal A ', (-))$ of
$(\mathcal A, \tT_\mathcal A , (-)) $, and $a \in  \tT_\mathcal A
',$ we define the $a$-\textbf{denominator set} $S_a = \{g(a): g$ is
a symmetrically functionally tangible polynomial   with $g(a) \in
\tT\}.$

\item An $\preceq$-affine extension $(\mathcal A', \tT_\mathcal A ', (-))$ of
$(\mathcal A, \tT_\mathcal A , (-)) $ is \textbf{fractionally
closed} over $a$ if every element of $S_a$ is invertible in
$\tT_\mathcal A '$.
\end{enumerate}\end{defn}

\begin{lem} If $a$ is symmetrically algebraic, i.e., $f(a) = g(a)$,
and $t$ is the largest number such that at least one of the
coefficients of $\la^n$ in $f$ and $g$ is tangible, then $\one, a,
\dots, a^{t-1}$ is a symmetric base of $\mathcal A[a]$.
\end{lem}
\begin{proof} $\one, a,
\dots, a^{t-1}$ are independent, by choice of $t$. But $a^t$ is
dependent on  $\one, a, \dots, a^{n-1}$, and continuing inductively,
each $a^m$ is dependent on  $\one, a, \dots, a^{t-1}$.
\end{proof}
\begin{rem} $S_a$ is a monoid, so we can localize, and $(S_a^{-1}\mathcal
A', S_a^{-1}\tT_\mathcal A ', (-))$ is fractionally closed over~$a$.
\end{rem}

\begin{lem}\label{key1} Suppose $\tT$ has the property that for any
$a \in \tT,$ the set $\{ a + c: c \in \tT\}$ is infinite, and
$(\mathcal A', \tT', (-))$ is a symmetrized triple, fractionally
closed over~$a$, with $S\supseteq \tT$ a subgroup of $\mathcal A'$,
and $ a'\in \tT'.$ If the   module of fractions $Q_S(\mathcal
A[a'])$ of $\mathcal A[a']$ is $\preceq$-affine over~$\mathcal A$,
then $Q_S(\mathcal A[a'])=\mathcal A[a']$.\end{lem}
\begin{proof} We may assume that $\mathcal A' = Q_S(\mathcal A[a'])$. Writing $\mathcal A' =
\mathcal A\left[\frac {f_1(a')}{g_1(a')}, \dots , \frac
{f_n(a')}{g_n(a')}\right]$, for $f_i,g_i \in \mathcal A'[\la ],$
$g_i$ symmetrically algebraic, note by Remark\ref{5.13}(ii) that we
may assume that all the denominators are equal, i.e.,
$$g_1(a') = g_2(a') = \dots = g_n(a'),$$ which we write as
$g(a').$ Any element of $\mathcal A'$ can be written with
denominator a power of $g(a')$, which means in particular, for any
$c \in \tT,$
$$\frac \one{g(a')+c} = \frac {f(a')}{g(a')^m}$$ for suitable $m$ and
suitable $f \in F[\la ].$ Thus, $f(a')(g(a')+c) = g(a')^m.$
Consequently, $a'$ is a symmetric root of $(f(\la)(g(\la)+c) ,
g(\la)^m). $  Hence $f(\la)(g(\la)+c) = g(\la)^m$ as functions, for
infinitely many $c\in \tT$, which is impossible for $g$ nonconstant
(by comparing factorizations). Hence $g$ is a constant, and
$Q_S(\mathcal A[a'])=\mathcal A[a']$.
\end{proof}

The following assertion  sometimes is called the ``weak
Nullstellensatz,'' since in classical mathematics it can be used to
prove Hilbert's Nullstellensatz cf.~\cite[Theorem 10.11]{Row08}.

\begin{thm}\label{thma1} If $(\mathcal A, \tT, (-))$ is a semiring-group system  over   $\tT$ and
$(\mathcal A ' = \mathcal A[a_1, \dots, a_m], \tT',(-))$ is a
$\preceq$-affine semiring-group system, in which $(f,g)(a_i, \zero)$
is invertible for every symmetrically functionally tangible pair
$(f,g)$ of polynomials, then $\widehat{\mathcal A '}$
 has a symmetric base over $\tT$.
\end{thm}
\begin{proof}
 We
follow the proof given in \cite[Theorem A]{Row16}, based on the
Artin-Tate lemma. Namely, take~$\mathcal K$ to be
$(S^{-1}\widehat{\mathcal A'},\tT_{S^{-1}\widehat{\mathcal
A'}},(-)),$ where $S = \{ (f,g)(a_1,\zero): (f,g) \text{
symmetrically functionally tangible} \}$. By induction on $n$,
$\mathcal A '$  has a symmetric base over  $\mathcal K$, so is
$\preceq$-affine by Theorem~\ref{thma}, and thus has a symmetric
base.  We conclude with Lemmas~\ref{key111} and~\ref{5.14}.
\end{proof}

\subsection{Classical Krull dimension}$ $

\begin{defn} The \textbf{height} of a chain  $\mathcal P _0 \supseteq  \mathcal P _1 \supseteq \dots \supseteq  \mathcal P _t$ of  prime $\tT$-congruences
is $t$. The \textbf{(Krull) dimension} is the maximal length $n$ of
a chain  of prime $\tT$-congruences of $\mathcal A.$
\end{defn}

$\tT$-Homomorphisms from polynomial triples have an
 especially nice form.

 \begin{defn} The \textbf{transcendence degree} of $\mathcal A [[\la_1, \dots, \la_n]]
 $ over a semiring-group system $\mathcal A $ is $n$.

 A congruence  $\Cong$ on
$\mathcal A [[\la_1, \dots, \la_n]]$ is \textbf{projectively
$\tT$-trivial} if all of the elements of $\CongT$ have the form
$(a_0 h, a_1h)$ where $a_i \in \tT$ and $h$ is a monomial.

 A \textbf{substitution homomorphism} $\varphi: \mathcal A [[\la_1, \dots, \la_n]] \to ( \mathcal A',
\tT',(-))$    is a $\tT$-homomorphism determined by substitution of
$ \la_1, \dots, \la_n $ to elements of $\tT'$.
 \end{defn}

\begin{lem}\label{homim1} If $\Cong$ is a non-projectively trivial prime $\tT$-congruence of the Laurent polynomial triple $(\mathcal A [[\la_1, \dots,
\la_n ]], \tT_{\mathcal A [[\la_1, \dots, \la_n]]},(-))$, and $
\mathcal A' := (\mathcal A [[\la_1, \dots, \la_n ]], \tT_{\mathcal A
[[\la_1, \dots, \la_n]]},(-))/Cong,$ then $\mathcal A'$ is
isomorphic to a  Laurent polynomial triple of transcendence degree
$< n$, and the natural $\tT$-homomorphism
$$(\mathcal A [[\la_1, \dots, \la_n ]], \tT_{\mathcal A [[\la_1,
\dots, \la_n]]},(-))\to \mathcal A'$$ is a substitution
homomorphism.
\end{lem}
\begin{proof} We view the images of $\la_1, \dots, \la_n $ as Laurent monomials, in the sense that $(a_0 h_0, a_1 h_1)$ is identified
with $(a_0 a_1)^{-1} \frac {h_0}{h_1}$ where $a_i \in \tT$ and $h_i$
are pure monomials in the $\la _i.$ The hypothesis that $\Cong$ is
non-projectively trivial means that the image of some $\la_i$ can be
solved in terms of the others, so the transcendence degree
decreases.
\end{proof}

Because of our restricted definition of homomorphism (sending
monomials to monomials), the next theorem comes easily.

 \begin{thm}\label{polyzyz}
Both the polynomial $\mathcal A[\la_1, \dots, \la_n]$ and Laurent
polynomial systems $\mathcal A[[\la_1, \dots, \la_n]]$ in $n$
commuting indeterminates over a $\tT$-semiring-group system have
dimension $n$.\end{thm}
\begin{proof} In view of Lemma~\ref{homim1}, the chain of prime $\tT$-congruences of  $\mathcal A[[\la_1, \dots, \la_n]]$
correspond to a homomorphic chain $$\mathcal A[[\la_1, \dots,
\la_n]]\to  \mathcal A[[\la_1, \dots, \la_{n-1}]]\to  \mathcal
A[[\la_1, \dots, \la_{n-2}]]\to \dots$$ (reordering the indices if
necessary) and this must stop after $n$ steps.
\end{proof}

\section{Tensor products}\label{tenpro}$ $

Two of the most important functors in the category theory of modules
are the tensor product $\otimes$ and $\Hom.$   We turn to triples
and systems emerging over a given ground triple $(\mathcal A, \tT ,
(-))$, and the categories ($\otimes$ and $\Hom$) that arise from
them. Both appertain to systems, but each with somewhat unexpected
difficulties.  $\Hom$ was studied in \S\ref{sysrep}, so we focus on
tensor products,  a very well-known process in general category
theory \cite{Gr,Ka1}, as well as over semirings \cite{Ka2,Tak},
which has been studied formally in the context of \textbf{monoidal
categories}, for example in \cite[Chapter 2]{EGNO}.

\subsection{Tensor products of systems}\label{tenpro1}$ $

 Here we need
the tensor product of   systems over a ground $\tT$-system. These
are described (for semirings) in terms of congruences, as given for
example in~\cite[Definition~3]{Ka2} or, in our notation,
\cite[\S3]{Ka3}. This material also is a special case of
\cite[\S~1.4.5]{grandis2013homological}, but we present details
which are specific to systems, to see just how far we can go with
$\preceq$-morphisms and the negation map.

 En route we also hit a
technical glitch in applying universal algebra, which historically
appeared  before tensor categories. The tensor product, which exists
for systems, cannot be described directly in universal algebra,
since the length $t$ of a sum $\sum _{i=1}^t  a_i \otimes b_i$ of
simple tensors need not be bounded. So one needs a ``monoidal''
universal algebra, where the signature contains
 the tensor products of the original structures, which is beyond the
 scope of this paper.

 Let
us work with a right $\mathcal A$-module system $\mathcal M_1$ and
left $\mathcal A$-module system $\mathcal M_2$ over a given ground
$\tT$-system $\mathcal A$. The following observations are well
known.

 \begin{defn}\label{bil17} A  map $\Phi: \mathcal M_1 \times
  \mathcal M_2 \to \mathcal N$ is \textbf{bilinear} if
  \begin{equation}\label{bil177}\Phi\bigg(\sum_j x_{1,j}, \sum_k x_{2,k}\bigg) =  \sum_{j,k}\Phi\big( x_{1,j},  x_{2,k}\big) ,\quad \Phi(  x_1a, x_1' )= \Phi(x_1,a x_1' ),\end{equation}$ \forall
x_{i,j}\in \mathcal M_i,  \,  a \in \mathcal A$.
\end{defn}

One   defines the tensor product $\mathcal M_1 \otimes _{\mathcal A}
\mathcal M_2$ of $\mathcal M_1$ and $\mathcal M_2$ in the usual way,
to be $(\mathcal F_1 \oplus \mathcal F_2)/\Cong,$ where $\mathcal
F_i$ is the free system (respectively right or left) with
base~$\mathcal M_i$ (and $\tT_{\mathcal F_i} = \mathcal M_i$), and
$\Cong$ is the congruence generated by all
 \begin{equation}\label{defcong}\bigg(\big(\sum_j x_{1,j}, \sum_k x_{2,k}\big), \sum_{j,k}\big( x_{1,j},  x_{2,k}\big)\bigg) ,\quad \bigg((  x_1 a, x_2 ), (x_1,a
 x_2
)\bigg)\end{equation} $ \forall x_{i,j},x_{i,k}\in \mathcal M_i,  \,
a \in \mathcal A$.

\begin{rem}\label{bil2} Any bilinear  map $\Psi  : \mathcal M_1 \times
  \mathcal M_2 \to \mathcal N$
induces a map $\bar \Psi : \mathcal M_1 \otimes
  \mathcal M_2 \to \mathcal N$ given by $\bar \Psi (a_1 \otimes a_2)
  = \Psi (a_1 , a_2),$ since $\Psi $ passes through the defining congruence $\Cong$ of
  \eqref{defcong}.
\end{rem}

 To handle negation maps we take a
slightly more technical version emphasizing $\tT_{\mathcal A}$.

 \begin{defn}\label{tens1} The \textbf{$\tT_{\mathcal A}$-tensor product} $\mathcal M_1  \otimes _{\tT_{\mathcal A}} \mathcal M_2$ of a right $\tT_{\mathcal A}$-module system  $\mathcal M_1$
 and  a left $\tT_{\mathcal A}$-module system
$\mathcal M_2$ is $(\mathcal F_1 \oplus \mathcal F_2)/\Cong,$ where
$\mathcal F_i$ is the free system with base $\mathcal M_i$ (and
$\tT_{\mathcal F_i} = \mathcal M_i$), and $\Cong$ is the congruence
generated as in \eqref{defcong}, but now with $a \in \tT_{\mathcal
A}.$

 If $\mathcal M_1,
\mathcal M_2 $ have  negation maps $(-)$, then we define a
\textbf{negated tensor product} by further imposing the extra axiom
$$((-)x) \otimes y = x \otimes ((-)y).$$ Note that this is done by modding
out by the congruence generated by all elements $((-)x \otimes y,\ x
\otimes (-)y)$, $x,y \in \tT_{\mathcal M}$, in the congruence
defining the tensor product in the universal algebra framework. From
now on, the notation $\mathcal M_1 \otimes \mathcal M_2$ includes
this negated tensor product stipulation, and $\mathcal A$ and
$\tT_{\mathcal A}$ are understood.
\end{defn}

We can incorporate the negation map into the tensor product,
defining   $(-)(v \otimes w) := ((-) v) \otimes w.$

\begin{rem}
As in the classical theory, if $\mathcal M_1$ is an $(\mathcal
A,\mathcal A')$-bimodule system, then $\mathcal M_1 \otimes \mathcal
M_2$ is an $\mathcal A$-module system. In particular this happens
when $ \mathcal A $ is commutative and the right and left actions
on~$\mathcal M_1$ are the same; then we take $ \mathcal A' =
\mathcal A .$\footnote{There is some universal algebra lurking
beneath the surface, since one must define an abelian carrier. We
are indebted to D.~Stanovsky and M.~Bonato for 
 pointing out the
references~\cite{FM} and \cite[Definitions~3.5, 3.7]{Ou}, which are
rather intricate; also see \cite[Definition~4.146]{MMT}.}\end{rem}
 Since there are more relations in the defining
congruence, the $\tT_{\mathcal A}$-tensor product maps down onto the
negated tensor product.

%


\begin{defn}\label{bil3}
The \textbf{tensor product of triples} $(\mathcal A, \tT, (-))$ and
$ (\mathcal A', \tT', (-)')$  is the triple $$(\mathcal A \otimes
\mathcal A', \{ a_1 \otimes a_2 : a_1 \in \tT, a_2 \in \tT'\},
(-)\otimes 1_{\mathcal A'}).$$
 \end{defn}

%
%
%

%

In order to be able to apply the theory of monoidal categories, we
need to be able to show that the tensor product is functorial; i.e.,
given morphisms $f_i : \mathcal M_i \to \mathcal N_i$ for $i = 1,2,$
we want a well-defined morphism $f_1 \otimes f_2 : \mathcal M_1
\otimes \mathcal M_2 \to \mathcal N_1 \otimes \mathcal N_2$.

 We   would define the tensor product $f_1  \otimes f_2 $ of
$\preceq$-morphisms by $(f_1\otimes f_2)(a \otimes b) = f(a_1)
\otimes f_2(a_2),$ as a case of $f_1 * f_2$ in
Example~\ref{trivconv}, but   we run into immediate difficulties, even over free modules. 

\begin{example}\label{nonmonoidal} Consider the polynomial triple $\mathcal A[ \la_1,
\la _2]$ and the $\preceq$-morphism $f: \mathcal A[ \la_1, \la _2]
\to \mathcal A[ \la_1, \la _2]$ given by taking $f$ to be the
identity on all monomials and $f(q) = \zero$ whenever $q$ is a sum
of at least two nonconstant monomials.

Then $\la_1 \otimes \la_1 + \la_1 \otimes \la_2 + \la_2 \otimes
\la_2 =  \la_1 \otimes (\la_1 +   \la_2 )+ \la_2 \otimes \la_2 =
\la_1 \otimes \la_1 + (\la_1 + \la_2 )\otimes \la_2 ,$ so $$\la_2
\otimes \la_2 = (f \otimes f)(\la_1 \otimes \la_1 + \la_1 \otimes
\la_2 + \la_2 \otimes \la_2 ) = \la_1 \otimes \la_1,$$ so   $f
\otimes f$ is not well-defined. The same argument shows that even $f
\otimes 1$ is not well-defined.
\end{example}

As is well known, the process does work for homomorphisms, as a
consequence of Remark~\ref{bil2}:

\begin{prop}\label{bil37} Suppose that $f_i: \mathcal M_i \to \mathcal N_i$
are  homomorphisms. Then the map  $$f_1 \otimes f_2 : \mathcal M_1
\otimes \mathcal M_2 \to \mathcal N_1 \otimes \mathcal N_2$$ given
by $(f_1 \otimes f_2)(\sum _i a_{1,i} \otimes a_{2,i}) = \sum
f_1(a_{1,i}) \otimes f_2(a_{2,i})$ is a well-defined homomorphism.
 \end{prop}
%
%

%
%
%
\begin{rem}\label{adjmap}
 In view of \cite{Ka2}, also cf.~\cite[Proposition~17.15]{golan92},
one
 has the natural adjoint isomorphism
 $\Hom _{\mathcal A} (\mathcal M_1 \otimes _{\mathcal B} \mathcal M_2, \mathcal M_3)
  \to \Hom _{\mathcal B}(\mathcal M_1, \Hom _{\mathcal A} (M_2, M_3),$
  since the  standard proof given for algebras
does not involve negation. This respects tangible homomorphisms, so
yields an isomorphism of triples.
\end{rem}

\subsubsection{The tensor semialgebra triple}$ $\label{polys}

 Next, as usual, given a bimodule $V$ over  $\tT_{\mathcal A}$, one defines
$V^{\otimes (1)} = V,$ and inductively $$V^{\otimes (k)} = V \otimes
V^{\otimes (k-1)} .$$ From what we just described, if  $V$ has a
negation map $(-)$ then  $V^{\otimes (k)}$
also has a natural negation map, and often is a triple when $V$ is a triple. 
%

\begin{defn}\cite[Remark 6.35]{Row16} Define the \textbf{tensor semialgebra} $T(V) =
\bigoplus _k V^{\otimes (k)}$ (adjoining a copy of $\tT_{\mathcal
A}$ if we want to have a unit element), with the usual
multiplication.

Given a $\tT_{\mathcal A}$-module triple $(\mathcal M, \tT_{\mathcal
M},(-))$, the \textbf{tensor semialgebra triple} $(T(\mathcal M),
\tT_{T(\mathcal M)},(-))$ of $\mathcal M$ is defined by using the
negated tensor products of Definition~\ref{tens1} to define
$T(\mathcal M)$, induced from the negation maps on $\mathcal
M^{\otimes (k)}$; writing $\tilde a_k = a_{k,1} \otimes \dots
\otimes a_{k,k}$ for $a_{k,j}\in \tT_{T(\mathcal M)},$ we put
$$(-) (\tilde a_k) = (-) (a_{k,1} \otimes \dots \otimes a_{k,k}).$$
$\tT_{T(\mathcal M)}$ is $\cup \{ \tilde a_k :\}$, the set of simple
multi-tensors.
\end{defn}
 %
\
We can now view the  polynomial \semiring0 (Definition~\ref{poly1})
in this context.

\begin{example} Suppose $\mathcal A =
(\mathcal A, \tT, (-))$ is  a triple. The polynomial \semiring0
$\mathcal A [\la]$ now is defined as  a special case of the tensor
semialgebra.   $\tT_{A [\la]}$ again is the set of monomials with
coefficients in $\tT.$\end{example}
%





\section{The structure theory of module  systems via congruences}
\label{modsys}$ $

Aiming for a representation theory, we take a category of  module
systems  over our ground triples.  Throughout, we take $\tT$-module
systems over a ground system with a fixed signature.

\subsection{Submodules versus subcongruences}\label{modcpng}
%
%
%

\begin{defn}\label{idealdef}  A \textbf{$\tT$-submodule system}
 of $(\mathcal M, \tT, (-), \preceq)$
 is a submodule $\mathcal N$ of $\mathcal M$,   satisfying the following
 conditions, where  $a\in \tT$:
 \begin{enumerate}\eroman
 \item  Write $\tTNz$ for
$\tTN \cup \{\zero \}$.  $(\mathcal N, \tTN, (-), \preceq)$ is a
subsystem of $(\mathcal M, \tT_{\mathcal M}, (-),
 \preceq)$. (In particular, $\tTNz$ generates $(\mathcal N, +)$.)
  \item  $\tT^\circ \subseteq \mathcal N$ (and thus $\mathcal M^\circ \subseteq \mathcal N$).
\item If $a \preceq b +v,$ for $v \in \mathcal N$, then there is $w \in \tTNz$ for
which $a \preceq b +w.$
\end{enumerate}
\end{defn}

Note that (ii) implies that $\{ \zero\}$ is not a submodule, for we
want submodules to contain the null elements. In what follows,
${\mathcal N}$ always denotes a $\tT$-submodule system of ${\mathcal
M}$.

\begin{rem} The definition implicitly includes the condition  that
$(-)\tTNz = \tTNz,$ since $(-)a = ((-)\one)a.$
\end{rem}
%

\begin{lem} If $a \in \mathcal N$ and $a \preceq_\circ b$,  then $ b \in \mathcal N.$
\end{lem}
\begin{proof} Just write $b = a + c^\circ,$ noting that $c^\circ \in \mathcal N$.
\end{proof}

%
%

The first stab at  defining a $\tT$-module of a $\tT$-congruence
$\Cong$ might be to take $\{ a(-)b: (a,b) \in \Cong\},$ which works
in classical algebra. We will modify this slightly, but the real
difficulty  lies in the other direction. The natural candidate for
the $\tT$-congruence of a $\tT$-submodule $\mathcal N$ might be
$\{(a,b): a(-)b \in \mathcal N\}$, but it fails to satisfy
transitivity!

\begin{defn}\label{idealdef1} $ $
\begin{enumerate}
\item
Given a $\tT$-submodule $\mathcal N$ of $\mathcal M$, define the
$\tT$-congruence $\Cong_{\mathcal N}$ on $\mathcal N$ by $a \equiv
b$  if and only if we can write $a = \sum _j a_j$ and $b= \sum _j
b_j$ for $a_j,b_j \in \tTNz$ such that  $a_j \preceq b_j + v_j$ for
$v_j \in \tTNz$, each $j.$
\item
Given a  $\tT$-congruence $\Cong$, define ${\mathcal N}_\Cong$ to be
the additive sub-semigroup of $\mathcal M$ generated by all $c\in
\tTMz$ such that $c =   a(-)b$ for some $a,b \in \tTNz$ such that
$(a,b) \in {\Cong}$.
\end{enumerate}

%

\end{defn}

\begin{example} When the system $\mathcal M$ is meta-tangible,
then in the definition of $\Cong_{\mathcal N}$, either $b_j =
(-)v_j$ in which case $a_j \preceq v_j^\circ \in \tT_{\mathcal
N}^\circ,$ or $a_j = b_j
 $ (yielding the diagonal) or $a_j = v_j \in \tT_{\mathcal N}.$
\end{example}

The results from {\cite{Row16}} pass over, with the same proofs, to
module systems.

\begin{defn}\label{metadef1}
A $\tT$-module  system $\mathcal M = (\mathcal M, \tT_{\mathcal M} ,
(-), \preceq)$ is
 $\tT_{\mathcal M} $-\textbf{reversible} if
 $a_1 \preceq  a_2 +b   $ implies  $a_2 \preceq a_1  (-) b$  for $a_1 ,a_2 \in \tT_{\mathcal M}
 $ and $b \in  \mathcal M $.
\end{defn}

\begin{lem}[{\cite[Proposition~6.12]{Row16}}]\label{revsys}
In a $\tT$-reversible system, $a \equiv b$ (with respect to
$\Cong_{\mathcal N}$) for $a,b\in \tT_{\mathcal N}$,  if and only if
either $a=b$ or $\tT_{\mathcal N}$ contains an element $v$ such that
$v \preceq a(-)b.$
\end{lem}

\begin{rem}[{\cite[Remark~8.29]{Row16}}]\label{rever} In a $\tT$-reversible system, Condition (iii) of Definition~\ref{idealdef} yields $w \preceq a(-)b.$
Likewise, in Definition~\ref{idealdef1}, $ a_j \preceq b_j +v_j$
implies $b_j \preceq a_j (-)v_j$.
\end{rem}

\begin{prop}[{\cite[Proposition~8.30]{Row16}}]\label{ideal} In a $\tT$-reversible system, $\Cong_{\mathcal N}$ is a
$\tT$-congruence for any $\tT$-submodule ${\mathcal N}$. For any
$\tT$-congruence~$\Cong,$ ${\mathcal N}_\Cong$ is a $\tT$-submodule.
Furthermore, $\Cong_{{\mathcal N}_\Cong} \supseteq \Cong$ and
${\mathcal N}_{\Cong_{\mathcal N}} = {\mathcal N}.$
 \end{prop}

\subsection{The N-category of  $\tT$-module  systems}\label{surpre3}$ $

 Since we lack the classical
negative, the trivial subcategory of $\zero$ morphisms and $\zero$
objects is replaced by a more extensive subcategory of quasi-zero
morphisms and quasi-zero objects. The quasi-zero morphisms have been
treated formally in \cite[\S 1.3]{grandis2013homological} under the
name of \textbf{N-category} and \textbf{homological category}, with
the terminology ``null morphisms'' and ``null objects,'' and with
the null subcategory designated as $ \Null $\footnote{\cite[\S
1.3]{grandis2013homological} uses the notation $(E,\mathcal N),$ but
we already have used $\mathcal N$ otherwise.}.  

To view systems $\mathcal A= (\mathcal A, \tT, (-), \preceq)$ in the
context of N-categories, we must identify the null objects $\mathcal
A_{\Null}$. The intuitive choice might be $\mathcal A^\circ,$ but
${\mathcal A}_{\Null}$ often seems to be more inclusive. We  turn to
$\tT$-module systems $(\mathcal M, \tT_{\mathcal M}, (-), \preceq),$
to which we refer merely as $\mathcal M$ for shorthand. From now on
we take $\preceq = \preceq_\circ$ in order to simplify the
exposition.

\begin{defn} \label{t-trivialmorhphism}$ $
\begin{enumerate} \eroman
\item
A \textbf{chain} of $\tT$-morphisms of $\tT$-module systems is a
sequence
\[ \cdots \to  \mathcal K \too {g}  \mathcal M  \too {f}  \mathcal N \to \cdots \]
such that $(fg)(k) \in \mathcal N_{\Null} $ for all $k \in \mathcal
K$.

\item (Compare with \cite{Abu}) The chain is \textbf{exact}  at $ \mathcal M $ if $g(\mathcal K) = \{ b
\in \mathcal M : f(b)\in \mathcal N_{\Null}\}.$
\end{enumerate}\end{defn}

One can easily see from the above definitions that if $\mathcal K $
and $\mathcal N$ are null (i.e., $\mathcal K= \mathcal K_{\Null}$
and $\mathcal N= \mathcal N_{\Null}$) and the chain is exact, then
$\mathcal M $ is null as well.

\begin{defn} \label{t-trivialmorhphism1}$ $
Let $f: \mathcal M \to \mathcal N$ be a morphism of $\tT$-modules.
\begin{itemize}
\item
The \textbf{$\tT$-module kernel} $\tT$-$\ker f$ of $f$ is $\{ a \in
\tT: f(a) \in \mathcal N_{\Null}\}.$
\item
$f$ is \textbf{null}  if $f(a) \in \mathcal N_{\Null}$ for all $a
\in \tT_{\mathcal M}$, i.e. $\tT$-$\ker f = \tT_\mathcal M $.
\item
The \textbf{$\tT$-module image} $f_{\tT_{\mathcal M}}(\mathcal M)$
is the $\tT$-submodule spanned by $\{ f(a): a  \in \tT_{\mathcal
M}\}$.
\item
$f$ is $\Null$-\textbf{monic} if $f(a_0) = f(a_1)$ implies that $a_0
(-) a_1 \in \mathcal M_{\Null}.$
\item
$f$ is $\Null$-\textbf{onto} if $f_{\tT_{\mathcal M}}(\mathcal M) +
\mathcal N_{\Null}=\mathcal N.$
\end{itemize}
\end{defn}

Thus, the null morphisms are closed null in the categorical sense,
and
  take the place of the $\zero$ morphism. Since   $\zero \notin
\tT_{\mathcal M}$, we might take $f(\tT_{\mathcal M})$ to be any
tangible constant $z$, i.e., $z \in \tT_{\mathcal N}$.
%



\subsection{The  category of congruences of  $\tT$-module  systems}\label{surpre31}$ $

We have the analogous results for $\tT$-congruences. The next
definition takes into account that any $\tT$-congruence contains the
diagonal.

\begin{defn}\label{cok8} For any
$\tT$-congruence $\Cong$ and any morphism $f: \mathcal M \to
\mathcal N$, define the  \textbf{congruence image} $f(\mathcal M
)_{\congg}$ to be the $\tT$-congruence \begin{equation}\label{cong1}
\Diag_\mathcal N + \left\{ \left(\sum_i f(a_i), \sum_j f(a'_j)\right):   \
\left(\sum a_i,\sum a'_j\right) \in \Cong \right\}.\end{equation}


The $\tT$-\textbf{congruence kernel} \, $\TkerC (f)$ is the
$\tT$-congruence of $\mathcal M$, generated by $$\{ (a_0,a_1)\in \tT_{\mathcal M} :
f(a_0) = f(a_1) \}.$$
\end{defn}

For $f$, the term in the summation in \eqref{cong1} is $\{
(f(x_0),f(x_1)) : x_i\in \mathcal M, \, (x_0, x_1) \in \Cong \}$.

 In particular,   $ f(\mathcal M)_{\congg} $ is $\Diag
_\mathcal N + \left\{ \left(\sum_i f(a_i), \sum_j f(a'_j)\right):
\sum a_i = \sum a'_j\right\},$ the $\tT$-congruence of $\mathcal N$
generated by $\Diag _\mathcal N$ and $\{(f(a_1),f(a_2)) : a_i\in
\tT_{\zero,\mathcal M} \}$.

The fact that  $ f(\Cong) $ is more complicated for
$\preceq$-morphisms presents serious obstacles later on.

\begin{defn} A  congruence morphism $f : \Cong \to
\Cong'$ is \textbf{trivial} if $f(\zero) = \zero$ and $ f(\Cong)
\subseteq \Diag_{\Cong'}$.

\end{defn}
%

\begin{lem}\label{prescon} For any   morphism $f:  \mathcal M \to \mathcal
N$, $ \TkerC (f) = \{ (x,x') \in  \mathcal M : f(x) = f(x')\}.$
 \end{lem}
 \begin{proof} Write $x = \sum x_i$ and $x' = \sum x_j'$ and $y = \sum y_i$ and  $y' = \sum y_j'$. If $f(x) =
 f(x')$ and $f(y) =
 f(y')$, then $ f(x+y)=  \sum f(x_i)+  \sum f(y_i)= \sum f(x_j')+\sum f(y_j') =
 f(x'+y')$.
 \end{proof}

\begin{lem}\label{monep}
For any  $\tT_{\mathcal M}$-morphism $f:  \mathcal M \to \mathcal
N$, the induced morphism $\widehat f:\widehat{\mathcal M} \to
\widehat{\mathcal N}$ is  $\Null$-onto if and only if $f$ is epic.
\end{lem}
 \begin{proof}
If $\widehat f:\widehat{\mathcal M} \to \widehat{\mathcal N}$  is
$\Null$-onto and $\widehat{gf}$ is trivial, then
$\widehat{gf}(\mathcal
 M)
 = \widehat g (\widehat{\mathcal M}),$ implying $\widehat g$ is trivial.\\
Conversely, assume that $\widehat f:\widehat{\mathcal M} \to
\widehat{\mathcal N}$  is not $\Null$-onto.
  Then take the canonical map $g : \mathcal N \to \mathcal N/f(\mathcal M)_{\congg}$ to get $gf$ trivial,
 but $g$ is not trivial, so $f$ is not epic.
\end{proof}

To construct cokernels, we define $\mathcal N \to \mathcal N
/f(\mathcal M)_{\congg},$ which will turn out to be the
categorical  cokernel of $ f.$ 

\begin{lem}\label{kercok0} For any  morphism $f: \mathcal M \to \mathcal N$, there is a $\Null$-monic $\bar f :\mathcal M/\TkerC (f) \overset
{\overline{f}} \to  \mathcal N,$ given by  $\overline{f} ([a]) =
f(a)$, where $[a]$ is the equivalence class of $a \in \tT$ under the
congruence kernel $\TkerC (f)$.
\end{lem}
\begin{proof} If $(a,a') \in \TkerC (f),$ then $f(a) = f(a')$,
showing that $\overline{f}$ is well-defined. Clearly $\overline{f}$
is a $\tT$-morphism since $\Tker (f)$ is a $\tT$-congruence.
Finally, one can easily check that $\overline{f}$ is a monic as it
is injective.
\end{proof}

\begin{lem}\label{kercok10}
  A  $\tT_{\mathcal M}$-morphism $f:  \mathcal M \to \mathcal N$  is  $\Null$-monic  if and only if $ \TkerC (f)$ is
 diagonal.
\end{lem}
\begin{proof}
Assume first that the $\tT$-congruence $ \TkerC (f)$ is not diagonal, i.e., $b
= \sum \a_i \ne \sum a_j' = b'$ although $\sum f(a_i) = \sum
f(a_j').$ Define the cokernel $\bar f: \TkerC (f) \to \mathcal M$ by
$\bar f (a) = \sum f(a_i)$ where $a = \sum a_i.$   $f \bar f$ is
trivial,  implying $ \TkerC (f)$ is trivial.

On the other hand, if  $ \TkerC (f)$ is   diagonal, and $fg $ is
trivial, then $g( \mathcal M )$ is diagonal, i.e., $g$ is trivial.
\end{proof}
%

%

\begin{lem}\label{kercok}
Any  morphism  $ f: \mathcal M \to \mathcal N $ is composed as  $
\mathcal M  \to \mathcal M/\TkerC (f) \overset {\overline{f}} \to
\mathcal N,$ where  the first map is the canonical  morphism, and
the second map is given in Lemma~\ref{kercok0}.
\end{lem}
\begin{proof}
The first map sends $a \in \mathcal M$ to $\bar{a}$, where $\bar{a}$
is the equivalence class of $a$ under $\TkerC (f)$. This is clearly
a $\tT$-morphism and  since $\TkerC (f)$ is a congruence relation.
The second map is just Lemma~\ref{kercok0}.
\end{proof}

We call $\bar f$   the monic \textbf{associated} to the
$\tT$-congruence kernel.

\section{Functors among semiring ground triples and systems}\label{func1}$ $

In this section we recapitulate the previous connections among the
notions of systems, viewed categorically. We focus on ground
triples, but at the end indicate some of the important functors for
module systems, in prelude to \cite{JMR1}.


\begin{nota}
Let us introduce the following notations:
\begin{itemize}
\item
$\mathfrak{Rings}$= the category of commutative rings.
\item
$\mathfrak{Doms}$= the category of integral domains.
\item
$\mathfrak{SRings}$= the category of commutative semirings.
\item
$\mathfrak{SDoms}$= the category of commutative semidomains.
\item
$\mathfrak{SDM}$= the category whose objects are pairs $(c S,\tT)$
consisting of a semiring $\mathcal S$ and a multiplicative submonoid
$\tT$ of $\mathcal S$ which (additively) generates $\mathcal S$ and
does not contain $0_\mathcal S$. A homomorphism from $(\mathcal
S_1,\tT_1)$ to $(\mathcal S_2,\tT_2)$ is a semiring homomorphism
$f:\mathcal S_1 \to \mathcal S_2$ such that $f(\tT_1) \subseteq
\tT_2$ and $f(\tT_1)$ generates $\mathcal S_2$.
\item
$\mathfrak{SRT}_{\tT}$= the category of $\tT$-triples, with
$\preceq_\circ$-morphisms, whose objects are semirings.
\item
$\mathfrak{HDoms}$= the category of hyperrings without
multiplicative zero-divisors, but with $\preceq$-morphisms.
\item
$\mathfrak{HFields}$= the category of hyperfields.
\item
$\mathfrak{FRings}_w$= the category of fuzzy rings with
$\preceq$-morphisms (cf. \cite{GJL}).
\item
$\mathfrak{FRings}_{str}$= the category of coherent fuzzy rings
(cf. \cite{GJL} and \cite{Row16} for the notion of coherence).
\end{itemize}
\end{nota}


The following diagram illustrates how various categories are related
(note that this is not a commutative diagram, for instance,
$\textbf{e}\circ \textbf{t} \circ \textbf{i} \neq
\textbf{c}\circ\textbf{j}$):

\begin{equation}\label{algdiagram}
\begin{tikzpicture}[baseline=-0.8ex]
\matrix(m)[matrix of math nodes, row sep=7.5em, column sep=7em, text
height=1.5ex, text depth=0.25ex]
{\mathfrak{Doms}&   \mathfrak{SDoms}  & \mathfrak{SDM} & \mathfrak{SRT}_{\tT}\\
\mathfrak{HFields} &\mathfrak{HDoms}& \mathfrak{FRings}_w\\};
\path[->,font=\normalsize] (m-1-1) edge node[auto] {\textbf{j}}
(m-2-2) (m-1-1) edge node[auto] {\textbf{i}} (m-1-2) (m-1-2) edge
node[auto] {\textbf{t}} (m-1-3) (m-2-1) edge node[auto] {\textbf{k}}
(m-2-2) (m-2-1) edge node[auto] {\textbf{a}} (m-1-3)
(m-2-2) edge (m-2-3) (m-2-3) edge node[auto] {\textbf{d}}(m-1-4)
(m-1-3) edge node[auto] {\textbf{e}}(m-1-4) (m-2-2) edge node[auto]
{\textbf{c}} (m-1-4) (m-2-2) edge[dotted] node[auto] {\textbf{g}}
(m-2-3);
\end{tikzpicture}
\end{equation}
\vspace{0.3cm}

(If we are willing to bypass $\mathfrak{SDM}$ in this diagram, then
we could generalize the first two terms of the top row to
$\mathfrak{Rings}$ and $\mathfrak{SRings}$ and accordingly, in the second row, $\mathfrak{HDoms}$ can be generalized to the category of hyperrings.)

In the following propositions, we explain the functors in the above
diagram.  All of these functors are stipulated to preserve the
negation map.
 First, one can easily see that the functors
$\textbf{i},\textbf{j},\textbf{k}$ are simply embeddings. To be
precise:

\begin{prop}(The functors $\textbf{i}$, $\textbf{j}$, and $\textbf{k}$)
\begin{enumerate}
\item
The functor $\textbf{i}:\mathfrak{Doms} \to \mathfrak{SDoms}$,
sending an object $A$ to $A$ and a homomorphism $f \in
\Hom_{\mathfrak{Rings}}(A,B)$ to $f \in
\Hom_{\mathfrak{SRings}}(\textbf{i}(A),\textbf{i}(B))$, is fully
faithful.
\item
The functor $\textbf{j}:\mathfrak{Doms} \to \mathfrak{HDoms}$,
sending an object $A$ to $A$ and a homomorphism $f \in
\Hom_{\mathfrak{Rings}}(A,B)$ to $f \in
\Hom_{\mathfrak{HRings}}(\textbf{j}(A),\textbf{j}(B))$, is fully
faithful.
\item
The functor $\textbf{k}:\mathfrak{Hfields} \to \mathfrak{HDoms}$,
sending an object $A$ to $A$ and a homomorphism $f \in
\Hom_{\mathfrak{Hfields}}(A,B)$ to $f \in
\Hom_{\mathfrak{HRings}}(\textbf{k}(A),\textbf{k}(B))$, is faithful.
\end{enumerate}
\end{prop}
\begin{proof}
This is straightforward.
\end{proof}

\begin{rem}
The functor $\textbf{k}:\mathfrak{Hfields} \to \mathfrak{HDoms}$
cannot be full since $\preceq$-morphisms exist which are not
homomorphisms, cf.~\cite{Ju}.
\end{rem}

Now, for a commutative semidomain $A$, we let
$\textbf{t}(A)=(A,A-\{0_A\})$. It is crucial that $A$ is a
semidomain for $A-\{0_A\}$ to be a monoid. For a semiring
homomorphism $f:A_1 \to A_2$, we define
$\textbf{t}(f):(A_1,A_1-\{0_{A_1}\}) \to (A_2,A_2-\{0_{A_2}\})$
which is induced by $f$. Then, clearly $\textbf{t}$ is a functor
from $\mathfrak{SDoms}$ to $\mathfrak{SDM}$. In fact, we have the
following:

\begin{prop}\label{functor t}
The functor $\textbf{t}:\mathfrak{SDoms} \to \mathfrak{SDM}$ is fully
faithful.
\end{prop}
\begin{proof}
This is clear from the definition of $\textbf{t}$.
\end{proof}

\begin{rem}
One may notice that our construction of $\textbf{t}$ is not canonical since any semiring may have different sets of monoid generators. For instance, the coordinate ring of an affine tropical scheme may have different sets of monoid generators depending on torus embeddings, cf. \cite{GG}.
\end{rem}

The functors $\textbf{a}$ and $\textbf{g}$ are already constructed in \cite{GJL}. For the sake of completeness, we recall the construction.\\

Let $H$ be a hyperring. Then one can define the following set:
\begin{equation}\label{def: S(H) generating}
S(H):=\bigg\{\sum_{i=1}^n h_i \mid h_i \in H, \quad n \in
\mathbb{N}\bigg\}\subseteq \mathcal P(H),
\end{equation}
where $\mathcal{P}(H)$ is the power set of $H$. By
\cite[Theorem~2.5]{Row16}, \cite{GJL},   $S(H)$ is a semiring with
multiplication and addition as follows:
\begin{equation}\label{def: addition and multiplication}
(\sum_{i=1}^n h_i)(\sum_{j=1}^m h_j)=\sum_{i,j} h_ih_j \in S(H), \quad (\sum_{i=1}^n h_i)+(\sum_{i=1j}^m h_j)=\sum_{i,j}(h_i+h_i) \in S(H).
\end{equation}

Now, the functor $\textbf{a}:\mathfrak{Hfields} \to \mathfrak{SDM}$
sends any hyperfield $H$ to $(S(H), H^\times)$, i.e.,
$\textbf{a}(H)=(S(H),H^\times)$. Also, if $f:H_1 \to H_2$ is a
 homomorphism of hyperfields, then $f$ canonically induces a
homomorphism
\[
\textbf{a}(f):(S(H_1),H_1^\times,-) \to (S(H_1),H_1^\times,(-)),
\quad \textbf{a}(f)(\sum_{i=1}^n h_i)=\sum_{i=1}^n f(h_i).
\]
We emphasize that since the subcategory $\mathfrak{Hfields}$ of
$\mathfrak{HDoms}$ only has  homomorphisms, $\textbf{a}$ becomes
a~functor.

 The construction of the functor \textbf{g} is similar to
$\textbf{a}$; we use the powerset $\mathcal P(H)$ of $H$ instead
of $S(H)$ in this case. For details, we refer the readers to
\cite{GJL}. It is proved in \cite{GJL} that when one restricts the
functors to hyperfields, the functors \textbf{a} and \textbf{g} are
faithful, but not full.

\begin{rem}
Since a fuzzy ring assumes weaker axioms than semirings, the functor
\textbf{g} can be defined for all hyperrings, whereas the functor
\textbf{a} can be only defined for hyperfields with homomorphisms.
\end{rem}

Next, we construct the functor $\textbf{e}:\mathfrak{SDM} \to \mathfrak{SYS}_\tT$. To this end, we need to fix a negation map of interest, so a priori the functor $\textbf{e}$ is not canonical. For an object $(S,M)$ of $\mathfrak{SDM}$, we let $\textbf{e}(S,M)$ be the $\tT$-system $(\mathcal{A},\tT,(-),\preceq)$, where $\mathcal{A}=S$, $\tT=M$, and $(-)$ is the identity map 
and $\preceq = \preceq_\circ$. Since, we choose  $(-)$ to be   equality, any homomorphism $f:(S_1,M_1) \to (S_2,M_2)$ induces a homomorphism $\textbf{e}(f)$. \\

The functor $\textbf{c}$ is defined as follows: For a hyperring $R$ without zero-divisors,
we associate $S(R)$ as in \eqref{def: S(H) generating} and also
impose addition and multiplication as in \eqref{def: addition and
multiplication}.
 Now, the $\tT$-system $\textbf{c}(R)=(\mathcal{A},\tT,(-),\preceq)$
  consists of $\mathcal{A}=S(R)$, $\tT=R$, $(-):S(R) \to S(R)$ sending $A$ to $-A:=\{-a \mid a \in A\}$,
  where $-$ is the negation in $R$, and $\preceq$ is   set inclusion~$\subseteq$.
  One    checks easily that $\textbf{c}(R)$ is indeed a $\tT$-system and
  any homomorphism $f:R_1 \to R_2$ of hyperrings induces a morphism $\textbf{c}(f)$ of
  $\tT$-systems.

Finally, we review the functor $\textbf{d}:\mathfrak{FRings}_w \to
\mathfrak{SYS}_{\tT}$, defining $(-)a$ to be $\varepsilon a$.
\cite[Lemma~14.5]{Row16} shows how this can be retracted at times. One can easily see that the definition of coherent fuzzy rings is similar to $\tT$-systems; we only have to specify a negation map $(-)$ and a surpassing relation $\preceq$. To be precise, let $F$ be a fuzzy ring. The system $\textbf{d}(F)=(\mathcal{A},\tT,(-),\preceq)$ consists of $\mathcal{A}=F$, $\tT=\mathcal{A}^\times$, $(-):\mathcal{A} \to \mathcal{A}$ sending $a$ to $\varepsilon\cdot a$, and $\preceq$ is defined to be the equality. One can easily check that any $\preceq$-morphism $f:F_1 \to F_2$ of fuzzy rings induces a homomorphism $\textbf{d}(f)$ of the corresponding $\tT$-systems. \\


\begin{rem}
In the commutative diagram, one can think of forgetful functors in opposite directions. For instance, the functor $\textbf{t}$ has a forgetful functor (forgetting $\tT$) as an adjoint functor.
\end{rem}

\begin{rem}
Although we do not pursue them in this paper, we point out two
possible links to partial fields, first introduced by C.~Semple and
G.~Whittle \cite{SW} (see, also \cite{PZ} to study representability
of matroids). Recall that a commutative partial field
$\mathbb{P}=(R,G)$ is a commutative ring $R$ together with a
subgroup $G \leq R^\times$ of the group of multiplicative units of
$R$ such that $-1 \in G$.
\begin{enumerate}
\item
If one considers the subring $R'$ of $R$ which is generated by $G$,
then the pair  $(R',G,-,=)$ becomes a system.
\item
Any commutative partial field $\mathbb{P}=(R,G)$ gives rise to a
quotient hyperring $R/G$. This defines a functor from the
category of commutative partial fields to the category of
hyperrings.
\end{enumerate}
\end{rem}

\begin{rem}
Since any \semiring0 containing $\tT$ is a $\tT$-module, we have
forgetful functors from $\tT$-semiring systems to $\tT$-module
systems. We also have the tensor functor and Hom functors.
\end{rem}

\subsection{Valuations of semirings via systems}\label{val
7}$ $

We briefly mention one potential application of systems. 
In \cite{Ju2}, the first author introduced the notion of valuations
for semirings by implementing the idea of hyperrings, and this was
put in the context of systems in \cite[Definition 8.8]{Row16}.

\begin{defn}
Let $\mathbb{T}:=\mathbb{R} \cup \{-\infty\}$, where $\mathbb{R}$ is
the set of real numbers. The multiplication $\boxdot$ of
$\mathbb{T}$ is the usual addition of real numbers such that
$a\boxdot (-\infty)=(-\infty)$ for all $a \in \mathbb{T}$.
Hyperaddition is defined as follows:
\[
a\boxplus b =\left\{ \begin{array}{ll}
\max\{a,b\} & \textrm{if $a\neq b$}\\
\left[-\infty,a\right]& \textrm{if $a=b$},
\end{array} \right.
\]
\end{defn}

For a commutative ring $A$, a homomorphism from $A$ to $\mathbb{T}$
is what sometimes is called a ``semivaluation,'' (but in a different
context from which  we have used the prefix ``semi'') i.e., which
could have a non-trivial kernel. Inspired by this observation, in
\cite{Ju2}, the first author proposed the following definition to
study tropical curves by means of valuations; this is analogous to
the classical construction of abstract nonsingular curve via
discrete valuations.

\begin{defn}\label{def: valuation}
Let $S$ be an idempotent semiring and $\mathbb{T}$ be the tropical hyperfield. A valuation on $S$ is a function $\nu:S \to \mathbb{T}$ such that
\[
\nu(a\cdot b)=\nu(a)\boxdot \nu(b), \quad \nu(0_S)=-\infty, \quad \nu(a+b) \in \nu(a)\boxplus \nu(b), \quad \nu(S) \neq \{-\infty\}.
\]
\end{defn}


The implementation of systems, through the aforementioned functors,
allows one to reinterpret Definition \ref{def: valuation} as a
morphism in the category of systems. To be precise, let $S$ be a
semiring which is additively generated by a multiplicative submonoid
$M$. This gives rise to the $\tT$-system
$\textbf{e}(S,M)=(S,M,id,=)$ via the functor $\textbf{e}$. Also, for
the hyperfield $\mathbb{T}$, via the functor $\textbf{e}\circ
\textbf{a}$, we obtain a $\tT$-system, say $\mathcal{A}_\mathbb{T}$.
Then a semiring valuation on $S$ is simply a morphism
$\nu:(S,M,id,=) \to \mathcal{A}_\mathbb{T}$ of $\tT$-systems.

\subsection{Functors among module triples and systems}\label{func11}$
$

We conclude with some important functors needed to study module
triples and systems. Given a $\tT$-system $(\mathcal A,
\tT_{\mathcal A} , (-), \preceq),$ define $\mathcal
A$-$\mathfrak{Mod}$ to be the category of $(\mathcal A,
\tT_{\mathcal A} , (-), \preceq)$-module systems.
\begin{rem} \label{commatch3} The
\textbf{symmetrizing} functor injecting $\mathcal A$-$\mathfrak{Mod}$ into
$(\widehat{\mathcal A}, \tT_{\widehat{\mathcal A}} , (-), \preceq)$-$\mathfrak{Mod}$
is given by $f \mapsto (f,0),$ with the reverse direction  $(\widehat{\mathcal A}, \tT_{\widehat{\mathcal A}} , (-), \preceq)$-$\mathfrak{Mod}$
to $\mathcal A$-$\mathfrak{Mod}$
given by $(f_0, f_1) \mapsto f_0 (-)f_1.$ This functor respects the universal algebra approach since $\tT_{\mathcal M} \mapsto
\tT_{\widehat{\mathcal M}} .$

\end{rem}

Other important functors in this vein are the tensor functor and the
Hom functor.

\section{Appendix A: Interface between systems and tropical mathematics}\label{comp}

We relate our approach to tropical mathematics in view of systems to other approaches taken in tropical mathematics.

\subsection{Tropical versus supertropical}\label{supert1}$ $

\subsubsection{The ``standard'' tropical approach}$ $

Let $\mathbb{R}_{max}$ be the tropical semifield with the maximum convention. One often works in the polynomial semiring $\mathbb
R_{\operatorname{max}}[\Lambda]$, although  here we replace $\mathbb
R_{\operatorname{max}}$ by any  ordered semigroup $(\Gamma,\cdot)$,
with $\Gz : = \Gamma \cup \{ \zero \}$ where $\zero a = \zero$ for
all $a \in \Gamma$.   For $\bold i = (i_1, \dots, i_n)\in \mathbb
N^{(n)}$, we write $\Lambda ^\bold i $ for $ \la_1^{i_1}\cdots
\la_n^{i_n}.$ A tropical hypersurface of a tropical polynomial $f =
\sum _{\bold i } \a_{\bold i} \Lambda ^\bold i  \in \Gamma[\Lambda]$
is defined as the set of points in which two monomials take on the
same dominant value, which is the same thing as the supertropical
value of $f$ being a ghost.

\begin{defn}[{\cite[Definition~5.1.1]{GG}}] Given a   polynomial $f =
\sum \a_{\bold i}\Lambda ^\bold i \in \Gamma[\Lambda],$  define  $\supp( f )$ to be all
the tuples $\bold i =({i_1}\cdots  {i_n})$  for which the monomial
in $f$ has nonzero coefficient $\a_{\bold i}$, and for any such
monomial~$h,$ write $f_{\hat h}$ for the polynomial obtained from
deleting $h$ from the summation.

 The \textbf{bend relation} of $f$ 
 is a congruence relation on $\Gamma[\Lambda]$ which is generated by the following set $$\{f  \equiv_{\operatorname{bend}} f_{\hat
 h=  \a_{\bold i}\Lambda ^\bold i }:\
\bold i\in  \supp( f )\}.$$
\end{defn}

The point of this definition is that the tropical variety $V$ defined by a
tropical polynomial is defined by two monomials (not necessarily the
same throughout) taking equal dominant values at each point of $V$,
and then the bend relation reflects the equality of these values on
$V$, thence the relation.

\subsubsection{Tropicalization and tropical ideals}\label{trop1}$ $

In many cases, one relates tropical algebra to Puiseux series via
the following tropicalization map.

\begin{defn} For  any additive group $F$, one can define the additive group
$F\{\{t\}\} $ of Puiseux series (actually due to Newton) on the
variable $t$, which is the set of formal series of the form $ p =
\sum_{k = \ell}^{\infty} c_k t^{k/N}$ where $N \in \mathbb{N}$,
$\ell \in \mathbb{Z}$, and $c_k \in \mathcal K$.
\end{defn}

\begin{rem} If $F$ is an algebra, then $
F\{\{t\}\} $ is also an algebra under the usual convolution product.
\end{rem}

For $\Gamma = (\mathbb{Q},+)$, one has a canonical semigroup
homomorphism (the \textbf{Puiseux
    valuation}) as follows:
\begin{equation}
\textrm{val} : F\{\{t\}\} \setminus \{0\} \rightarrow \Gamma, \quad
p = \sum_{k =
    \ell}^{\infty} c_k t^{k/N} \mapsto \min_{c_k \neq 0}\{k/N\}.
\end{equation}

This induces the semigroup homomorphism
\begin{equation}
\tropa: F\{\{t\}\}[\Lambda] \to \Gamma[\Lambda], \quad p(\la_1
\cdots \la_n) : = \sum p_{\bold i}\la^{\bold i} \mapsto \sum
\val(p_{\bold i})\la^{\bold i}.
\end{equation}
The map $\tropa$ is called \textbf{tropicalization}.

\begin{defn}\cite{GG}
With the same notations as above, suppose $I$ is an ideal of
$F\{\{t\}\}[\Lambda].$ The bend congruence on $\{ \tropa(f): f \in I
\}$ is called the \textbf{tropicalization congruence} $\tropa(I).$
\end{defn}

 Suppose $F$
is a field. We can \textbf{normalize} a Puiseux series $p= \sum
p_{\bold i}\la_1^{i_1} \cdots \la_n^{i_n}$ at any given index $\bold
i \in \supp( p )$ by dividing through by $p_{\bold i}$; then the
normalized coefficient is $\one$. Given two Puiseux series $p= \sum
p_{\bold i}\la_1^{i_1} \cdots \la_n^{i_n}$, $q = \sum q_{\bold
i}\la_1 ^{i_1} \cdots \la_n^{i_n}$ having a common monomial
$\Lambda^{\bold i} = \la_1 ^{i_1} \cdots \la_n^{i_n}$ in their
support, one can normalize both and assume that $ p_{\bold i} =
q_{\bold i} = \one,$ and
 remove this
monomial from their difference $p-q$, i.e., the coefficient of
$\Lambda^{\bold i}$ in $\val(p-q)$ is $\zero \  ( = - \infty  ).$

Accordingly, a  \textbf{tropical ideal} of $\Gamma[\Lambda]$ is an
ideal $\mathcal I$ such that for any two polynomials $f= \sum
f_{\bold i}\la^{\bold i}$, $g= \sum g_{\bold i}\la^{\bold i} \in
\mathcal I$ having a common monomial $\Lambda^{\bold i}$   there is
$h =  \sum h_{\bold i}\la_1^{i_1} \cdots \la_n^{i_n}\in \mathcal I$
whose coefficient of $\Lambda^{\bold i}$ is $\zero,$ for which $$
h_{\bold i} \ge \min \{a_f f _{\bold i}, b_g g_{\bold i}\} $$ for
suitable $a_f,b_g \in F.$

For any tropical ideal $\mathcal I$, the sets of minimal indices of
supports constitutes the set of circuits of a matroid. This can be
formulated in terms of valuated matroids, defined in
\cite[Definition~1.1]{DW0} as follows:

  A
\textbf{valuated matroid} of \textbf{rank} $m$ is a pair $(E,v)$
where  $E$ is a set and $v: E^{m} \to \Gamma$   is a   map
satisfying the following properties:

\begin{enumerate}\eroman\item    There exist $e_1, \dots, e_m \in E$ with $v(e_1, \dots, e_m) \ne 0.$
 \item
 $v(e_1, \dots, e_m)= v(e_{\pi(1)}, \dots, e_{\pi(m)})$. for each $e_1, \dots, e_m \in E$ and every
permutation $\pi$. Furthermore,  $v(e_1, \dots, e_m)=\zero$ in case
some $e_i = e_j$.  \item For $(e_0, \dots, e_m,  e_2', \dots, e_m'
\in E$ there exists some $i$ with $1 \le i \le m$ and $$v(e_1,
\dots, e_m)v(e_0,  e_2', \dots, e_m')\le v(e_0, \dots,
e_{i-1},e_{i+1}, \dots, e_m)v(e_i,  e_2', \dots, e_m').$$
\end{enumerate}

This information is encapsulated in the following result.

\begin{theorem}
\cite[Theorem 1.1]{MacR} Let $K$ be a field with a valuation $\val : K \to \Gz$, and let
$Y$
be a  closed subvariety defined of  $({K^\times})^{n}$ defined by
an ideal $\mathcal I \triangleleft K[\la _1^{\pm 1}  , \dots,\la
_n^{\pm 1}  ].$ Then any of the following three objects determines
the others:
\begin{enumerate}\eroman
\item
The bend congruence $\Trop(\mathcal I)$ on the semiring
    $\mathcal S := \Gamma[\la _1^{\pm 1}  , \dots,\la _n^{\pm 1}  ]$ of
    tropical Laurent polynomials;
\item
The ideal $\tropa(\mathcal I)$
    in $  S$;
    \item
    The set of valuated matroids of the vector spaces
    $\mathcal I_h ^ d ,$ where $\mathcal I_h ^ d $ is the degree $d$
    part of  the homogenization of the tropical ideal $\mathcal I$.
\end{enumerate}
\end{theorem}

\subsubsection{The supertropical approach}$ $

In supertropical mathematics the definitions run somewhat more
smoothly. A tropical variety $V$ was defined in terms of $\tT^\circ$ so for $f,g \in
\tT [\Lambda]$ we define the \textbf{$\circ$-equivalence} $f
\equiv_\circ g$ on $\tT [\Lambda]$ if and only if $f^\circ = g^\circ
$, i.e. $f(a)^\circ = g(a)^\circ$ for each $a \in \tT .$

 \begin{prop}\label{epicspl}  The bend relation implies   $\circ$-equivalence,
 in the sense that if $f \ \equiv_{\operatorname{bend}}\
 g$ then $f \equiv_\circ g$, for any polynomials in $f,g \in \tT [\Lambda]$.
 \end{prop}  \begin{proof} The bend relation is obtained from a
 sequence of steps, each removing or adding on a monomial which
 takes on the same value of some polynomial $f$. Thus the
 defining relations of the bend relation are all $\circ$ relations.
 Conversely, given a $\circ$ relation $f^\circ = g^\circ $, where
 $f = \sum f_i$ and $g = \sum g_j$ for monomials $f_i, g_j$,
 we have $$f\ \equiv_{\operatorname{bend}}\  g_1 + f \ \equiv_{\operatorname{bend}}\  g_1 + g_2 + f
 \ \equiv_{\operatorname{bend}}\  \dots \ \equiv_{\operatorname{bend}}\  g  + f \ \equiv_{\operatorname{bend}}\  g
 +  f_{\hat
 h} \ \equiv_{\operatorname{bend}}\  \dots \ \equiv_{\operatorname{bend}}\  g, $$
 implying
 $f \ \equiv_{\operatorname{bend}}\
 g$.
  \end{proof}

In supertropical algebra, given a   polynomial $f = \sum \a_{\bold
i}\Lambda ^\bold i ,$  define  $\vsupp( f )$ to be all the tuples
$\bold i =({i_1}\cdots  {i_n})$  for which the monomial in $f$ has
 coefficient $\a_{\bold i}\in \tT$.

 The supertropical version of tropical
ideal is that if $f,g \in \mathcal I$ and $\bold i \in \vsupp( f )
\cap \vsupp( g ),$ then, by normalizing, there are $a_f, b_g$ such
that  $\bold i \notin \vsupp(a_f f + b_g g)$. This is somewhat
stronger than the claim of the previous paragraph, since it
specifies the desired element.

Supertropical ``$d$-bases'' over a super-semifield are treated in
\cite{IKR}, where vectors are defined to be independent iff no
tangible linear combination is a ghost. If a tropical variety $V$ is defined as the set
 $\{v \in F^{(n)}: f_j(v) \in \nu(F),\ \forall j \in J\}$ for a set $\{ f_j : j \in J\}$ of homogeneous polynomials of  degree
 $m$, then taking $\mathcal I = \{ f: f(V) \in \nu(F)\}$ and  $\mathcal
 I_m$ to be its polynomials of degree $m$, one sees that the
$d$-bases of $I_m$ of cardinality $m$  comprise a matroid (whose
circuits are those polynomials of minimal support), by
\cite[Lemma~4.10]{IKR}. On the other hand, submodules of free
modules can fail to satisfy Steinitz' exchange property
(\cite[Examples~4.18,4.9]{IKR}), so there is room for considerable
further investigation.

``Supertropicalization'' then is the same tropicalization map as
$\tropa$, now taken to the standard supertropical semifield $\strop:
\Gamma \cup \tGz$ (where $\tT = \Gamma$). In view of
Proposition~7.5, the analogous proof of \cite[Theorem 1.1]{MacR}
yields the corresponding result:

\begin{thm} Let $K$ be a field with a valuation $\val : K \to \Gz$, and
let
 $Y$
be a  closed subvariety defined of  $({K^\times})^{n}$ defined by an
ideal $ I \triangleleft K[\la _1^{\pm 1}  , \dots,\la _n^{\pm 1} ].$
Then any of the following  objects determines the others:
\begin{enumerate}\eroman\item The congruence given by $\circ$-equivalence on $\mathcal
I = \strop(I)$ in the supertropical \semiring0 $\mathcal S :=
\Gamma[\la _1^{\pm 1} , \dots,\la _n^{\pm 1}  ]$ of tropical Laurent
polynomials; \item The ideal $\tropa(\mathcal I)$ in $\mathcal  S$;
\item  The set of valuated matroids of the vector spaces $\mathcal
I_h ^ d ,$ where $\mathcal I_h ^ d $ is the degree $d$ part of  the
homogenization of the tropical ideal $\mathcal I$.
 \end{enumerate}
 \end{thm}

\subsubsection{The systemic approach}$ $

The supertropical approach can be generalized directly to the
systemic approach, which also includes hyperfields and fuzzy rings.
We assume $(\mathcal A, \tT, (-), \preceq)$ is a system.

\begin{defn}\label{sys1}
The \textbf{$\circ$-equivalence} on $\Fun (S, \mathcal A)$ is
defined by,  $f \equiv_\circ g$ if and only if $f^\circ = g^\circ $,
i.e. $f(s)(b)^\circ = g(s)(b)^\circ$ for each $s \in S$ and $b \in
\mathcal A.$
\end{defn}

This matches the bend congruence.

\begin{defn}\label{sys2}
Given  $f \in \Fun (S, \mathcal A)$ define $\circ$-$\supp( f )= \{ s
\in S: f(s) \in \tT\}$.
\end{defn}

 The systemic version of tropical
ideal is that if $f,g \in \mathcal I$ and $s \in \circ$-$\supp( f
)\, \cap\, \circ$-$\supp( g),$ then there are $a_f, b_g \in \tT$
such that
  $s \notin \circ$-$\supp (a_f  f (-) b_g g)$.

Now one can view tropicalization as a functor as in \cite[\S
10]{Row16}, and define the appropriate valuated matroid. Then one
can address the recent work on matroids and valuated matroids, and
formulate them over systems. Presumably, as in \cite{AGR},  in the
presence of various assumptions, one might be able to carry out the
proofs of these assertions, but we have not yet had the opportunity
to carry out this program.

\section{Appendix B: The categorical approach}\label{catconvol}$ $

In this appendix we present the ideas more categorically, with an
emphasis on universal algebra.

\subsection{Categorical aspects of systemic theory}\label{abstcat}$ $

One can apply categorical concepts to better understand triples and
$\tT$-systems and their categories.

\subsubsection{Categories with a negation functor and categorical
triples}$ $

We introduce another categorical notion, to compensate  for lack of
negatives.

\begin{defn} Let $\mathcal{C}$ be a category. A \textbf{negation functor} is an endofunctor   $(-): \mathcal C \to \mathcal
C $ satisfying $(-)\mathbf A = \mathbf A$ for each object $\mathbf
A,$ and, for all morphisms $f,g$:
\begin{enumerate}\label{negpro}\eroman
   \item  $(-)((-)f)= f$.
    \item $(-)(fg) = ((-)f)g = f((-)g)$, i.e., any composite of morphisms
    (if defined) commutes with $(-)$.
   \end{enumerate}
\end{defn}

We write $f(-)g$ for $f + ((-)g),$ and $f^\circ$ for $f(-)f.$

\begin{rem} In each situation the negation map gives rise to a
negation functor in the category arising from universal algebra,
defining   $(-)f$ for any morphism $f$ to be given by $((-)f)(a)
=(-)(f(a))$. The identity functor obviously is a negation functor,
since these conditions become tautological. But
 categories in general may fail to
have a natural non-identity negation functor. For example, a
nontrivial negation functor for the usual category {\bf  Ring} might
be expected to contain ``negated homomorphisms'' $-f$ where $(-f)(a)
:= f(-a).$ Note then that $(-f)(a_1 a_2) = -(-f)(a_1)(-f)(a_2),$ so
$-f$ is not a homomorphism unless $f = -f.$ In such a situation  we
must expand the set of morphisms to contain ``negated
homomorphisms.'' Also $(-f)\cdot (-g) = f\cdot g,$ where $\cdot$
denotes pointwise multiplication.
\end{rem}

%
%
%
%

\subsubsection{$\preceq$-morphisms in the context of universal algebra}\label{surpre202}$ $

One of our main ideas   is to bring the surpassing relation
$\preceq$ into the picture (although it is not an identity), and in
the universal algebra context we require $\preceq$ to be
 defined on each $\mathcal A_j $.

 Let
$\AA = (\mathcal A _1, \dots \mathcal A_m)$ and $\BA= (\mathcal A'
_1, \dots \mathcal A'_m)$  be  carriers (of $\mathcal S$) with
surpassing POs $\preceq$ and $\preceq'$, respectively. A
$\preceq$-\textbf{morphism} $f: \AA \to \BA$ is a set of maps $f_j:
\mathcal A_j\to \mathcal A_j'$, $1 \le j \le m,$ satisfying the
properties for every operator $\omega$  and all $b_i,  c_i  \in
\mathcal A _{j_i}, $:
\begin{enumerate}\eroman \item $f(\omega (b_{1 }, \dots, b_{m }))\preceq ' \omega
(f_{j_1}(b_{1 }), \dots, f_{j_m}(b_{m })),$

\item  If $b_{i} \preceq c_{i}$, then
$$\omega (f_{j_1}(b_{1}), \dots,f_{j_m}(b_{m}))\preceq '
\omega (f_{j_1}(c_{1}), \dots, f_{j_m}(c_{m})).$$
\item  $f_j(\zero) = \zero$ whenever $\zero \in \mathcal A_j.$ \end{enumerate}

 We call
$f$ an $\omega$-\textbf{homomorphism} when equality holds in (i).
 A \textbf{homomorphism} is an $\omega$- homomorphism for each
operator $\omega$. Any  homomorphism is a $\preceq$-morphism when we
take $\preceq$ to be the identity map. This is a delicate issue,
since although $\preceq$-morphisms play an important structural
role, as indicated in \cite{JMR}, homomorphisms fit in better with
general monoidal category theory, as we shall see.
    Thus, our default terminology is according to the standard
    universal algebra version.

\begin{defn}\label{surprel1}
An \textbf{abstract surpassing relation} on a system $\mathcal A=
(\mathcal A, \tT, (-), \preceq),$ is a relation $\preceq$ satisfying
the following properties: for $a,b\in \mathcal A$,  viewed as a
carrier in universal algebra:
\begin{itemize}
 \eroman
 \item $\zero  \preceq a$.
  \item If $a_i \preceq b_i$ for $1 \le i \le m$ then  $ \omega_{m,j}(a_{1,j}, \dots,
a_{m,j}) \preceq \omega_{m,j}(b_{1,j}, \dots, b_{m,j}) .$
  \end{itemize}
An \textbf{abstract surpassing PO} is an abstract surpassing
  relation that is a PO (partial order).
 \end{defn}

In particular, if $a \preceq \zero$ for an abstract surpassing PO,
then $a = \zero.$ Surpassing relations can play a role in  modifying
universal algebra. This can be formulated categorically, but we
state it for systems.

\begin{defn}\label{surpa} The \textbf{surpassing relation on morphisms}
$\preceq$ for categories of   systems is defined by putting
$f\preceq g$ if $f(b) \preceq g(b)$ for each $b \in \mathcal A$.
\end{defn}

\begin{example} $f\preceq_\circ g$ when $g = f+
h^\circ$ for some morphism $h$, but there could be other instances
where we cannot obtain $h$ from the values of $f$ and
$g$.\end{example}

\subsection{Functor categories}\label{functcat}$ $


\begin{defn}\label{Func19} Let $S$ be a small category, i.e.,
 $\textrm{Obj}(S)$ and $\Hom(S)$ are sets. Let $\mathcal C$ be a category. We  define
$\mathcal C^S$ to be the category  whose objects are the sets of
functors  from $S$ to $\mathcal C$ and morphisms are natural
transformations.
\end{defn}



Let $S$ be a small category.  For any ground  $\tT$-system $\mathcal
A= (\mathcal A, \tT, (-), \preceq),$ considered as a carrier,
 we
   view  the carrier in the  functor
category $\mathcal A^S$ ($\mathcal A$ viewed as a small category) as
the carrier
  $\{\mathcal A_1^S, \mathcal A_2^S, \dots, \mathcal
A_t^S \}$, where $\mathcal A_i^S$ denotes the morphisms from $S$ to
$\mathcal A_i$ and, given an operator $\omega_{m,j}: \mathcal
A_{i_{j,1}} \times \dots \times \mathcal A_{i_{j,m}} \to \mathcal
A_{i_{j}}$, we define the operator
 $\tilde \omega : \mathcal
A_{i_{j,1}} ^S\times \dots \times  \mathcal A_{i_{j,m}}^S \to
\mathcal A_{i_{m,j}}^S$ ``componentwise,''  by
$$\tilde \omega  (f_1, \dots, f_t)(s) = \omega (f_1(s), \dots
f_t(s)), \quad \forall s\in S .$$

Likewise for $\mathcal A^{(S)}$. Universal relations clearly pass
from $\mathcal A$ to $\mathcal A^S$ and $\mathcal A^{(S)}$, verified
componentwise. For convenience, we assume that the signature
includes the operation $+$ together with the distinguished zero
element $\zero$.

%


\subsubsection{N-categories}$ $

Since
 we lack negatives in semirings and their modules, $\Hom(A,B)$  is not a group  under addition, but rather a
semigroup, and the zero morphism loses its special role, to be
supplanted by a more general notion. Grandis in
\cite{grandis2013homological} generalizes the usual categorical
definitions to ``$N$-categories.''

\begin{defn}
Let $\mathcal{C}$ be a category.
\begin{enumerate} \eroman
\item
A \textbf{left absorbing set of morphisms} of $\mathcal C$ is a
collection of sets of morphisms $I$ such that if $f$ belongs to $I$,
then any composite $gf$ (if defined) belongs to $I$.
\item
A   \textbf{right absorbing set of morphisms} of  $\mathcal C$ is a
collection of sets of morphisms $I$ such that if $f$ belongs to $I$,
then any composite $fg$ (if defined) belongs to $I$.
\item
An \textbf{absorbing set of morphisms} is a left and right absorbing
set of morphisms. (\cite[\S~1.3.1]{grandis2013homological} calls
this an ``ideal'' but we prefer to reserve this terminology for
semirings.)
\end{enumerate}
\end{defn}

For any given absorbing set $N$ of morphisms  of a category
$\mathcal{C}$, one can associate the set $O(N)$ of \textbf{null
objects}
 as follows:
\[
O(N):=\{A \in \textrm{Obj}(\mathcal{C})\mid 1_A \in N\}.
\]
Conversely, we  can fix a class $O$ of null objects in
$\mathcal{C}$, and can associate the absorbing   morphisms $N(O)$ as
follows:
\[
N(O):=\{f \in \Hom(\mathcal{C}) \mid \textrm{ $f$ factors through
some object in $O$}\}.
\]
Then one clearly has the following:
\[
O \subseteq ON(O), \quad NO(N) \subseteq N.
\]
For details, we refer the reader to
\cite[\S1.3]{grandis2013homological}.

 \begin{defn}
 Let $\mathcal{C}$ be a category.
An absorbing set $N$ of $\mathcal{C}$ is \textbf{closed} if $N=N(O)$
for some set of objects $O$.
\end{defn}

The null  morphisms $\Null$ (depending on a fixed class of null
objects) are an example of an absorbing set of morphisms; the null
morphisms in $\Hom (A,B)$ are designated as $\Null_{A,B},$ which
will take the role of $ \{ 0 \}.$ Such a category with a designated
class of null objects and of null morphisms is called a
\textbf{closed} N-\textbf{category}; we delete the word ``closed''
for brevity.

\begin{defn}\cite[\S~1.3.1]{grandis2013homological}\label{cake}
Let $\mathcal{C}$ be a closed N-category with a fixed class of null
objects $O$ and the corresponding null morphisms $N$.
\begin{enumerate}\eroman
\item
The \textbf{kernel with respect to $N$} of a morphism $f:A\to B$,
denoted by $\ker f$, is a monic which satisfies the universal
property of a categorical kernel with respect  to  $N$, i.e.,
\begin{enumerate}
\item
$f (\ker f)$ is null.
\item
If $fg$ is null then $g$ uniquely factors through $\ker f$ in the
sense that $g = (\ker f) h$ for some $h$.
\end{enumerate}
\item
The \textbf{cokernel with respect to $N$} of a morphism $f:A\to B$,
denoted by $\coker f$, is defined dually, i.e., $\coker f$ is an
epic satisfying
\begin{enumerate}
\item
$  (\coker f) f$ is null.
\item
If $gf$ is null then $g$ uniquely factors through $\coker f$ in the
sense that $g = h (\coker f) $ for some $h$.
\end{enumerate}
\item
Products and coproducts (direct sums) of morphisms also are  defined
in the usual way, cf.~\cite[Definition~1.2.1(A2)]{EGNO}.
\end{enumerate}
\end{defn}

%
%
%
%

%

\subsubsection{$\tT$-linear categories with a negation functor}\label{semiadd}$ $

Since most of this paper involves $\tT$-modules over semirings, let
us pause
to see in what direction we want to proceed.  

\begin{defn}
Let $\mathcal C$ be a category.
\begin{enumerate} \eroman
\item
$\mathcal C$ is $\tT$-\textbf{linear} over a monoid $\tT$ if it
satisfies the following two properties:
\begin{enumerate}
\item
Composition is bi-additive, i.e.,
$$(g+h)f = gf+ hf, \qquad k(g+h) = kg + kh$$
for all morphisms $f: \AA\to \BA,$ $g,h:\BA \to \CA$, and $k: \CA\to
\DA.$ In particular, $\Mor (\AA,\AA)$ is a \semiring0, where
multiplication is given by composition.
\item
$\tT$ acts naturally on $\Mor(A,B)$ in the sense that the action
commutes with morphisms. (In practice, the objects will be
$\tT$-modules, and the action will be by left multiplication.)
\end{enumerate}

\item
A  $\tT$-linear N-category $\mathcal C$ with categorical sums is
called \textbf{semi-additive}.
\end{enumerate}

\end{defn}

 \subsubsection{Systemic categories}$ $

\begin{defn} $ $
\begin{enumerate} \eroman
\item
A $\tT$-\textbf{linear  category with negation} is a $\tT$-linear
category with a negation functor~$(-)$.
\item
A \textbf{semi-additive} category \textbf{with negation} is a
semi-additive  category  with a negation functor.
\end{enumerate}
\end{defn}

\begin{example}\label{cattrp}$ $ Any additive category (in the classical sense) is  $\mathbb Z$-linear and
 semi-additive with negation.
\end{example}

\begin{prop} For any $\tT$-linear  category with negation, the ``quasi-zero'' morphisms (of the form
$f^\circ:=f+((-)f)$) comprise an absorbing set.\end{prop}
\begin{proof}
One can easily observe the following: $(f(-)f)g = fg (-) fg$ and
$g(f(-)f)  = gf (-)  gf$ and this shows the desired result.
\end{proof}

\begin{example}
Let $\mathcal{C}$ be a category with a negation functor. We can
impose an N-category structure on $\mathcal{C}$ by defining the
absorbing set $N$ of null morphisms to be the morphisms of the form
$f^\circ$; this may or may not be a closed $N$-category. On the
other hand, later in \S\ref{surpre3}, we will introduce a closed
N-category structure by means of elements of the form $a^\circ.$
\end{example}

Expressing unique negation categorically enables us both to relate
to the work \cite[\S2]{Ju} of the first author on hyperrings, and
also appeal to the categorical literature.

\begin{defn}
A \textbf{semi-abelian category  (resp.~with negation)} is a
semi-additive category (resp.~with negation) $\mathcal C$ satisfying
the following extra property:

\medskip

Let $N$ be the set of null morphisms of $\mathcal C$.
 Every morphism of $\mathcal C$ can be written as the
composite of a cokernel with respect to $N$ and a kernel with
respect $N$.

\medskip

(This property is called ``semiexact'' in
 \cite[\S~1.3.3]{grandis2013homological}.)
\end{defn}

One can define the monoidal property in terms of the adjoint
isomorphism, cf.~ \cite{Mac2}. In short, our discussion also fits
into the well-known theory  of ``rigid'' monoidal categories as
described in \cite{EGNO}, and makes the  system category amenable to
\cite{ToV}.

\begin{prop}\label{kercok1}
The category of module systems $\mathcal A $-module triples  is a
monoidal semi-abelian category,  with respect to negated
$\tT$-tensor products.
\end{prop}
\begin{proof} Immediate from Lemma~\ref{kercok} and Remark~\ref{adjmap}.
\end{proof}

\begin{rem} The following properties pass from the category $\mathcal C$ to $\mathcal C ^S,$
seen componentwise:

 $\tT$-linear, semi-additive,  semi-additive with negation,
 semi-abelian.\end{rem}

There is a very well-developed theory of abelian categories, which
one would like to utilize by generalizing to  semi-abelian
categories. This has already been done for a large part in
\cite{grandis2013homological}, so one main task should be to arrange
for the category of systems to fit into Grandis' hypotheses. 

\end{document}